\documentclass[11pt]{article}

\usepackage[margin=1in]{geometry} 
\usepackage{graphicx}

\usepackage{booktabs} 
\usepackage{array} 
\usepackage{paralist} 
\usepackage{verbatim}
\usepackage{mathrsfs}
\usepackage{amssymb}
\usepackage{amsthm}
\usepackage{amsmath,amsfonts,amssymb}
\usepackage{esint}
\usepackage{graphics}
\usepackage{enumerate}
\usepackage{mathtools}
\usepackage{xfrac}
\usepackage{bbm}
\usepackage{caption}
\usepackage{subcaption}
\usepackage{tikz-cd}
\usepackage{adjustbox}

 \usepackage[T1]{fontenc} 
\usepackage[utf8]{inputenc}  
\usepackage[english]{babel}

\usepackage[colorlinks=true, pdfstartview=FitV, linkcolor=blue, citecolor=blue, urlcolor=blue]{hyperref}

\numberwithin{equation}{section}
\numberwithin{figure}{section}

\newtheorem{theorem}{Theorem}[section]

\newtheorem{corollary}[theorem]{Corollary}
\newtheorem{proposition}[theorem]{Proposition}
\newtheorem{lemma}[theorem]{Lemma}
\newtheorem{OQ}{Open Question}

\theoremstyle{definition}
\newtheorem{definition}[theorem]{Definition}

\newtheorem{remark}{Remark}

\newcommand{\md}[1]{\ (\text{mod}\ #1)}

\newcommand*{\N}{\ensuremath{\mathbb{N}}}

\newcommand*{\Z}{\ensuremath{\mathbb{Z}}}

\newcommand*{\R}{\ensuremath{\mathbb{R}}}
\newcommand*{\Zd}{\ensuremath{\mathbb{Z}^d}}
\newcommand*{\Rd}{\ensuremath{\mathbb{R}^d}}

\newcommand{\eps}{\varepsilon}

\renewcommand*{\tilde}{\widetilde}
\renewcommand*{\hat}{\widehat}

\newcommand{\ep}{\eps}

\newcommand{\E}{\mathbb{E}}

\DeclareSymbolFont{boldoperators}{OT1}{cmr}{bx}{n}
\SetSymbolFont{boldoperators}{bold}{OT1}{cmr}{bx}{n}
\edef\bar{\unexpanded{\protect\mathaccentV{bar}}\number\symboldoperators16}

\definecolor{darkgreen}{rgb}{0,0.6,0.05}

\newcommand{\margin}[1]{\textcolor{magenta}{*}\marginpar[\textcolor{magenta} {  \raggedleft  \footnotesize  #1 }  ]{ \textcolor{magenta} { \raggedright  \footnotesize  #1 }  }}

\definecolor{labelkey}{rgb}{0,0,1}

\newcommand{\indc}{{\boldsymbol{1}}}

\newcommand{\var}{\mathrm{Var}}
\newcommand{\W}{W}

\renewcommand{\P}{\mathbb{P}}

\def\Bb#1#2{{\def\md{\bigm| }#1\bigl[#2\bigr]}}

\def\Pb{\Bb\P}

\def\<#1{\langle #1\rangle}

\def\bi{\begin{itemize}}  %USE\bi[WHATEVER]
\def\ei{\end{itemize}}
\def\bnum{\begin{enumerate}} % USE \bnum[i)] if want i), ii) .. OR \bnum[{\bf (a)}] etc ..!
\def\enum{\end{enumerate}}
\def\ni{\noindent}
\def\bf{\bfseries}

\usepackage{titlesec}

\newcommand{\addperiod}[1]{#1.}
\titleformat{\section}
   {\centering\normalfont\Large}{\thesection.}{0.5em}{}
\titleformat*{\subsection}{\bfseries}
\titleformat{\subsubsection}[runin]
  {\normalfont\bfseries}
  {\thesubsubsection.}
  {0.5em}
  {\addperiod}
\titleformat*{\subsubsection}{\bfseries}
\titleformat*{\paragraph}{\bfseries}
\titleformat*{\subparagraph}{\large\bfseries}

\title{The impact of disorder and non-convex interactions on delocalisation of height functions}
 
\author{Paul Dario
\thanks{Laboratoire AGM, CNRS UMR 8088, Cergy Paris Universit\'e, Cergy-Pontoise, France.
{\footnotesize paul.dario@cyu.fr.}
}
\and
Diederik van Engelenburg
\thanks{Technical University of Vienna, Austria
{\footnotesize diederik.engelenburg@tuwien.ac.at.}
}
\and 
Christophe Garban
\thanks{
Université Claude Bernard Lyon 1, CNRS UMR 5208, Institut Camille Jordan, 69622 Villeurbanne, France, and Courant Institute (NYU), New York, 
{\footnotesize christophe.garban@gmail.com.}
}
}
\date{ }

\usepackage[nottoc,notlot,notlof]{tocbibind}

\begin{document}

\maketitle

\begin{abstract}
We study the behaviour of four spins systems (the XY model, the Villain model, 
the XY height function and the integer-valued Gaussian free field) in the presence of a non-elliptic quenched disorder. 
In the article~\cite{DG25}, it was shown that the phase transitions of the XY model 
(the Berezinskii-Kosterlitz-Thouless phase transition in $d = 2$ and the order/disorder phase transition 
when $d \geq 3$) persist on the infinite cluster of a supercritical Bernoulli 
percolation. A first objective of this article is to extend these results to the Villain model.

Our second objective is to analyse, for $d=2$, 
how the corresponding dual integer-valued height function models behave in the presence of a dual quenched disorder. 
These dual models are respectively 
the XY height function and the integer-valued Gaussian free field. 
Without disorder, these models are known to  exhibit a phase transition in two dimensions called the 
roughening transition~\cite{frohlich1981kosterlitz, lammers2022height}. 
We show that this phase transition persists when the quenched disorder is given by enforcing $\varphi(x) = \varphi(y)$
independently with probability $\bar{p} < 1/2$ for neighboring sites $x, y$. 

Finally, we apply our methods to integer-valued height functions with 
annealed Gaussian interactions and prove the existence of a (quantified) rough phase.
This includes all potentials of the form $|\nabla h|^p$ for $p \in (0, 2]$, recovering recent results of 
\cite{ott2025quantitative}.

 \end{abstract}

\tableofcontents

\section{Introduction}

In this article, we are interested in the behaviour of four models of statistical physics in the presence of a quenched disorder on the underlying graph. We first present the XY and Villain models in Section~\ref{sec:presentationXYVil} as well as the results obtained on these spin systems. Section~\ref{sec:presentationheight} is devoted to the presentation of the integer-valued Gaussian free field and the XY height function and contains the statement of the results obtained on these models.

\subsection{The XY and Villain models} \label{sec:presentationXYVil}

\subsubsection{The models}

We consider the XY (or rotator) and the Villain models on the hypercubic lattice $\mathbb{Z}^d$ with $d \geq 2$. The models are formally defined as follows: let us fix a finite set $\Lambda \subseteq \Z^d$, $\beta \geq 0$, and define the XY model on the set $\Lambda$ with inverse temperature $\beta$ to be the probability measure $\mu_{\Lambda, \beta}$ on the space $[0 , 2\pi)^{\Lambda} := \left\{ \theta : \Lambda \to [0 , 2 \pi) \right\}$
\begin{equation} \label{eq:XYmodel}
    \mu_{\Lambda , \beta}^{\mathrm{XY}}(d \theta) := \frac{1}{Z_{\Lambda, \beta}^{\mathrm{XY}}} \exp \left( \beta \sum_{\substack{x , y \in \Lambda \\ x \sim y}} \cos(\theta_x - \theta_y )\right) \prod_{x \in \Lambda} d \theta_x,
\end{equation}
where $d \theta_x$ denotes the Lebesgue measure on the interval $[0 , 2\pi)$, $Z_{\Lambda, \beta}^{\mathrm{XY}}$ is the normalizing constant and the notation $x \sim y$ means that $x$ and $y$ are nearest neighbour in $\Zd$. We denote by $\left\langle \cdot \right\rangle_{\Lambda, \beta}^{\mathrm{XY}}$ the expectation with respect to the measure $\mu_{\Lambda , \beta}^{\mathrm{XY}}$.

The Villain model is a variant of the XY model defined, on the set $\Lambda \subseteq \Z^d$ (finite) and with inverse temperature $\beta \geq 0$, to be the probability distribution on the set of angle configurations $[0 , 2\pi)^{\Lambda}$
\begin{equation*}
     \mu_{\Lambda , \beta}^{\mathrm{Vil}}(d \theta) := \frac{1}{Z_{\Lambda, \beta}^{\mathrm{Vil}}} \prod_{\substack{x , y \in \Lambda \\ x \sim y}} v_\beta \left( \theta_x - \theta_y \right) \prod_{x \in \Lambda} d \theta_x,
\end{equation*}
where $v_\beta: [0 , 2 \pi) \to (0 , \infty)$ is the heat kernel on the circle defined according to the identity (N.B. we set $v_0 = 1$)
\begin{equation} \label{def:introdefVillain}
      v_\beta (\theta)  := \sqrt{2 \pi \beta} \sum_{k \in \Z} \exp \left( - \frac{\beta}{2} (\theta - 2\pi k)^2 \right) = \sum_{k \in \Z} e^{- \frac{k^2}{2 \beta}} e^{ik \theta}.
\end{equation}
We denote by $\left\langle \cdot \right\rangle_{\Lambda, \beta}^{\mathrm{Vil}}$ the expectation with respect to the measure $\mu_{\Lambda , \beta}^{\mathrm{Vil}}$.
By setting $(S_x)_{x \in \Lambda} := (e^{i \theta_x})_{x \in \Lambda}$, these model can be equivalently considered as spin systems taking values in the circle $\mathbb{S}^1$.

A compactness argument and the Ginibre correlation inequality~\cite{ginibre1970general} (together with the observation that the Villain model can be obtained as a metric graph limit of the XY model, see Theorem~\ref{theoremGinibreineq} and Proposition~\ref{prop:prop3.6}) imply that the two measures $\mu_{\Lambda , \beta}^{\mathrm{XY}}$ and $\mu_{\Lambda , \beta}^{\mathrm{Vil}}$ converge, as $\Lambda \uparrow \Z^d$, toward infinite-volume and translation-invariant measures denoted by $\mu_{\beta}^{\mathrm{XY}}$ and $\mu_{\beta}^{\mathrm{Vil}}$. We will use the notation $\left\langle \cdot \right\rangle_{\beta}^{\mathrm{XY}}$ and $\left\langle \cdot \right\rangle_{\beta}^{\mathrm{Vil}}$ for the corresponding expectations.

These models have been subject to extensive research and they are known to exhibit phase transitions (of different natures) in dimension 2 and in dimensions 3 and higher. These phase transitions are characterised by a different asymptotic behaviour of the two-point functions
\begin{equation} \label{eq:deftwopointfunction}
    x \mapsto \left\langle \cos(\theta_0 - \theta_x) \right\rangle_{\beta}^{\mathrm{XY}} ~~\mbox{and}~~x \mapsto \left\langle \cos(\theta_0 - \theta_x) \right\rangle_{\beta}^{\mathrm{Vil}}.
\end{equation}
In dimension $d \geq 3$, the XY and Villain models undergo an order/disorder phase transition characterised by long-range order at low temperature and exponential decay of the two-point function at high temperature, i.e.,
\begin{equation} \label{LRO}  \tag{LRO}
    \left\langle \cos(\theta_0 - \theta_x) \right\rangle_{\beta}^{\mathrm{XY}} \geq c > 0 ~~\mbox{for}~~ \beta \gg 1 ~~\mbox{and}~~ \left\langle \cos(\theta_0 - \theta_x) \right\rangle_{\beta}^{\mathrm{XY}} \leq \exp(-c |x |) ~~\mbox{for}~~\beta \ll 1
\end{equation}
(N.B. The same result holds for the Villain model). The existence of this phase transition was originally established by Fr\"{o}hlich, Simon and Spencer~\cite{frohlich1976infrared}, and we also mention alternative approaches by Fr\"{o}hlich and Spencer~\cite{frohlich1982massless}, Kennedy and King~\cite{kennedy1986spontaneous}, and the third author and Spencer~\cite{garban2022continuous}.

In two dimensions, Mermin and Wagner~\cite{mermin1966absence} proved that long-range order does not happen at any inverse temperature, i.e., $\left\langle \cos(\theta_0 - \theta_x) \right\rangle_{\beta}^{\mathrm{XY}} \underset{x \to \infty}{\longrightarrow} 0$ for any $\beta > 0$. This result was then refined by McBryan and Spencer~\cite{mcbryan1977decay} who showed that the two-point function decays at least polynomially fast at any positive temperature (see also \cite{shlosman1978decrease}). In this dimension, the phase transition is thus of different nature and is characterised by a different decay rate of the two-point function: it decays polynomially fast at low temperature and exponentially fast at high temperature, i.e.,
\begin{equation} \label{BKT} \tag{BKT}
    \frac{c}{|x|} \leq \left\langle \cos(\theta_0 - \theta_x) \right\rangle_{\beta}^{\mathrm{XY}} ~~\mbox{for}~~ \beta \gg 1 ~~\mbox{and}~~ \left\langle \cos(\theta_0 - \theta_x) \right\rangle_{\beta}^{\mathrm{XY}} \leq \exp(- c |x|)  ~~\mbox{for}~~\beta \ll 1
\end{equation}
(N.B. Once again, the same result holds for the Villain model). The existence of this phase transition, known as the Berezinskii-Kosterlitz-Thouless (BKT) phase transition, was originally established in the breakthrough work by Fr\"{o}hlich and Spencer~\cite{frohlich1981kosterlitz}, and has been the subject of recent new developments by Lammers~\cite{lammers2022height, lammers2022dichotomy, lammers2023bijecting}, van Engelenburg and Lis~\cite{van2023duality, van2023elementary}, and Aizenman, Harel, Peled and Shapiro~\cite{AHPS} (see also the survey \cite{kharash2017fr}).

\subsubsection{The result}
In the article~\cite{DG25}, we were interested in the existence of these phase transitions for the XY model in the presence of a random quenched disorder in the underlying graph. Specifically, we studied the XY model placed on a random graph which is either given by:
\begin{enumerate}
\item The infinite cluster of a supercritical i.i.d Bernoulli percolation (either site or edge) on the
lattice $\Z^d$,
\item The Poisson-Voronoi graph induced by a Poisson Point Process of intensity $1$ on $\R^d$,
\end{enumerate}
and showed that the BKT and the order/disorder phase transitions of the XY model are still observed in this setting. 
As pointed out in \cite{DG25}, the proof for XY did not extend to the Villain model as the implications of Wells'  inequality in this case vanishes. 
The first main result of this article is to extend the result of~\cite{DG25} to the case of the Villain model (for simplicity, we will only treat the case of the i.i.d. edge Bernoulli percolation).

Before stating the results, we introduce some definitions and notation. In two dimensions, we denote by $\beta_{\mathrm{KT}}^{\mathrm{Vil}} > 0$ the critical inverse temperature for the BKT phase transition of the Villain model, i.e.,
\begin{equation*}
    \beta_{\mathrm{KT}}^{\mathrm{Vil}} := \sup \left\{ \beta \geq 0 \, : \, \left\langle \cos(\theta_0 - \theta_x) \right\rangle_{\beta}^{\mathrm{Vil}} ~\mbox{decays exponentially fast}\right\}.
\end{equation*}
Let us remark that the results of~\cite{frohlich1981kosterlitz, AHPS, van2023elementary} imply that $\beta_{\mathrm{KT}}^{\mathrm{Vil}} \in ( 0, \infty)$ and an adaptation of the results of~\cite[Section 6]{van2023elementary} imply that the two-point function decays polynomially fast for any $\beta \geq \beta_{\mathrm{KT}}^{\mathrm{Vil}}$.

In dimension $d \geq 3$, we denote by $\beta^{\mathrm{Vil}}_c(d)$ the critical inverse temperature
\begin{equation*}
    \beta^{\mathrm{Vil}}_{c}(d) := \sup \left\{ \beta \geq 0 \, : \, \left\langle \cos(\theta_0 - \theta_x) \right\rangle_{\beta}^{\mathrm{Vil}} \underset{|x| \to \infty}{\longrightarrow}  0\right\}.
\end{equation*}
The results of~\cite{frohlich1976infrared, frohlich1982massless, kennedy1986spontaneous, garban2022continuous} imply that $\beta^{\mathrm{Vil}}_{c}(d) \in ( 0, \infty)$ and the Messager-Miracle-Sole inequality~\cite{messager1977correlation} implies that for any $\beta > \beta^{\mathrm{Vil}}_c(d)$, the two-point function remains bounded away from $0$.

We next introduce the disordered version of the Villain model studied in this article. 

\begin{definition}[Villain model on a percolation configuration]
Given an edge percolation configuration $\omega \in \{0 , 1\}^{E(\Zd)}$, a finite set $\Lambda \subseteq \Zd$ and an inverse temperature $\beta \geq 0$, we consider the Villain model on the percolation configuration $\omega$ to be the probability distribution on the set of configurations $[0 , 2\pi)^{\Lambda}$
\begin{equation*}
    \mu_{\Lambda , \beta , \omega}^{\mathrm{Vil}}(d \theta) := \frac{1}{Z_{\Lambda, \beta, \omega}^{\mathrm{Vil}}} \prod_{\substack{x , y \in \Lambda \\ x \sim y}} v_{\omega_{xy}\beta} \left( \theta_x - \theta_y \right) \prod_{x \in \Lambda} d \theta_x.
\end{equation*}
We denote by $\mu_{\beta,\omega}^{\mathrm{Vil}}$ the thermodynamic limit obtained by taking the limit $\Lambda \uparrow \Zd$ (whose existence
and uniqueness is once again guaranteed by compactness arguments and the Ginibre inequality), and denote by $\left\langle \cdot \right\rangle_{\beta, \omega}$ the corresponding expectation.
\end{definition}

\begin{remark}
The incorporation of this disorder reduces the value of the two-point function of the model and the Ginibre inequality implies that, for any $\omega \in \{0 , 1\}^{E(\Zd)}$ and any $x \in \Zd$,
\begin{equation} \label{eq:Ginibreanyperco}
    \left\langle \cos \left( \theta_0 - \theta_x \right) \right\rangle^{\mathrm{Vil}}_{\beta , \omega} \leq \left\langle \cos \left( \theta_0 - \theta_x \right) \right\rangle^{\mathrm{Vil}}_{\beta}.
\end{equation}
\end{remark}

We further assume that the percolation configuration $\omega \in \{0 , 1\}^{E(\Zd)}$ is i.i.d. Bernoulli of probability $p \in [0 , 1]$ and denote by $\E_p$ the expectation with respect to this percolation. The first main result of this article shows that, if the probability $p$ satisfies $p > p_{c, \mathrm{edge}}(d)$ where $p_{c, \mathrm{edge}}(d)$ is the critical probability for the Bernoulli edge percolation (N.B. we have $p_{c, \mathrm{edge}}(2) = 1/2$) and if $\beta$ is chosen sufficiently large, then the expectation (with respect to the underlying percolation configuration) of the two-point function of the Villain model decays polynomially fast in two dimensions and remains bounded away from $0$ in dimension $3$ and higher.

\begin{theorem}[Phase transitions for the Villain model on a supercritical Bernoulli percolation cluster] \label{proofKTsitegeneral}
For the Villain model in a random environment given by a Bernoulli edge percolation of parameter $p$, the following hold:
\begin{itemize}
        \item In dimension $d = 2$, for any $p > 1/2$, there exists an inverse temperature $\beta_{\mathrm{KT}}^{\mathrm{Vil}}(p) < \infty$, such that, for any $\beta \geq \beta_{\mathrm{KT}}^{\mathrm{Vil}}(p)$, the function $x \mapsto \mathbb{E}_p [  \left\langle \cos  ( \theta_0 - \theta_x ) \right\rangle_{\beta, \omega}^\mathrm{Vil}]$ decays polynomially fast in $|x|$.
        \item In dimension $d \geq  3$, for any $p > p_{c , \mathrm{edge}}(d)$, there exists an inverse temperature $\beta_{c}^{\mathrm{Vil}}(d, p) < \infty$, such that, for any $\beta \geq \beta_{c}^{\mathrm{Vil}}(d, p)$, the function $x \mapsto \mathbb{E}_p [  \left\langle \cos  ( \theta_0 - \theta_x ) \right\rangle_{ \beta, \omega}^{\mathrm{Vil}}]$ is bounded away from $0$.
    \end{itemize}
\end{theorem}

\begin{remark} \label{remark1.4}
Let us make a few remarks about the previous theorem:
\begin{itemize}
    \item The proofs written below (based on the argument of~\cite{van2023elementary, garban2022continuous}) give the following lower bounds: there exists a constant $c := c(p) > 0$ such that
    \begin{equation*}
        \mathbb{E}_p [  \left\langle \cos(\theta_0 - \theta_x) \right\rangle_{\beta, \omega}^{\mathrm{Vil}}] \geq \frac{c}{|x|} ~\mbox{in}~ d=2 ~~\mbox{and}~~ \inf_{x \in \Zd} \mathbb{E}_p [  \left\langle \cos(\theta_0 - \theta_x) \right\rangle_{\beta, \omega}^{\mathrm{Vil}}] \underset{\beta \to \infty}{\longrightarrow} 1 ~\mbox{in}~ d\geq 3.
    \end{equation*}
    \item It is known that, for any $x \in \Zd$, the expected two-point function $\mathbb{E}_p [  \left\langle \cos  ( \theta_0 - \theta_x ) \right\rangle^{\mathrm{Vil}}_{ \beta, \omega}]$ is increasing in the probability $p$ and in the inverse temperature $\beta$ (N.B. These results are a consequence of the Ginibre inequality).
    Since the right-hand side of~\eqref{eq:Ginibreanyperco} decays exponentially fast in $|x|$ for sufficiently small $\beta$, the same result holds for the expected two-point function $\mathbb{E}_p [  \left\langle \cos  ( \theta_0 - \theta_x ) \right\rangle^{\mathrm{Vil}}_{ \beta, \omega}]$ for any $p \in [0,1]$.
    \item For $p < 1/2$, the percolation configuration has only finite clusters (almost surely) and it is relatively easy to show that the expectation of the two-point function decays exponentially fast for any value of $\beta \in (0 , \infty)$.
    \item It would be possible to extend the result to other type of disorders such as the supercritical Bernoulli site percolation or the Poisson-Voronoi environment (as in~\cite{DG25}).
\end{itemize}
\end{remark}

\subsubsection{The annealed Villain model} \label{sec:sec113}

As a byproduct of Theorem~\ref{proofKTsitegeneral}, we obtain the existence of phase transitions for a more general class of spin systems introduced in the following definition. 

\begin{definition}[Annealed Villain interaction]\label{d.AV}
    Let $\kappa$ be a probability measure on $(0 , \infty)$ and $\beta \in (0 , \infty)$ be an inverse temperature. We define the annealed Villain interaction associated with the measure $\kappa$ and with inverse temperature $\beta$ to be the function $F_{\kappa , \beta} : [0 , 2 \pi) \to \R$ given by the formula
    \begin{equation*}
        F_{\kappa, \beta}(\theta) = \int_0^\infty v_{\beta J}( \theta) \kappa(d J).
    \end{equation*}
    Given a finite set $\Lambda \subseteq \Zd$, we define the spin system with annealed Villain interaction associated with the measure $\kappa$ and at inverse temperature $\beta$ in the set $\Lambda$ to be the probability distribution on the set of configurations $[0 , 2\pi)^{\Lambda}$
    \begin{equation} \label{def.annealedVillainmodelintro}
        \mu^{\mathrm{Ann-Vil}}_{L, \kappa, \beta}(d \theta) := \frac{1}{Z^{\mathrm{Ann-Vil}}_{L, \kappa, \beta}} \prod_{\substack{x , y \in \Lambda \\ x \sim y}} F_{\kappa, \beta}(\theta_x - \theta_y) \prod_{x \in \Lambda} d\theta_x.
    \end{equation}
    We denote by $\mu^{\mathrm{Ann-Vil}}_{\kappa, \beta}$ the limit as $\Lambda \uparrow \Zd$ of the finite-volume measures $\mu^{\mathrm{Ann-Vil}}_{\Lambda, \kappa, \beta}$ and by $\left\langle \cdot \right\rangle_{\kappa, \beta}^{\mathrm{Ann-Vil}}$ the corresponding expectation.
\end{definition}

\begin{remark} 
    Let us point out that:
\begin{itemize}
\item Regarding the measure $\kappa$, we will always choose it so that the measure~\eqref{def.annealedVillainmodelintro} is well-defined for any $\Lambda \subseteq \Zd$ and $\beta \geq 0$ (i.e., so that the partition function is finite).

\item The existence and uniqueness of the thermodynamic limit is guaranteed by an extension of the Ginibre inequality 
to the annealed Villain models 
(which follows from an adaptation of the argument of~\cite[Lemma 1]{campbell1998isotropic} or~\cite[Appendix D]{AHPS}). 
\item By~\cite[(9.10)]{AHPS}, the function $\theta \mapsto e^{\beta \cos \theta}$ is an annealed Villain interaction, 
so that the $XY$ model in particular can be seen as an annealed Villain.
\end{itemize}
\end{remark}

The second main result of this article shows the existence of phase transitions for spin systems with annealed Villain potential.

\begin{theorem}[Phase transitions for spin systems with annealed Villain interaction]\label{Th:annealedVillain}
Let $\kappa$ be a probability measure on $(0 , \infty)$. Then the following hold:
    \begin{itemize}
        \item In dimension $d = 2$, there exists an inverse temperature $\beta_{\mathrm{KT}}^{\mathrm{Ann-Vil}}(\kappa) < \infty$, such that, for any $\beta \geq \beta_{\mathrm{KT}}^{\mathrm{Ann-Vil}}(\kappa)$, the function $x \mapsto  \left\langle \cos  ( \theta_0 - \theta_x ) \right\rangle^{\mathrm{Ann-Vil}}_{\kappa , \beta}$ decays polynomially fast in $|x|$.
        \item In dimension $d \geq  3$, there exists an inverse temperature $\beta_{c}^{\mathrm{Ann-Vil}}(d, \kappa) < \infty$, such that, for any $\beta \geq \beta_{c}^{\mathrm{Vil}}(d, \kappa)$, the function $x \mapsto  \left\langle \cos  ( \theta_0 - \theta_x ) \right\rangle^{\mathrm{Ann-Vil}}_{\kappa , \beta}$ is bounded away from $0$.
    \end{itemize}
\end{theorem}

\begin{remark}
Let us make a few remarks about this theorem:
\begin{itemize}
    \item This BKT transition was established in \cite{AHPS} under the assumption that the Fourier coefficients of the function $F_{\beta, \kappa}$ are log-concave and that they satisfy a \emph{divisibility} condition (see~\cite[Definition 9.4]{AHPS}). This second condition can be removed if the spin system is considered on a planar graph which is dual to a degree-$3$ graph. 
    \item Our result would also hold if $\kappa$ were assumed to be a probability measure on $[0, \infty)$ under the additional condition $\kappa(\{ 0 \}) < p_{c,\mathrm{edge}}(d)$. 
    \item Even though this will not be formally proved in this article, it would be possible to show the exponential decay of the two-point functions for $\beta \ll 1$ by adapting the argument of Dobrushin~\cite{dobruschin1968description} (at least if the tail of the measure $\kappa$ decays sufficiently fast).
 We also refer to the recent work \cite{d2026free} which established analyticity of the pressure for XY and Villain in the entire high temperature region.  
    
    \item It would be possible to modify the proof of Theorem~\ref{Th:annealedVillain} so as to prove the the existence of the BKT and order/disorder phase transitions for spin systems with annealed Villain interaction on a supercritical Bernoulli percolation cluster.
\end{itemize}

\end{remark}

\subsection{Integer-valued height functions} \label{sec:presentationheight}

\subsubsection{The models}

The second class of models considered in this article are the integer-valued height functions on the two-dimensional hypercubic lattice $\Z^2$. One of the prototypical example of integer-valued height functions is the integer-valued Gaussian free field defined as follows. Let us fix a finite set $\Lambda \subseteq \Z^2$, an inverse temperature $\beta \in (0, \infty)$ and consider the space of integer-valued height functions with Dirichlet boundary condition $\Omega_{\Lambda}^{\Z} := \left\{ \varphi : \Z^2 \to \Z \, : \, \varphi \equiv 0 ~\mbox{on}~\Z^2 \setminus \Lambda \right\}$. The integer-valued Gaussian free field with inverse temperature $\beta$ on the set $\Lambda$ is then defined to be the probability distribution on the space $\Omega_{\Lambda}^{\Z}$ given by
\begin{equation}\label{e.ZGFF}
   \mu_{\Lambda, \beta}^{\Z-\mathrm{GFF}}( \left\{ \varphi \right\}) := \frac{1}{Z^{\Z-\mathrm{GFF}}_{\Lambda , \beta}} \exp \left( - \frac{1}{2\beta}  \sum_{x \sim y} \left( \varphi(x) - \varphi(y) \right)^2 \right).
\end{equation}
We will mostly make use of this definition when the set $\Lambda$ is the box $\Lambda_L := \left\{ - L , \cdots, L \right\}^2$ and denote by $\mathrm{Var}^{\Z-\mathrm{GFF}}_{L , \beta}$ the variance with respect to the measure $\mu_{\Lambda_L, \beta}^{\Z-\mathrm{GFF}}$.

\smallskip

The XY height function is a variant of the integer-valued Gaussian free field defined according to the identity, for any $\varphi \in \Omega_{\Lambda}^{\Z}$,
\begin{equation*}
    \mu_{\Lambda , \beta}^{\Z-\mathrm{XY}} (\left\{ \varphi \right\}) := \frac{1}{Z_{\Lambda,\beta}^{\Z-\mathrm{XY}}} \prod_{x \sim y} I_{\left(\varphi(x) - \varphi(y)\right)} \left( \beta \right),
\end{equation*}
where we used the modified Bessel function defined by the identity, for $k \in \Z$, $\beta \in [0 , \infty)$
\begin{equation*}
    I_k(\beta) := \sum_{j = 0}^\infty \frac{1}{j! (j + |k|)!} \left( \frac{\beta}{2} \right)^{2j + |k|} ~~\mbox{so that}~~ e^{\beta \cos \theta} = \sum_{k \in \Z} I_k(\beta) e^{i k \theta}.
\end{equation*}
We denote by $\mathrm{Var}_{L, \beta}^{\Z-\mathrm{XY}}$ the variance with respect to $ \mu_{\Lambda_L, \beta}^{\Z-\mathrm{XY}}$.

\smallskip

These two models of height functions exhibit a phase transition known as the roughening phase transition characterised by the following property (see~\cite{frohlich1981kosterlitz, lammers2022height, lammers2022dichotomy} and Figure~\ref{fig:figure1.1}):
\begin{itemize} 
    \item A localised regime at low temperature
    \begin{equation} \label{eq:eq1.5localisation}
        \sup_{L \in \N} \var^{\Z-\mathrm{GFF}}_{L , \beta} \left[ \varphi(0) \right] < \infty ~~\mbox{for}~ \beta \ll 1 
    \end{equation}
    \item A delocalised regime at high temperature 
    \begin{equation} \label{eq:delochightemp}
        \var^{\Z-\mathrm{GFF}}_{L , \beta} \left[ \varphi(0) \right] \geq c \ln L ~~\mbox{for}~ \beta \gg 1,
    \end{equation}
\end{itemize}
We remark that the same property holds for the XY height function and that, due to a correlation inequality from Regev and Stephens-Davidowitz~\cite{regev2017inequality} (and to~\cite{AHPS, van2023duality} for the extension of the result to the XY height function, see also \cite{frohlich1978correlation}), the variance of $\varphi(0)$ is increasing in the parameters $\beta$ and $L$ (see Theorem~\ref{prop:monotonicity}).
%\margin{C: Should we "dare" mentioning Frohlich-Park as praised by Roland or Marek $:-)$ here ? what do you think ? D: I don't know, we could. C: I just did :)! }
We further note that there is a (well-studied) duality between the integer-valued height functions and the spin systems introduced in the previous section (more specifically, the Villain model is dual to the integer-valued Gaussian free field and the XY model is dual to the XY height function) and all the known proofs of the BKT phase transition require to establish the roughening phase transition.

\subsubsection{The result}

As it was the case for the XY and Villain models, we will be interested in the existence of the roughening phase transition for the integer-valued Gaussian free field and the XY height function in the presence of a quenched random disorder. More specifically, we introduce the disordered versions of the integer-valued height functions studied in this article.

\begin{definition}[Disordered height function] \label{def:XYheightandIVGFFintro}
Let $L \in \N$ be an integer, $\beta \in (0 , \infty)$ be an inverse temperature and $\omega \in \{0 , 1 \}^{E(\Z^2)}$ be an edge percolation configuration. We define:
\begin{itemize}
    \item \underline{The disordered integer-valued Gaussian free field} to be the probability distribution on the set $\Omega_{\Lambda_L}$ given by the identity
\begin{equation} \label{def:IVGFFperc}
    \mu_{\Lambda_L , \beta, \omega}^{\Z-\mathrm{GFF}} (\left\{ \varphi \right\}) := \frac{1}{Z_{\Lambda_L,\beta , \omega}} \exp \left( - \frac{1}{2\beta} \sum_{x \sim y}  \frac{\left( \varphi(x) - \varphi(y) \right)^2}{\omega_{xy}} \right),
\end{equation}
    where we used the convention $\frac{k^2}{0} = \infty$ for $k \in \Z \setminus \{ 0 \}$ and $\frac{k^2}{0} = 0$ for $k=0$.
    We denote by $\mathrm{Var}_{L , \beta, \omega}^{\Z-\mathrm{GFF}}$ the variance with respect to this measure (and omit the index $\omega$ for the percolation configuration $\omega \equiv 1$).
    \item \underline{The disordered XY height function} to be the probability distribution on the set $\Omega_{\Lambda_L}$ given by the identity
\begin{equation} \label{def:XYheightperco}
    \mu_{\Lambda_L , \beta,\omega}^{\Z-\mathrm{XY}} (\left\{ \varphi \right\}) := \frac{1}{Z_{\Lambda_L,\beta , \omega}^{\Z-\mathrm{XY}}} \prod_{x \sim y} I_{\left(\varphi(x) - \varphi(y)\right)} \left( \omega_{xy} \beta \right).
\end{equation}
We denote by $\mathrm{Var}_{L , \beta, \omega}^{\Z-\mathrm{XY}}$ the variance with respect to this measure (and omit the index $\omega$ for the percolation configuration $\omega \equiv  1$).
%\margin{C: Here maybe we can discuss less the $XY$ dual as it will just be part of annealed Gaussian, right ? So maybe in that sense we may just focus on $\Z$ GFF -- We NEED BOTH -- to get all Z annealed Gaussian fields we need delocalisation of BOTH !!!}
\end{itemize}
\end{definition}

\begin{remark}
This way of encoding the disorder means that we enforce the constraint $\varphi(x) = \varphi(y)$ for any pair of neighbouring vertices $x , y \in \Lambda_L$ such that $\omega_{xy} = 0$. This constraint tends to reduce the fluctuations of the height function, and more specifically, the results of~\cite{regev2017inequality, AHPS, van2023duality} (see Theorem~\ref{prop:monotonicity}) imply that the functions $\omega \mapsto \mathrm{Var}_{L , \beta, \omega}^{\Z-\mathrm{GFF}} \left[ \varphi(0) \right]$ and $\omega \mapsto \mathrm{Var}_{L , \beta, \omega}^{\Z-\mathrm{GFF}} \left[ \varphi(0) \right]$ are increasing. In particular, for any $\omega \in \{0 , 1\}^{E(\Z^2)}$,
\begin{equation} \label{ineq:12220608}
    \mathrm{Var}_{L , \beta, \omega}^{\Z-\mathrm{GFF}} \left[ \varphi(0) \right] \leq  \mathrm{Var}_{L , \beta}^{\Z-\mathrm{GFF}} \left[ \varphi(0) \right].
\end{equation}
\end{remark}

%\margin{C: should we add a comment somewhere that one could also handle edge-percolation disorder ?}

\iffalse
We define the suitable critical percolation threshold for this model according of the identity
\margin{C: do we need to go through dual $1-$ ? (I find it slightly confusing).  Would it be simpler to directly consider an edge percolation disorder ?}
\begin{multline} \label{eq:defofpcsite}
    \mathfrak{p}_c := \sup \left\{ p \in [0,1]\, : \, \mathbb{P}_p ~\mathrm{almost-surely}, ~\mathrm{ there ~ exists ~ an ~ infinite ~ connected ~component~ in} \right. \\ 
   \left. ~\mathrm{the~ edge ~ percolation ~ configuration} ~ \tilde r \in \{0,1\}^{E(\Zd)} ~\mathrm{with}~ \tilde r_{xy} := 1 - r_x r_y \right\}.
\end{multline}
The second main result of this article shows the delocalisation of the integer-valued height functions defined above for sufficiently high $\beta$ and for $p > \mathfrak{p}_c$.
\fi

\begin{figure}
\begin{subfigure}{.5\linewidth}
\centering
\includegraphics[scale=0.6]{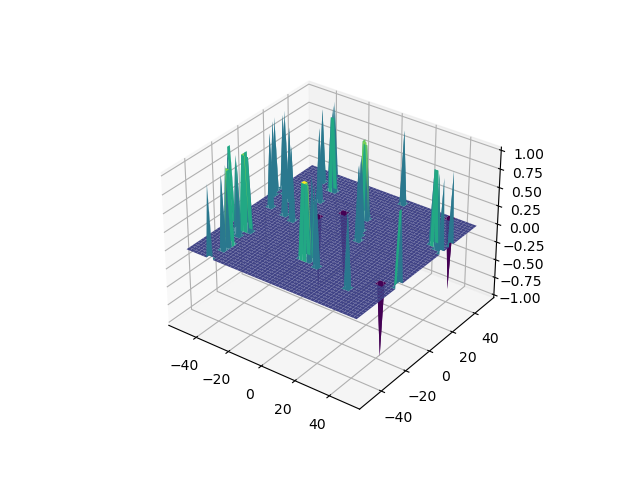}
\caption{Localised phase at low temperature.}
\end{subfigure}%
\begin{subfigure}{.5\linewidth}
\centering
\includegraphics[scale=0.6]{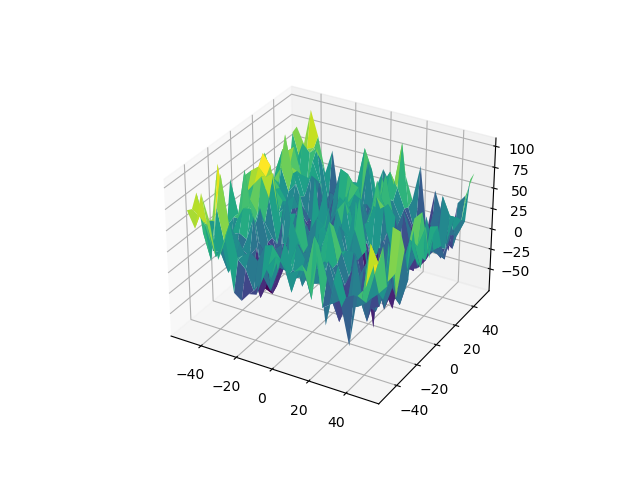}
\caption{Delocalised phase at high temperature.}
\end{subfigure}
\caption{An illustration of the phase transition for integer-valued height functions.} \label{fig:figure1.1}
\end{figure}

The third main result of this article is the existence of a roughening phase transition for the disordered integer-valued height functions when the disorder $\omega \in \{0 , 1\}^{E(\Z^2)}$ is distributed according to a supercritical i.i.d. edge percolation.% (N.B. We recall that $p_{c,\mathrm{edge}}(2) = 1/2$). 

\begin{theorem}[Delocalisation for disordered height functions] \label{Thm:delocheightsupercrit}
%\margin{C: The parameter $p$ here is not the same as in the abstract (in abstract it is $1-p$). I guess this is fine, as we do not want to see a $1-p$ in the abstract. Could be $\bar p$ possibly ? \textbf{I tried that in abstract, please change if you don't like.} A very minor issue in abstract "we assume $\phi(x)=\phi(y)$ independently on each edge" I guess this is not quite what we're doing right ?We are enforcing this condition independently on edges indeed but then there may be some other edges where the field $\phi$ will decide to agree on both enpoints. As this is an abstract sentence, maybe it is fine if it is slightly vague ? I did not see a better, yet simple, way to do.}
    For any $p > 1/2$, there exists an inverse temperature $\beta_{\mathrm{Deloc}}(p) \in (0 , \infty)$ and a constant $c_0 \in (0 , \infty)$ such that for any $\beta \geq \beta_{\mathrm{Deloc}}(p)$ and any $L \in \N^*$,
    \begin{equation*}
        \mathbb{E}_p \left[ \mathrm{Var}_{L , \beta, \omega}^{\Z-\mathrm{GFF}} \left[ \varphi(0) \right] \right] \geq c_0 \ln L ~~\mbox{and}~~ \mathbb{E}_p \left[  \mathrm{Var}_{L , \beta, \omega}^{\Z-\mathrm{XY}}  \left[ \varphi(0) \right] \right] \geq c_0 \ln L.
    \end{equation*}
\end{theorem}
\begin{remark}
We make four comments regarding this statement:
\begin{itemize}
\item In the abstract of this paper, where this statement is stated less precisely, note that $\bar p$ then corresponds to $1-p$.
\item In the case where $\beta$ is small, a combination of~\eqref{eq:eq1.5localisation} and~\eqref{ineq:12220608} shows that the variance of the height $\varphi(0)$ is bounded uniformly in $L$.
\item A logarithmic upper bound on the variance was proved (in a more general setting) in~\cite[Corollary 3.1]{van2023duality}.
\item For $p < 1/2$, there exists $\mathbb{P}_p$ almost-surely an infinite connected component of edges for which $\omega_{xy} = 0$. In particular, in the box $\Lambda_L$, there exists (with high probability) a macroscopic connected component on which the height function is constant and equal to $0$. From this observation, it can be deduced that the expectation of the variance of the height $\varphi(0)$ is bounded uniformly in $L$ for any $\beta \in (0 , \infty)$. 
\end{itemize}
\end{remark}

\subsubsection{Integer-valued height functions with annealed Gaussian potential}

Similarly to Section~\ref{sec:sec113}, we make use of Theorem~\ref{Thm:delocheightsupercrit} to show a delocalisation result for a broader class of integer-valued height functions, namely the ones with ``annealed Gaussian potentials" defined below. 

\begin{definition}[Annealed Gaussian potentials] 
A potential $V : \R \to \R$ is said to be an annealed Gaussian potential if the map $k \mapsto k^2 e^{-V(k)}$ is summable over $\Z$ and if there exists a Borel probability measure $\mu_V$ on $(0,\infty)$ such that, for any $x \in \R$,
%\margin{C: naive question : is it the same as requiring this constraint only on $k\in\Z$ ? (since we have integrability conditions, this may ?) maybe very naive! (but also clearly relevant!) }

\begin{equation} \label{eq:defannealedgaussian}
    e^{-V(x)} = \int_0^\infty e^{- \frac{x^2}{2J}} \mu_V(d J).
\end{equation}
\end{definition}

\begin{remark}
We remark about this definition that:
\begin{itemize}
    \item 
    %The summability assumption is necessary and sufficient this assumption is necessary and sufficient to ensure that the model introduced in~\eqref{def:IVGFFgeneral}
     The summability assumption ensures that the model introduced in~\eqref{def:IVGFFwithpotentialV}
     below is well-defined and that the variances of the heights are finite. In terms of the probability distribution~$\mu_V$, this is equivalent to requiring that $\int_0^\infty J \sqrt{J} \mu_V(dJ) < \infty$. 

%\margin{This condition may also be seen on the annealed Villain side. In the sense that at very high $J\equiv \beta$, we have that $\int_{\S^1} p_{1/J}(\theta)d\theta \asymp $ But we miss a $\beta$-normalisation here I guess to make a nice duality with $IV-GFF$ ? For GFF, the partition function at high $\beta$ should be like $1/\sqrt{2\pi \beta}^{edges}$ }

    %\item The function $V$ and the measure $\mu_V$ completely characterise each other. 
    \item There is no gain of generality if, instead of assuming that $\mu_V$ is a probability measure, we assume that $\mu_V$ is a non-negative finite measure on $(0,\infty)$. This convention is convenient because it ensures that $V(0) = 0$.
    \item Some examples of annealed Gaussian potentials include the functions $V(k) = |k|^{\alpha}$ for $\alpha \in (0,2]$ (see~\cite{penson2010exact} and the references therein) and the (normalised) modified Bessel function introduced in~\eqref{def:XYheightandIVGFFintro}  (see~\cite[Lemma 9.6]{AHPS}).
\end{itemize}

\end{remark}

\begin{definition}[Integer-valued height functions with annealed Gaussian potential]
Given an integer $L \in \N$ and an inverse temperature $\beta \geq 0$, we define the integer-valued height function with potential $V$ to be the probability distribution on the set $\Omega_{\Lambda_L}^\Z$ given by the identity

%\purple{\bf C: it is a nice way to prescribe a temperature in those systems! Note that this is not the way the temperature inters in the dual potential to $XY$. There, rather to each $\beta$, we have a $V_\beta^{\mathrm{XY}}$ potential. And it is only at high enough $\beta$ that we reach delocalisation ! But this view point is relevant only for $XY$ I guess, and for $e^{- |x|^\alpha}$, both viewpoints are equivalent! \textbf{N.B. One may argue that $\phi \mapsto \phi / \sqrt{\beta}$ would sound more natural, but this is not true as this scaling is specific to IV-GFF}.  As I mentioned further below, it would be nice to have such a rescaling for $\S^1$ but this looks less clear to me. So Probably, there is hope there to get some BKT just by saying IF THE POTENTIAL is ... enough... then ..... } 

%\margin{C: I changed here and further below in the text $\beta$ to $\sqrt{\beta}$ (indicated throughout in blue), as discussed with Paul, so that it matches with the other definition. Please check I did not forget places $:)$. It was a rather substantial change. You may check $\blue{\sqrt{\beta}}$, $\blue{\beta}$ or $\blue{\beta_V}$.  I also added below  a Remark that annealed Villain is the Fourier dual of the annealed Gaussian under the right definition. }
\begin{equation} \label{def:IVGFFwithpotentialV}
    \mu_{\Lambda_L , \beta}^{\Z-V} (\left\{ \varphi \right\}) := \frac{1}{Z_{\Lambda_L,\beta}^{\Z-V}} \exp \left( - \sum_{x \sim y}  V \left( \frac{ \varphi(x) - \varphi(y)}{\sqrt{\beta}} \right)\right).
\end{equation}
We denote by $\mathrm{Var}_{L , \beta}^{\Z-V}$ the variance with respect to the measure $\mu_{L , \beta}^{\Z-V}$.
\end{definition}

%\purple{C: I see so here $\beta$ is set to be inside the potential $V$ right ? Maybe outside is simpler for Ginibre ? Paul: No !! if $\beta$ is not an integer ! in the general case, it may no longer be an ANNEALED gaussian. $\beta=3$ is maybe a 3 convolution but otherwise ?? Fractional convolution. So conclusion $\beta$ is outside ! But $\beta$ inside is not wonderful either : only graph monotony, not quite ginibre !! Beta integer oustside, then Ginibre along integers by Graph monotony. }

The main result of this section shows that, on the lattice $\Z^2$, these height functions are delocalised at sufficiently large inverse temperature.

\begin{theorem}[Delocalisation for integer-valued height functions with annealed Gaussian potential] \label{Thm:delocheightanealedgauss}
Let $V : \Z \to \R$ be an annealed Gaussian potential. Then there exists an inverse temperature $\beta_V \in (0 , \infty)$ (depending on $V$) and a constant $c_0 > 0$ (universal) such that, for any $\beta \geq \beta_V$ and any $L \in \N^*$,
\begin{equation*}
     \mathrm{Var}_{L , \beta}^{\Z-V} \left[ \varphi(0) \right] \geq c_0 \ln L.
\end{equation*}
%\margin{C: also here a $\beta_c=\beta_c(V)$ I guess. \textbf{No for the moment we don't have Ginibre.}}
\end{theorem}

\begin{remark} \label{remark4intro}
We make a few comments regarding this statement:
\begin{itemize}

\item Note that the above integer-valued height functions with annealed Gaussian potentials is precisely the Fourier dual of the annealed Villain model from Definition \ref{d.AV}. This can be seen from the fact that
\begin{align*}\label{}
\hat F_{\kappa,\beta}(k)  & = \int_0^{2\pi} F_{\kappa,\beta}(\theta) e^{i k \theta} d\theta \\
& = \int_0^{2\pi}  \left( \int_0^\infty v_{\beta J}(\theta) e^{i k \theta} \kappa(dJ)  \right)  d\theta 
 = \int_0^\infty e^{-\frac {k^2} {2 \beta J}} \kappa(dJ)
\end{align*}

\item In the case where $\beta$ is small, a Peierls argument can be used to show that the variance of the height $\varphi(0)$ is bounded uniformly in $L$ (at least when $V$ is convex). 
%\margin{C: Why $V$ convex here ? I see $:)$.}
\item A matching logarithmic upper bound on the variance was proved (in a more general setting) in~\cite[Corollary 3.1]{van2023duality}.
\item 
Under the additional conditions that the potential $V$ is convex, this delocalisation result is proved in \cite{AHPS} for degree-3 graphs, and on $\Z^2$ if the potential is furthermore assumed to be {\em divisible} (see Definition 9.4 in \cite{AHPS}). For example when the graph is $\Z^2$, they manage to handle $V(x)=|x|^\alpha$ with $\alpha\in(1,2]$ for the $\Z$-models. This result was recently extended to any exponent $\alpha \in (0, 2 ]$ by Ott and Schweiger~\cite{ott2025quantitative}, and our theorem extends the delocalisation to all annealed Gaussian potentials via a very different approach. N.B. The convexity assumption in \cite{AHPS} was important in order to rely on Lammers' delocalisation result \cite{lammers2022height}, where convexity is hard to remove as it ensures an essential FKG property.
\item The result can be extended to integer-valued height functions with annealed Gaussian potential on a percolation configuration (as in Theorem~\ref{Thm:delocheightsupercrit}).
\end{itemize}
\end{remark}

We finally extend the previous results to more general (and non-necessarily planar) two-dimensional graphs with finite-range interactions. The class of graphs considered is defined below (see Figure~\ref{fig:2dshiftinv}).

\begin{definition}[Two-dimensional shift-invariant graphs] A two-dimensional shift-invariant graph is a connected graph $G = (W, E)$ (with vertex set $W$ and edge set $E$) which is embedded into $\R^2$, locally finite and is invariant under the action of some full rank lattice $\mathcal{L}$ of $\R^2$. We will assume that the graph is embedded in $\R^2$ with $0 \in \R^2$ being a vertex of $G$.
\end{definition}

\begin{figure}
\centering
\includegraphics[scale=0.4]{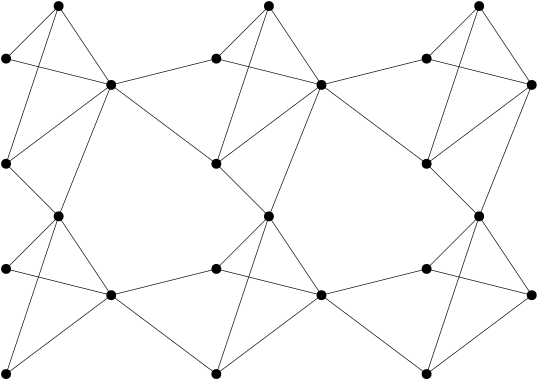}
\caption{A two-dimensional shift-invariant graph.} \label{fig:2dshiftinv}
\end{figure}

\begin{definition}[Integer-valued height function on a shift-invariant graph]
	Given a two-dimensional shift-invariant graph $G = (\W, E)$ embedded in $\R^2$ and an integer $L \in \N$, we consider the space of functions 
	$$\Omega_{L,\W} := \left\{ \varphi : \W \to \Z \, : \, \varphi(x) = 0 ~\mbox{for}~ x \notin [-L , L]^2 \right\}.$$ 
	Given an inverse temperature $\beta \geq 0$, we define the integer-valued height function with potential $V$ to be the probability distribution on the set $\Omega_{L,\W}^\Z$ given by the identity
\begin{equation*}
    \mu_{L , G , \beta}^{\Z-V} (\left\{ \varphi \right\}) := \frac{1}{Z_{L , G,\beta}^{\Z-V}} \exp \left( - \sum_{xy \in E}  V \left( \frac{ \varphi(x) - \varphi(y)}{\sqrt{\beta}} \right)\right).
\end{equation*}
We denote by $ \mathrm{Var}_{L , G , \beta}^{\Z-V}$ the variance with respect to the measure $\mu_{L , G , \beta}^{\Z-V}$.
\end{definition}

The final result of this article shows that these height functions are delocalised when the inverse temperature is sufficiently high.

\begin{theorem}[Delocalisation for integer-valued height functions with annealed Gaussian potential] \label{Thm:delocheightanealedgaussgeneralgraph}
Let $V : \Z \to \R$ be an annealed Gaussian potential and let $G := (W , E)$ be a two-dimensional shift-invariant graph. Then there exists an inverse temperature $\beta_{V,G} \in (0 , \infty)$ (depending on the potential $V$ and the graph $G$) and a constant $c_0 > 0$ (universal) such that, for any $\beta \geq \beta_{V,G}$ and any $L \in \N^*$,
\begin{equation*}
     \mathrm{Var}_{L , G , \beta}^{\Z-V} \left[ \varphi(0) \right] \geq c_0 \ln L.
\end{equation*}
\end{theorem}
We complete this section by mentioning that the comments stated in Remark~\ref{remark4intro} apply in the setting of this theorem.

\subsection{Technique of proof and related results.} \label{sec:1.4techniquesofproof}

\ni
\textbf{Related results.} As mentioned above, for the models without disorder, all the known proofs of the BKT phase transition require to establish the existence of the roughening phase transition for the (dual) models of integer-valued height functions (see~\eqref{eq:delochightemp}). In order to prove the existence of the phase transition~\eqref{eq:delochightemp}, two distinct approaches can be used:
\begin{itemize}
    \item A first one based on Coulomb gas expansions due to Fr\"{o}hlich and Spencer~\cite{frohlich1981kosterlitz} 
    (see also~\cite{kharash2017fr}), which has led to several subsequent contributions~\cite{wirth2019maximum, garban2023quantitative, garban2023statistical, garban2023invisibility, ott2025quantitative}.
    \item A second one based on percolation arguments and loop expansions. 
    This approach was initiated by Lammers~\cite{lammers2022height} and relies on a non-coexistence theorem for planar percolation due to Sheffield~\cite{sheffield2005random} 
    (see~\cite{duminil2019sharp} for a simpler proof). 
    It has then been extended in various directions~\cite{lammers2022dichotomy, lammers2023delocalisation, van2023elementary, AHPS, lammers2023bijecting}.
\end{itemize}
None of these techniques seem to extend (easily) to the disordered setting studied in this article 
(we refer to~\cite[Section 1.3]{DG25} where this is discussed in more details).

More refined properties of these height functions can be studied, and in this direction, 
a fundamental question is the one of the identification of the scaling limit of these models. 
A longstanding conjecture asserts that this scaling limit should be a continuum two-dimensional 
Gaussian free field for a broad class of height functions, and was recently confirmed for the
(high-temperature) integer-valued Gaussian free field in~\cite{bauerschmidt2022discrete, bauerschmidt2022discreteII} 
and for a family of SOS models with nonzero slope in~\cite{laslier2024tilted}. 
Let us also mention the recent breakthrough \cite{duminil2026gaussian} 
which proves an invariance principle to Gaussian Free Field for a large class of six-vertex model height functions. 
%\margin{C: I added this reference $:)$}

We finally point out that there has been some recent progress in proving delocalisation 
for other models of height functions (such as the uniform Lipschitz functions and the height function associated with 
the six-vertex model) in~\cite{GM21, DCHLRG22, lis2021delocalization, DCKMO24, glazman2025delocalisation}. 
Also the fact the BKT transition of XY is robust to the presence of disorder (\cite{DG25}) is an essential step in the recent work \cite{yuan2025infinite}.
% \margin{C: I also added this one} 

\ni

\medskip

\textbf{Proof idea.} We start by briefly presenting the main ideas of~\cite{DG25} where the existence of BKT and the order/disorder phases transitions for the XY model is proved, and then show how the arguments of~\cite{DG25} are adapted to the Villain model and the integer-valued height-functions.

\medskip

\textbf{The XY model (from~\cite{DG25}).} In the case of the XY model, the argument is decomposed into two steps: the existence of phase transitions for the XY model on a percolation configuration is first proved in the case of a high-density percolation (i.e., for $p = 1 - \ep$ for some small $\ep \ll 1$), the result is then upgraded to any percolation probability $p > p_c(d)$.

The first step relies on a correlation inequality due to Wells~\cite{Wellsthesis} (see also~\cite{aizenman1980comparison, bricmont1981periodic, dunlop1985correlation, madrid2022comparison}), and a simplified version of his result is stated below. Let us fix a dimension $d \geq 2$, an integer $L \in \N$, an inverse temperature $\beta \geq 0$, and define a probability distribution (referred to as the Wells' disorder) on the set of percolation configurations $\omega \in \{0 , 1 \}^{E\left(\Lambda_L\right)}$,
%\margin{C: I guess I should know this $:)$ but just to make sure, if this Wells was stated for site percolation configurations, I would be on board without hesitation. Is it also fine for edge configuation stated like that?}

\begin{equation} \label{eq:17201108}
    \nu_{L , \beta} \left( \left\{  \omega \right\} \right) := \frac{Z^{\mathrm{XY}}_{L , \beta, \omega}}{\sum_{\omega' \in \{0 , 1 \}^{E\left(\Lambda_L\right)}} Z^{\mathrm{XY}}_{L , \beta, \omega'}}.
\end{equation}
If we denote by $\E_{\nu_{L , \beta}}$ the expectation with respect to the measure~\eqref{eq:17201108}, then Wells' inequality~\cite{Wellsthesis} states that, for any $x \in \Zd$,
\begin{equation} \label{eq:wellsineqpresentation}
    \E_{\nu_{L , \beta}} \left[ \left\langle \cos \left( \theta_0 - \theta_x \right) \right\rangle_{L , \beta , \omega}^{\mathrm{XY}} \right] \geq \left\langle \cos \left( \theta_0 - \theta_x \right) \right\rangle_{L , \beta/2}^{\mathrm{XY}}.
\end{equation}
This inequality is a powerful result which asserts that the expectation of the two-point function of a disordered XY model (with the specific law~\eqref{eq:17201108} for the disorder) is larger than the two-point function of an XY model without disorder but with an inverse temperature divided by $2$. This inequality can be combined with two additional results:
\begin{itemize}
    \item The probability distribution $\nu_{L , \beta}$ is stochastically dominated by an i.i.d. Bernoulli percolation measure with probability $p := p(\beta) < 1$.
    \item By the Ginibre inequality, the function $\omega \mapsto \left\langle \cos \left( \theta_0 - \theta_x \right) \right\rangle_{L , \beta , \omega}^{\mathrm{XY}}$ is increasing.
\end{itemize}
A combination of these two properties with the Wells' inequality~\eqref{eq:wellsineqpresentation} and~\eqref{BKT} (for $d = 2$) and~\eqref{LRO} (for $d \geq 3$) completes the proof in the case of a high-density percolation.

\medskip

To extend the result to all supercritical percolation probability 
\( p > p_{c,\mathrm{site}}(d) \), the strategy of~\cite{DG25} was to use a renormalisation argument based on the notion of \emph{good boxes} introduced by Penrose and Pisztora~\cite{penrose-pisztora-1996, pisztora-percolation}. 
These boxes satisfy the following properties:
\begin{itemize}
    \item The probability that a box of sidelength \(L\) is good tends to \(1\) as \(L \to \infty\).
    \item Any two intersecting good boxes of the same sidelength are connected by a cluster contained within their union.
\end{itemize}
One may choose an integer \(L_0\) sufficiently large so that the probability for a box of sidelength \(L_0\) to be good is larger than \(p(\beta)\).
Heuristically, partitioning \(\mathbb{Z}^d\) into boxes of sidelength \(L_0\) allows us to reduce the problem of the phase transitions for the XY model on a supercritical percolation cluster to the corresponding problem on a high-density percolation cluster.
The latter can be treated using Wells' inequality as described above.

\medskip

\textbf{The Villain model.} In order to extend the proof described in the previous paragraph to the Villain model, a first attempt would be to try to prove a version of the Wells' inequality for this model. Unfortunately, this approach faces some difficulties as the proof of the Wells' inequality uses in a crucial way the definition of the XY model and it is unclear to us if/how this result could be proved (see Open Question \ref{OQ.1}).
%\margin{C: added a pointer to that OQ} 
In order to bypass this difficulty, we consider a spin system which is a combination of the XY and Villain models and is defined as follows. We first add two spins on each edge of $\Zd$. This operation splits the edges into three edges and we assign to two of them the XY potential with inverse temperature $\beta_1$ and to the remaining edge the Villain potential with inverse temperature $\beta_2$ (see Figure~\ref{intro:XYVillainmodel}). This model will be referred to as \textbf{the XY/Villain model} and it is introduced formally in Section~\ref{sectiondef.XY} (N.B. The existence and uniqueness of a thermodynamic limit can be established using the same arguments as for the XY or the Villain model).

\smallskip

Given a percolation configuration $\omega \in \{0 , 1\}^{E(\Lambda_L)}$, we define the disordered version of the XY/Villain model by considering the model on the graph where the closed edges for the percolation configuration $\omega$ are deleted. We denote by $\left\langle \cdot \right\rangle_{\beta_1 , \beta_2,\omega}^{\mathrm{XY/Vil}}$ the expectation with respect to this measure (and omit the subscript $\omega$ when $\omega \equiv 1$). 

\begin{figure}
    \centering
    \includegraphics[scale=0.6]{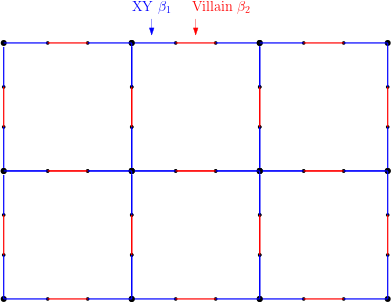}
    \caption{The XY/Villain model: two spins are added on each edge of $\Z^2$ (this operation splits each edge into three edges). The Villain potential with inverse temperature $\beta_2$ is assigned to the central edges (in red) and the XY potential with inverse temperature $\beta_1$ is assigned to the two other edges (in blue).}
    \label{intro:XYVillainmodel}
\end{figure}

\medskip

A number of properties can be established on the XY/Villain model:
\begin{enumerate}
    \item The Ginibre inequality applies to this model. In particular the two-point function is increasing in the parameters $\beta_1$ and $\beta_2$, and in the percolation configuration $\omega$.
    \item The two-point function of the XY/Villain model is smaller than the two-point function of the Villain model with inverse temperature $\beta_2$ (N.B. this result is obtained by sending the value $\beta_1$ to infinity and by applying the Ginibre inequality).
    \item The XY/Villain model has the same phase transitions as the XY (or the Villain) model: if the inverse temperatures $\beta_1$ and $\beta_2$ are chosen sufficiently large, then the two-point function decays polynomially fast in dimension $d = 2$ and remains bounded away from $0$ in dimensions $d \geq 3$.
    \item Contrary to the Villain model, it is relatively straightforward to adapt the argument of~\cite{Wellsthesis} to show that a version of Wells' inequality holds for the XY/Villain model. This is because the disorder can be incorporated in the model through the edges where the XY potential has been assigned (in which case one can apply the argument of~\cite{Wellsthesis}).
\end{enumerate}
These observations can then be combined with the argument of~\cite{DG25} presented above in order to show that, for $\beta_1 , \beta_2 \geq 0$, there exists $p := p(\beta_1 , \beta_2) < 1$ such that
%\margin{C: I get it for $\beta_1/2$ why $\beta_2/2$ ? \textbf{Not fully needed, but simplifies notations later. }}
\begin{equation*}
    \left\langle \cos \left( \theta_0 - \theta_x \right) \right\rangle_{\beta_1/2 , \beta_2/2}^{\mathrm{XY/Vil}}\leq \E_p\left[ \left\langle \cos \left( \theta_0 - \theta_x \right) \right\rangle_{\beta_1 , \beta_2 , \omega}^{\mathrm{XY/Vil}}  \right] \leq  \E_p\left[ \left\langle \cos \left( \theta_0 - \theta_x \right) \right\rangle_{\beta_2 , \omega}^{\mathrm{Vil}}  \right].
\end{equation*}
We may then conclude (in the case of a high-density percolation) by selecting $\beta_1$ and $\beta_2$ sufficiently large so that the left-hand side can be lower bounded using the third property listed above.

\smallskip

The extension from a high-density percolation to a general supercritical percolation can then be achieved by using the same renormalisation argument as the one presented above (N.B. The proof is essentially identical to that of the XY model in~\cite{DG25} and is similar to the one for the integer-valued height functions in Section~\ref{sec:renoarmargumentforheight}; we thus decided not to rewrite the entire argument and to only point out in Section~\ref{section5.2} below the main difference with the argument of~\cite{DG25}).

\medskip

\textbf{The XY height function.} Once again, we first prove the result in the case of a high-density percolation by relying on the strategy of~\cite{DG25} with some additional ideas from~\cite{van2023duality}.

More precisely, the main challenge is to establish the Wells' inequality for this model of height function. 
The proof of such a  Wells' inequality for height functions is given in Section~\ref{App.A}. It is one of the main novelties of this article.
%\margin{C: Almost copied/paste here a sentence you wrote Paul later in the text and that I liked! The sentence is still there, except I added "as mentioned earlier". }
This is achieved by using the results of the second author and Lis~\cite{van2023duality} on the duality between the XY model and the XY height function. Specifically, we rely on the identity~\cite[(2.2)]{van2023duality} which can be stated in our setting as follows: for any $L \in \N, \beta \geq 0$ and $\omega \in \{0,1\}^{E \left( \Lambda_L \right)}$,
\begin{equation} \label{intro:dualityXYheight}
     \mathrm{Var}_{L , \beta, \omega}^{\Z-\mathrm{XY}}  \left[ \varphi(0) \right] = \lim_{\lambda \to \infty} \left\langle F_\lambda (\theta)  \right\rangle_{L , \lambda, \beta, \omega}^{\mathrm{XY}}
\end{equation}
where $\lambda \in (0,\infty)$ is a non-negative real number, $\left\langle \cdot \right\rangle_{L , \beta,\omega, \lambda}^{\mathrm{XY}}$ is the expectation with respect to a certain variant of the XY model in which the parameter $\lambda$ has been incorporated and $F_\lambda$ is an explicit function of the angle configuration (N.B. we will not give more details on these definitions as it would require to introduce a few pieces of notation, but we refer to Section~\ref{App.A} for additional information). The key observation is that, for any $\lambda \in (0 , \infty)$, Wells' inequality applies to the right-hand side of~\eqref{intro:dualityXYheight}. As a result, we are able to show the following version of Wells' inequality for the XY height function: for any $L \in \N, \beta \geq 0$
\begin{equation} \label{eq:13421208}
    \E_{\bar{\nu}_{L , \beta}} \left[  \mathrm{Var}_{L , \beta, \omega}^{\Z-\mathrm{XY}}  \left[ \varphi(0) \right]  \right] \geq \mathrm{Var}_{L , \beta/2}^{\Z-\mathrm{XY}}  \left[ \varphi(0) \right] ~~\mbox{with}~~ \bar{\nu}_{L , \beta}(\{ \omega \}) := \frac{Z^{\Z-\mathrm{XY}}_{L , \beta, \omega}}{\sum_{\tilde \omega \in \{ 0 , 1 \}^{E\left(\Lambda_L\right)}} Z^{\Z-\mathrm{XY}}_{L , \beta, \tilde \omega}}.
\end{equation}
The rest of the proof follows from the same argument as in~\cite{DG25}: one can prove that the probability measure on the right-hand side of~\eqref{eq:13421208} is stochastically dominated by an i.i.d. measure with probability $p := p(\beta) < 1$ and that the variance of the height $\varphi(0)$ is increasing in the percolation configuration $\omega$ to deduce Theorem~\ref{Thm:delocheightsupercrit} for the XY height function on a high-density percolation.

The result can then be extended from high-density percolation to any supercritical percolation via a renormalisation argument. The argument for height functions is essentially the same as for the XY and Villain models, but is technically simpler in the case of the integer-valued Gaussian free field. We therefore present it for this latter model in Section~\ref{sec:renoarmargumentforheight}.

For the XY height function, although it would certainly be possible to carry out the renormalisation argument, there is an alternative, less technical approach based on the observation that this height function belongs to the general class of integer-valued height functions with annealed Gaussian potential (together with the delocalisation for the disordered integer-valued Gaussian free field). We present this argument in Section~\ref{sec:section4.7}.

\medskip

\textbf{The integer-valued Gaussian free field.} Finally the integer-valued Gaussian free field is treated by combining the ideas used in the case of the Villain model and the XY height function: we first introduce a height function which is defined as a combination of the XY height function and the integer-valued Gaussian free field (similarly to the XY/Villain model introduced above, see Section~\ref{sec:defXYGFFheightfunction} for a formal definition) and prove a version of Wells' inequality for this model. Equipped with this inequality, Theorem~\ref{Thm:delocheightsupercrit} can be established (for the integer-valued Gaussian free field and with a percolation probability $p$ close to $1$) using the same argument as for the Villain model.

The renormalisation argument, which allows us to upgrade the result from a high-density percolation to any supercritical percolation, is presented in Section~\ref{sec:renoarmargumentforheight} and is similar to that of the XY (or Villain) model.

\medskip

\textbf{The annealed Villain interaction and annealed Gaussian potentials.} The existence of phase transitions for the spin systems with annealed Villain interaction (resp. for the height functions with annealed Gaussian potential) can be obtained as a consequence of Theorem~\ref{proofKTsitegeneral} (resp. Theorem~\ref{Thm:delocheightsupercrit}). The proof relies on the observation that, by using the specific structure of the annealed Villain potential (resp. the annealed Gaussian potential), these spin systems (resp. height functions) can be written as a Villain model (resp. integer-valued Gaussian free field) with random conductances. The law of the conductances is not i.i.d., but one can show that it stochastically dominates an i.i.d. distribution (N.B. this argument goes back to
%\margin{C: I will add a remark on the fact we have stoch Domination BOTH WAYS by i.i.d disorder !! We only use the one Sellke does not use.}
to Sellke~\cite{sellke2024localization} who recently used it to study the localisation/delocalisation properties of a general class of \emph{real-valued} height functions). 
This observation can be combined with the Ginibre inequality 
(resp. the monotonicity of the variance in the conductances) 
and Theorem~\ref{proofKTsitegeneral} (resp. Theorem~\ref{Thm:delocheightsupercrit}) to deduce Theorem~\ref{Th:annealedVillain} 
(resp. Theorem~\ref{Thm:delocheightanealedgaussgeneralgraph}).
Interestingly, our proof of delocalisation for integer-valued height function with arbitrary annealed Gaussian potentials (Theorem~\ref{Th:annealedVillain}) relies on the delocalisation of {\em both} the $\Z$-GFF and the $\Z$-XY height function models ! (Even though the later is itself an annealed Gaussian).
%\margin{C: I added this comment as I had missed/forgotten this nice subtlelty! 
%D: It is not completely accurate though? We need delocalization of this mixed model, 
%but the corresponding potential will be convex and satisfy the absolute value property (the first suffices for delocalization, 
%of which we can thus give an easy proof, and the second is needed for logarithmic deloc in Piet's work). }

\subsection{Open questions}

We mention here some open questions left by this work.

\medskip

\begin{OQ}\label{OQ.1}
Prove or disprove the Wells' inequality for the Villain model, i.e., if we let
\begin{equation*}
    \nu_{L , \beta}^{\mathrm{Vil}} \left( \left\{  \omega \right\} \right) := \frac{Z^{\mathrm{Vil}}_{L , \beta, \omega}}{\sum_{\omega' \in \{0 , 1 \}^{E\left(\Lambda_L\right)}} Z^{\mathrm{Vil}}_{L , \beta, \omega'}}~~ \mbox{then}~~ \E_{ \nu_{L , \beta}^{\mathrm{Vil}}} \left[ \left\langle \cos \left( \theta_0 - \theta_x \right) \right\rangle_{L , \beta , \omega}^{\mathrm{Vil}} \right] \geq \left\langle \cos \left( \theta_0 - \theta_x \right) \right\rangle_{L , \beta/2}^{\mathrm{Vil}}.
\end{equation*}
\end{OQ}

A second interesting open question would be to show that if we select an inverse temperature slightly above the critical inverse temperatures for the BKT (resp. roughening) phase transition, then it is possible to find a probability $p \in (0 , 1)$ (but close to $1$) so that the spin systems (resp. the height functions) are still in the BKT phase (resp. delocalised phase) when they are placed on a percolation of probability $p$.

\medskip

\begin{OQ}
Prove that, for all $\beta > \beta_{KT}^{\mathrm{Vil}}$, there exists $p := p(\beta) < 1$ such that the BKT phase transition for the Villain model still holds on a supercritical percolation cluster of probability $p$.

Show the same result for the XY model (see also \cite[Open Question 4]{DG25}), for the XY height function and the integer-valued Gaussian free field.
\end{OQ}

More generally, the following question can be raised.

\begin{OQ}
Show that the function which associates to a probability $p$ the value of the critical inverse temperature for the BKT transition on a percolation of probability $p$, i.e.,
\begin{equation} \label{def:critpointGinibre}
p \mapsto \inf \left\{ \beta \geq 0 \, : \, x \mapsto \E_p \left[ \left\langle \cos(\theta_0 - \theta_x) \right\rangle_{\beta, \omega}^{\mathrm{Vil}} \right] ~\mbox{decays polynomially fast}\right\},
\end{equation}
is continuous on the interval $\left( \frac 12 , 1 \right]$. Show the same result for the XY model, the XY height function and the integer-valued Gaussian free field.
\end{OQ}

We note that the Ginibre inequality implies that the function~\eqref{def:critpointGinibre} is decreasing, so it must (at least) be continuous almost everywhere.

\medskip

\ni
{\bf Acknowledgments.}  We wish to thank Piet Lammers for useful discussions.
%\margin{C: anyone else we should thank here ? Also add any grant you any have here.}  
C.G.  acknowledges support from the  ERC grant VORTEX 101043450, the Institut Universitaire de France (IUF) and the French ANR grant ANR-21-CE40-0003.

\section{Definitions and preliminaries} \label{Section2}

We introduce in this section a few definitions, pieces of notation and preliminary results which will be used in this article.

\subsection{Lattice, extended lattice and multigraph} \label{sectiongeneraldef}

We consider the lattice $\Zd$ with $d \geq 2$, denote by $E(\Zd)$ and $\vec{E}(\Zd)$ the set of undirected and directed edges of $\Zd$. Given a subset $\Lambda \subseteq \Zd$, we denote by $E(\Lambda)$ and $\vec{E}\left( \Lambda \right)$ the set of undirected and directed edges of $\Lambda$ (i.e., the edges of $\Zd$ whose endpoints are both in $\Lambda$). An undirected edge will frequently be denoted by $e = \{x , y\}$ and a directed edge will be denoted by $\vec{e} = (x , y)$ (with $x$ being the first endpoint and $y$ being the second endpoint). In the case of a directed edge $\vec{e} = (y,x)$, we denote by $-\vec{e} = (x , y)$, the edge with reversed orientation. We say that two vertices $x , y \in \Zd$ are neighbours, and denote it by $x\sim y$, if $\{ x , y\} \in E(\Zd)$.  We denote by $\partial \Lambda$ the inner boundary of $\Lambda$, i.e.,
\begin{equation*}
	\partial \Lambda := \left\{ x \in \Lambda \, : \, \exists y \in \Z^d \setminus \Lambda ~\mbox{with} ~ x \sim y \right\}.
\end{equation*}
We denote by $|\cdot|$ and $|\cdot |_1$ the Euclidean and $1$-norm on $\Zd$ (or $\Rd$), i.e., $|x| := (\sum_{i=1}^d |x_i|^2)^{1/2}$ and $| x |_1 = \sum_{i=1}^d |x_i|$ for any $x = (x_1 , \ldots, x_d) \in \Zd$.

For $L \in \N$, we write $\Lambda_L := \{ - L , \ldots, L\}^d$. The cardinality of the box $\Lambda_L$ is denoted by $\left| \Lambda_L \right| = (2L+1)^d$. 

We next introduce two generalisations of the lattice $\Zd$, which we call the extended lattice and the multigraph (see Figures~\ref{fig:figure1.3} and~\ref{fig:figure1.4}).

\begin{definition}[Extended lattice $\mathbb{Z}^d_n$] \label{def.Zdn}
Given an integer $n \in \mathbb{N}$, we define the extended lattice $\mathbb{Z}^d_n$ to be the graph $\Zd$ to which $(n-1)$ vertices are added on each edge. This graph is represented on Figure~\ref{fig:figure1.3}.
\end{definition}

\begin{definition}[The multigraph $\mathbb{Z}^d_{n\mathrm{-mult}}$] \label{def.multigraph}
Given an integer $n \in \mathbb{N}$, we define  the multigraph $\mathbb{Z}^d_{n\mathrm{-mult}}$ to be the graph $\Zd$ on which each edge is replaced by $n$ parallel edges. This graph is represented on Figure~\ref{fig:figure1.4}. 

For each pair of neighbouring vertices $x , y \in \Lambda_L$, we index the edges connecting $x$ and $y$ with integers of $\{1 , \ldots, n\}$ and sometimes refer to an edge of this graph (either oriented or unoriented) by giving its two endpoints together with the index of the edge. %(see Figure~\ref{fig:labellingedges}).
\end{definition}

%\begin{figure}
%\centering
%\includegraphics[scale=0.6]{labellingedges.png}
%\caption{The edges of the multigraph are labelled using the two endpoints together with an index of $\{1 , \ldots, n\}$.} \label{fig:labellingedges}
%\end{figure}

We extend all the definitions of this section to the extended lattice $\Zd_n$ and the multigraph $\mathbb{Z}^d_{n\mathrm{-mult}}$. In particular, we define a box of $\Zd_n$ to be a box of $\Zd$ to which $n$ vertices have been added on each edge, and a box of $\mathbb{Z}^d_{n\mathrm{-mult}}$ to be a box of $\Zd$ on which each edge has been replaced by $n$ parallel edges. Figures~\ref{fig:figure1.3} and~\ref{fig:figure1.4} show a box on~$\Z^2$ next to a box on $\Z^2_4$ and a box on $\mathbb{Z}^d_{3\mathrm{-mult}}$. We denote by $\Lambda^n_L \subseteq \Zd_n$ the box $\Lambda_L \subseteq \Zd$ to which $(n-1)$ vertices have been added on each edge, and by $\Lambda_L(n)$ the box $\Lambda_L \subseteq \Zd$ on which each edge has been replaced by $n$ parallel edges.

\begin{figure}
\centering
 \includegraphics[scale=0.4]{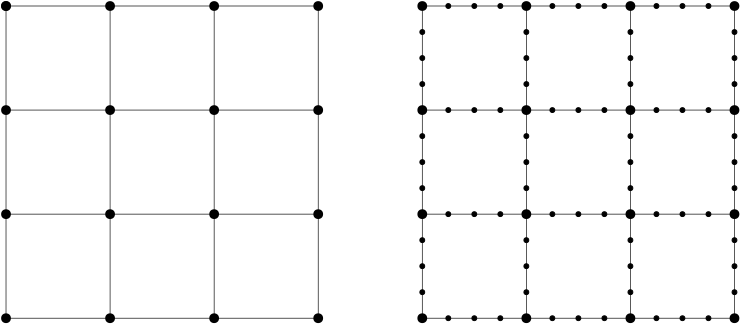}
 \caption{The graph $\Zd$, the extended graph $\Zd_n$ (with $n = 4$).} \label{fig:figure1.3}
\end{figure}

\begin{figure}
\centering
\includegraphics[scale=0.4]{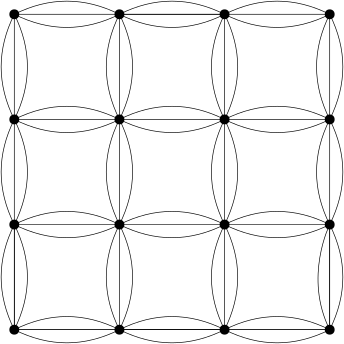}
\caption{The multigraph $\mathbb{Z}^d_{n\mathrm{-mult}}$ (with $n = 3$).} \label{fig:figure1.4}
\end{figure}

In the rest of this article, we will identify the vertices of $\Z^d$ with the corresponding vertices of the extended lattice $\Z^d_n$ and of the multigraph $\mathbb{Z}^d_{n\mathrm{-mult}}$. We emphasize that two vertices $x , y \in \mathbb{Z}^d$ which are neighbours on $\Zd$ are \emph{not} neighbour on $\Z^d_n$. We write $x \sim_n y$ to refer to pair of vertices which are neighbours in $\Zd_n$.

\subsection{Angles}

An angle is an element of $[0 , 2\pi)$ generally denoted by $\theta$. We identify the interval $[0 , 2\pi)$ with the torus $\R/(2\pi \Z)$. In particular, we may add two angles $\theta, \theta' \in [0 , 2\pi)$. Given a subset $\Lambda \subseteq \Zd$ (or $\Lambda \subseteq \Zd_n$), an angle configuration is a function $\theta : \Lambda \to [0, 2\pi]$ (i.e., $\theta \in [0,2\pi)^{\Lambda}$). Given a function $m \in \Z^{\Lambda}$ and an angle configuration $\theta \in [0,2\pi)^{\Lambda}$, we denote by 
\begin{equation*}
    m \cdot \theta = \sum_{x \in \Lambda} m_x \theta_x.
\end{equation*}

\subsection{Percolation} \label{subsecpercolation}

In this section, we introduce some definitions and notation about percolation configurations and stochastic domination which are used in the proofs of the main theorems.

 \begin{definition}[Percolation configuration on $\Zd$]
    A site percolation configuration on a subset $\Lambda \subseteq \Z^d$ is a function $r \in \{ 0 , 1\}^\Lambda$. We say that a site is open in the percolation configuration $r$ if $r_x = 1$ and closed if $r_x = 0$. We implicitly extend all the site percolation configurations defined on $\Zd$ (or on a subset of $\Zd$) to the extended graph $\Zd_n$ by setting $r_x = 1$ for $x \in \Zd_n \setminus \Zd$ (see Figure~\ref{fig:extrendedperco}).
    
    An edge percolation configuration on a subset $\Lambda \subseteq \Z^d$ is a function $\omega \in \{ 0 , 1\}^{E(\Lambda)}$. We extend this definition to the multigraph $\mathbb{Z}^d_{n\mathrm{-mult}}$ (N.B. in this case, any edge of the multigraph is allowed to be open or closed).
\end{definition}

\begin{remark}
We will use the notation $r$ to refer to site percolation configurations and $\omega$ to refer to edge percolation configurations. 
\end{remark}

\begin{definition}[Bernoulli site percolation measure] \label{DEf.bernoulliperc}
Given a set $\Lambda \subseteq \Zd$ and a probability $p \in [0 , 1]$, we denote by $\P_{\Lambda, p}$ (resp. $\P_{\Lambda, p}^{\mathrm{site}}$) the i.i.d. edge (resp. site) percolation measure of probability $p$ on the set of percolation configurations $\{ 0 , 1 \}^{E(\Lambda)}$ (resp. $\{ 0 , 1 \}^{\Lambda}$). In the case when $\Lambda = \Zd$, we write $\P_{p} := \P_{\Zd , p}$ (resp. $\P_{p}^{\mathrm{site}} := \P_{\Zd , p}^{\mathrm{site}}$).
\end{definition}

We next introduce the notion of stochastic domination, and will make use of the following notation: given a subset $\Lambda \subseteq \Zd$ and a probability measure $\nu$ on $\{ 0 , 1 \}^{\Lambda}$, we denote by $\E_\nu$ the expectation with respect to $\nu$. We simply write $\E_p$ (resp. $\E_p^{\mathrm{site}}$) instead of $\E_{\P_{\Lambda,p}}$ or $\E_{\P_p}$ (resp. $\E_{\P_{\Lambda,p}^{\mathrm{site}}}$ or $\E_{\P_p^{\mathrm{site}}}$).

\begin{definition}[Partial order, increasing functions and stochastic domination]
We introduce the three following definitions:
\begin{itemize}
\item[(i)] \textit{Partial order:} Given a subset $\Lambda \subseteq \Zd$, we define a partial order on the space $\{ 0 , 1 \}^{\Lambda}$ by writing, for any pair $r , r' \in \{ 0 , 1 \}^{\Lambda}$, $r \preceq r'$ if and only if, for all $x \in \Lambda, \, r_x \leq r_x'.$
\item[(ii)] \textit{Increasing function:} A function $f : \{ 0 , 1 \}^\Lambda \to \R$ is called increasing if $f(r) \leq f(r')$ for all pairs of configurations $r , r' \in \{ 0 , 1 \}^{\Lambda}$ with $r \preceq r'$.
\item[(iii)] \textit{Stochastic domination:} Given two probability measures $\nu , \nu'$ on $\{ 0 , 1 \}^{\Lambda}$, we say that $\nu$ stochastically dominates $\nu'$, and denote it by $\nu \preceq \nu'$, if $\E_\nu \left[  f \right] \leq \E_{\nu'} \left[  f \right]$ for any increasing function $f : \{ 0 , 1\}^{\Lambda} \to [0, \infty)$.
\end{itemize}
\end{definition}

\begin{figure}
\centering
\includegraphics[scale=0.4]{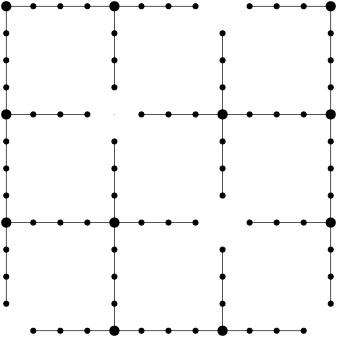}
\caption{A site percolation configuration on the extended lattice.} \label{fig:extrendedperco}
\end{figure}

\begin{remark} \label{rem:remark4}
    Let us make a few remarks about the previous definition:
\begin{itemize}
    \item For any $p , p' \in (0,1)$ with $p \leq p'$, the measure $\P_p^{\mathrm{site}}$ is stochastically dominated by the measure $\P_{p'}^{\mathrm{site}}$.
    \item We may extend all these definitions and properties to edge percolation instead of site percolation by replacing the set $\Lambda$ by the edges of $\Lambda$.
\end{itemize} 
\end{remark}

The following lemma provides a convenient criterion to prove stochastic domination. The statement and proof of the result can be found in~\cite[Lemma 1.1]{liggett1997domination} or in~\cite[Lemma 1.5]{duminil2019lectures} (specified to the case where the measure $\nu$ is the i.i.d. Bernoulli percolation measure $\P_{\Lambda,p}$ in the latter case). We say that a probability measure $\nu$ on $\{ 0,1\}^\Lambda$ is strictly positive if $\nu( \{ r \}) > 0$ for any $r \in \{0,1\}^\Lambda$.

\begin{lemma}[Lemma 1.1 of~\cite{liggett1997domination} or Lemma 1.5 of~\cite{duminil2019lectures}] \label{Lemma1.5}
        Let $\Lambda \subseteq \Zd$ be a finite set and let $\nu$ be a strictly positive probability measure on $\left\{ 0 , 1 \right\}^{\Lambda}$. We assume that for any vertex $x \in \Lambda$ and any configuration $r_1 \in \Lambda \setminus \{ x \}$, one has the inequality
        \begin{equation*}
            \nu \left( r_x = 1 \, | \, r = r_1 ~ \mbox{in} ~ \Lambda \setminus \{ x \} \right) \leq p.
        \end{equation*}
        Then $\P_{\Lambda,p}^{\mathrm{site}}$ stochastically dominates $\nu$.
\end{lemma}

We will need the following key result which is an important particular case of Theorem 0.0 in \cite{liggett1997domination}.
%\margin{C: The number of their main theorem is $0.0$ ? YES. apparently I keep asking the same question, already back in 2023 $:)$}
For any $k\geq 1$, we will say that a field of random variables $\{X_i\}_{i\in \Lambda}$ indexed by a subset $\Lambda \subseteq \Z^d$ is \textbf{$k$-dependent} if and only if for any subsets $A,B \subset \Lambda$ s.t. $\mathrm{dist}(A,B)>k$, the $\sigma$-algebras $\sigma(X_i, i\in A)$ and $\sigma(X_j, j\in B)$ are independent (where $\mathrm{dist}$ is the graph distance on $\Z^d$).
 
\begin{theorem}[Theorem 0.0 from \cite{liggett1997domination}]\label{th.Liggett}
For any $d\geq 1$, any $k\geq 1$ and any $p\in (0,1)$, there exists 
$\bar p=\bar p(d,k, p) < 1$ such that the following hold. For any subset $\Lambda \subseteq \Zd$, if $\{X_i\}_{i\in \Lambda}$ is a $k$-dependent field of Bernoulli variables satisfying $\Pb{X_i=1}\geq \bar p$, for each $ i \in \Lambda$, then if $\{Y_i\}_{i\in \Lambda}$ is an i.i.d field of Bernoulli variables with parameter $p$, one has 
\begin{align*}\label{}
\{X_i\}_{i\in \Lambda}  \text{ Stoch. dominates } \{Y_i\}_{i\in \Lambda}\,.
\end{align*}
\end{theorem}

\subsection{The FKG inequality}

% In this section, we state the standard FKG inequality for independent random variables. 
In the definition below, we say that a function $f: \R^n \to [0 , \infty)$ is increasing if it is increasing in each of its variables.

\begin{proposition}[Fortuin-Kasteleyn-Ginibre inequality] \label{prop.FKGinequality}
For $n \in \N^*$, let $J_1 , \ldots, J_n$ be independent random variables, and let $f , g : \R^n \to [0 , \infty)$ 
be increasing functions. One has the inequality:
\begin{equation} \label{eq:ineqFKG}
    \mathbf{E} \left[ f(J_1 , \ldots, J_n ) g(J_1 , \ldots, J_n ) \right] \geq  \mathbf{E} \left[ f(J_1 , \ldots, J_n ) \right]  \mathbf{E} \left[ g(J_1 , \ldots, J_n ) \right].
\end{equation}
\end{proposition}

\section{The Villain models on a percolation cluster} \label{sectiondef.XY}

This section is devoted to the proof of Theorem~\ref{proofKTsitegeneral} and Theorem~\ref{Th:annealedVillain} on the existence of phase transitions for the Villain model on a percolation cluster and the annealed Villain model. 

Regarding the proof of Theorem~\ref{proofKTsitegeneral}, 
we will only write the details of the first part of the argument 
(i.e., we will only prove the result in the case of a high-density percolation), 
as the second part of the proof 
(i.e., the renormalisation step which extends the result from the high-density percolation to any supercritical percolation) 
is almost identical to the proof of~\cite[Section 5.2]{DG25} and is similar to the proof of Section~\ref{sec:renoarmargumentforheight} (written for the height functions).

We also mention that we will prove the existence of phase transitions for the Villain model on a high-density \emph{site} percolation 
(instead of edge percolation as stated in Theorem~\ref{proofKTsitegeneral}). 
This is due to a technical reason: it is easier to implement the renormalisation argument once the result is established for a high-density site percolation. 
We mention that it would be easy to deduce the result for a high-density edge percolation from the proof written below 
(by either adapting the proof or by using Theorem~\ref{th.Liggett} to stochastically dominate the site percolation by an edge percolation).

\subsection{Preliminaries on the XY and Villain models} 

In this section, we introduce the general versions of the XY and Villain models used in this article, and collect some important properties of these models which are used in the proofs of Theorem~\ref{proofKTsitegeneral} in Sections~\ref{Section5} and~\ref{section5.2} below.

\subsubsection{The XY and Villain models} \label{sec:sec3spinsystems}

We first introduce some general versions of the XY and Villain models and state the Ginibre inequality in Theorem~\ref{theoremGinibreineq} below.

\begin{definition}[XY and Villain models with general conductances] \label{generalXY}
Let $n \in \N$ be an integer, $\Lambda \subseteq \Zd_n$ be a finite set, and $(J_{xy})_{\{x , y \} \in E(\Lambda)} \subseteq [0, \infty)^{E(\Lambda)} $ be a collection of non-negative conductances. We define the XY and Villain models to be the probability distributions on the space $[0 , 2\pi)^{\Lambda}$ given by the formulae
    \begin{equation} \label{eq:17131905}
        \mu_{\Lambda, J}^{\mathrm{XY}} (d\theta) := \frac{1}{Z_{\Lambda, J}^{\mathrm{XY}}}\exp \left( \sum_{x \sim y} J_{xy} \cos ( \theta_x - \theta_y ) \right) \prod_{x \in \Lambda} d \theta_x
    \end{equation}
    and
\begin{equation} \label{def:defVillaingeneralconduc}
    \mu_{\Lambda, J}^{\mathrm{Vil}} (d\theta) := \frac{1}{Z_{\Lambda, J}^{\mathrm{Vil}}} \prod_{x \sim y} v_{J_{xy}} \left( \theta_x - \theta_y \right) \prod_{x \in \Lambda} d \theta_x,
\end{equation}
where $Z_{\Lambda, J}^{\mathrm{XY}}$ and $Z_{\Lambda, J}^{\mathrm{Vil}}$ are the normalizing constants (N.B. Recall the definition of the heat kernel $v_J(\cdot)$ (or $v_\beta(\cdot)$) from~\eqref{def:introdefVillain}).
%\margin{C: I added a pointer to $v_J$ as it is defined much earlier and I was lost for a moment $:)$}
 We denote by $\left\langle \cdot \right\rangle_{\Lambda, J}^{\mathrm{XY}}$ and $\left\langle \cdot \right\rangle_{\Lambda, J}^{\mathrm{Vil}}$ the expectations with respect to the measures~$\mu_{\Lambda, J}^{\mathrm{XY}}$ and~$\mu_{\Lambda, J}^{\mathrm{Vil}}$.
\end{definition}

\begin{remark} \label{rem:remarkJ=infty}
We may take the limit $J_{xy} \to \infty$ for an edge $\{x , y\} \in E(\Lambda)$. This has the effect of contraining the angles $\theta_x$ and $\theta_y$ to be equal.
\end{remark}

We next introduce three spin systems which play a key role in the proof of Theorem~\ref{proofKTsitegeneral} below (see Figures~\ref{fig:3spinsystems}).

\begin{definition}[Three spin systems on the extended lattice] \label{def:threespinsystems}
We fix an integer $L \in \N$, an integer $n \in \N^*$, three inverse temperatures $\beta, \beta_1, \beta_2 \geq 0$ and introduce three spin systems on the extended lattice:
\begin{itemize}
    \item The first one is the XY model on the extended lattice $\Zd_n$ with inverse temperature $n \beta$, i.e.,
    \begin{equation} \label{eq:model1XY}
        \mu_{\Lambda_L^n, n,  \beta}^{\mathrm{XY}}(d \theta)  := \frac{1}{Z_{\Lambda_L^n, n,  \beta}^{\mathrm{XY}}} \exp \left( n \beta \sum_{x \sim_n y}  \cos ( \theta_x - \theta_y ) \right) \prod_{x \in \Lambda_L^n} d \theta_x.
    \end{equation}
    We denote by $\left\langle \cdot \right\rangle_{L, n  ,\beta}^{\mathrm{XY}}$ the expectation with respect to this measure.
    \item The second one is an XY model on the extended lattice $\Zd_n$ with heterogenous temperatures defined as follows. For any $n \geq 2$ and any pair of inverse temperatures $\beta_1 , \beta_2 > 0$, we set
\begin{equation} \label{eq:model2XY}
        \mu_{\Lambda_L^n, n, \beta_1 , \beta_2}^{\mathrm{XY}}(d\theta) := \frac{1}{Z_{\Lambda_L^n, n, \beta_1,  \beta_2 }^{\mathrm{XY}}}\exp \left( \beta_1 \sum_{\substack{x \sim_n y \\ x \in \Zd }} \cos (\theta_x - \theta_y) + n \beta_2 \sum_{\substack{x \sim_n y \\ x , y \notin \Zd}} \cos (\theta_x - \theta_y) \right) \prod_{x \in \Lambda_L^n} d \theta_x.
\end{equation}
 We denote by $\left\langle \cdot \right\rangle_{L, n  ,\beta_1 , \beta_2}^{\mathrm{XY}}$ the expectation with respect to this measure.
    \item The third one is a spin system defined on the extended graph $\Zd_{3}$ where the interaction on the edges connected to $\Zd$ is the one of an XY model with inverse temperature $\beta_1$ and the interaction along the other edges is the one of a Villain model with inverse temperature $\beta_2$, i.e.,
    \begin{equation} \label{eq:model3XY}
        \mu_{\Lambda_L^3, \beta_1 , \beta_2}^{\mathrm{XY/Vil}}(d\theta) := \frac{1}{Z_{\Lambda_L^3, \beta_1 , \beta_2 }^{\mathrm{XY}}}\exp \left( \beta_1 \sum_{\substack{x \sim_3 y \\ \{ x , y \} \cap \Zd \neq \emptyset}} \cos (\theta_x - \theta_y) \right) \prod_{\substack{x \sim_3 y \\ x , y \notin \Zd}} v_{\beta_2}(\theta_x - \theta_y) \prod_{x \in \Lambda_L^3} d \theta_x.
    \end{equation}
    We refer to this model as \textbf{the XY/Villain model} and denote by $\left\langle \cdot \right\rangle_{L  ,\beta_1, \beta_2}^{\mathrm{XY/Vil}}$ the expectation with respect to this measure.
\end{itemize}
\end{definition}

\begin{figure}
\centering
\includegraphics[scale=0.65]{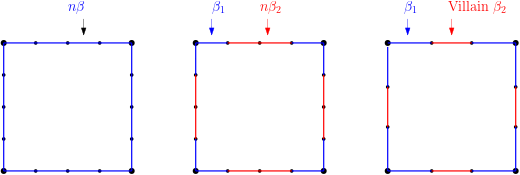}
\caption{The three spin systems of Section~\ref{sec:sec3spinsystems}: the XY models on the extended lattice $\Zd_n$ with $n = 4$ (on the left) and the XY/Villain model (on the right).} \label{fig:3spinsystems}
\end{figure}

\subsubsection{The Ginibre inequality}

We next state the Ginibre correlation inequality~\cite{ginibre1970general} which implies the monotonicity of the two-point function in the conductances. We specifically state two results from~\cite{ginibre1970general}, the first inequality~\eqref{ineq:Ginibre1} is an important ingredient in the proof of the properties~\eqref{eq:Ginibreineq}, and will also be used to establish a (generalised version of) Wells' inequality in Section~\ref{App.A}.

\begin{theorem}[Ginibre inequality~\cite{ginibre1970general}] \label{theoremGinibreineq}
 One has the following inequalities:
 \begin{itemize}
  \item For any pair of integers $k, l \in \mathbb{N}$, any collection $m_1 , \ldots, m_k \in \mathbb{Z}^l$ and any sequence of plus and minus signs,
 \begin{equation} \label{ineq:Ginibre1}
     \int_{[0, 2\pi]^l \times [0, 2\pi]^l} \prod_{i = 1}^k \left( \cos ( m_i \cdot \theta) \pm  \cos ( m_i \cdot \theta') \right) \prod_{j =1}^l d\theta_j d \theta'_j \geq 0.
 \end{equation}
 \item For any pair of vertices $x , y \in \Lambda$,
\begin{equation} \label{eq:Ginibreineq}
    \left\langle \cos(\theta_x - \theta_y) \right\rangle_{\Lambda, J}^{\mathrm{XY}} \geq 0 \hspace{5mm} \mbox{and} \hspace{5mm} J = (J_{uv}) \mapsto \left\langle \cos(\theta_x - \theta_y ) \right\rangle_{\Lambda, J}^{\mathrm{XY}} ~\mbox{is increasing}.
\end{equation}
\end{itemize}
\end{theorem}

\begin{remark}
Let us make three remarks about these results:
\begin{itemize}
\item The properties~\eqref{eq:Ginibreineq} hold when the XY model is defined on a general finite graph.
\item We may deduce from the inequality~\eqref{eq:Ginibreineq} the domain monotonicity of the two-point function: if we consider two finite sets $\Lambda' \subseteq \Lambda \subseteq \Zd_n$ and let $J$ be a collection of conductances defined on the edges of $\Lambda$, then we have the inequality (N.B. on the left-hand side, we denote by $J'$ the restriction of $J$ to the edges of $\Lambda'$)
\begin{equation*}
    \left\langle \cos(\theta_x - \theta_y) \right\rangle_{\Lambda', J'}^{\mathrm{XY}}  \leq \left\langle \cos(\theta_x - \theta_y) \right\rangle_{\Lambda, J}^{\mathrm{XY}}.
\end{equation*}
\item We will see in Proposition~\ref{prop:prop3.6} below that the same monotonicity property holds for the Villain model (and would hold for the XY/Villain model with general conductances).
\end{itemize}
\end{remark}

\subsubsection{Wells' inequality}

This section is devoted to the statement of Wells' inequality for the XY model. The inequality was originally established by Wells~\cite{Wellsthesis}, the proof of the result can be found in~\cite[Appendix]{bricmont1981periodic} in the case of the spin-1/2 Ising model and in~\cite[Section 3]{DG25} in the case of the XY model (N.B. the proof there is an adaptation of the ones of~\cite{Wellsthesis, bricmont1981periodic}). In all this section, we fix an integer $n \in \N$ and first introduce a definition for the XY and Villain models with general conductances on a site percolation configuration.

\begin{definition}[Disordered XY and Villain models] 
Let $\Lambda \subseteq \Zd_n$ be a finite set, and let $(J_{xy})_{\{x , y \} \in E(\Lambda)} \in [0, \infty)^{E(\Lambda)} $ be a collection of non-negative conductances. Given a percolation configuration $r \in \{0,1\}^{\Lambda \cap \Zd}$, we define the disordered XY and Villain models
\begin{equation*}
    \mu_{\Lambda, J,r}^{\mathrm{XY}} (d\theta) := \frac{1}{Z_{\Lambda, J,r}^{\mathrm{XY}}}\exp \left( \sum_{x \sim y} r_x r_y J_{xy} \cos ( \theta_x - \theta_y ) \right) \prod_{x \in \Lambda} d \theta_x,
    \end{equation*}
and
\begin{equation*}
    \mu_{\Lambda, J,r}^{\mathrm{Vil}} (d\theta) := \frac{1}{Z_{\Lambda, J,r}^{\mathrm{Vil}}} \prod_{x \sim y} v_{r_x r_y J_{xy}} \left( \theta_x - \theta_y \right) \prod_{x \in \Lambda} d \theta_x,
\end{equation*}
where $Z_{\Lambda, J, r}^{\mathrm{XY}}$ and $Z_{\Lambda, J,r}^{\mathrm{Vil}}$ are the normalizing constants. We denote by $\left\langle \cdot \right\rangle_{\Lambda, J,r}^{\mathrm{XY}}$ and $\left\langle \cdot \right\rangle_{\Lambda, J,r}^{\mathrm{Vil}}$ the expectations with respect to these measures (N.B. We use here and below the convention that the percolation configurations are extended by $1$ for the vertices of $\Zd_n \setminus \Zd$).
\end{definition}

\begin{remark}
    We may deduce from the Ginibre inequality (Theorem~\ref{theoremGinibreineq}) the following monotonicity properties:
    \begin{itemize}
        \item For any $x , y \in \Lambda$, the function $r \in \{0,1\}^{\Lambda \cap \Zd} \mapsto \left\langle \cos \left( \theta_x - \theta_y \right) \right\rangle_{\Lambda, J , r}^{\mathrm{XY}}$ is increasing (and we will see in Proposition~\ref{prop:prop3.6} below that the same property holds for the Villain model and would hold for the XY/Villain model with general conductances).
        \item If $\Lambda , \Lambda'$ are two sets such that $\Lambda' \subseteq \Lambda \subseteq \Zd_n$, $J$ is a collection of conductances defined on the edges of $\Lambda$ and $r \in \{0,1\}^{\Lambda \cap \Zd}$ is a site percolation configuration, then we have the inequality (N.B. On the left-hand side, we denote by $J'$ and $r'$ the restrictions of $J$ and $r$ to the edges and sites of $\Lambda'$)
        \begin{equation*}
         \left\langle \cos(\theta_x - \theta_y) \right\rangle_{\Lambda', J',r'}^{\mathrm{XY}}  \leq \left\langle \cos(\theta_x - \theta_y) \right\rangle_{\Lambda, J, r}^{\mathrm{XY}}.
         \end{equation*}
         We will see in Proposition~\ref{prop:prop3.6} below that the same inequality holds for the Villain model (and would hold for the XY/Villain model with general conductances).
    \end{itemize}
\end{remark}

We next proceed by introducing a probability distribution on the space of percolation configurations which we call the Wells' disorder (following~\cite{Wellsthesis}).

\begin{definition}[Wells' disorder] \label{def.finite-volHamiltonian}
Let $\Lambda \subseteq \Zd_n$ be a finite set, and let $(J_{xy})_{ \{ x, y \} \in E(\Lambda)} \in [0, \infty)^{E(\Lambda)} $ be a collection of non-negative conductances. We define the Wells' disorder to be the probability distribution on the space of site percolation configurations $\{0 , 1\}^{\Lambda \cap \Zd}$ given by the identity: for any $r \in \{0 , 1\}^{\Lambda \cap \Zd}$
    \begin{equation} \label{eq:17141905}
        \nu_{\Lambda, J} (\{ r \}) := \frac{Z_{\Lambda, J,r}^{\mathrm{XY}}}{\bar Z_{\Lambda, J}^{\mathrm{XY}}} ~~\mbox{where} ~~ \bar Z_{\Lambda, J}^{\mathrm{XY}} = \sum_{ r \in \{0 , 1\}^{\Lambda\cap \Zd}} Z_{\Lambda, J,r}^{\mathrm{XY}}.
    \end{equation}
When the set $\Lambda$ is the box $\Lambda_L^n,$ we denote this measure by $\nu_{L, J}$.
\end{definition}

Wells' inequality asserts that the expectation of the two-point function of a disordered XY model, when the law of the disorder is the one of Definition~\ref{def.finite-volHamiltonian}, is bounded from below by the one of an XY model without disorder at a higher temperature.

\begin{theorem}[Wells' inequality for the XY model~\cite{Wellsthesis, bricmont1981periodic}] \label{prop.Wells}
    Let $\Lambda \subseteq \Zd_n$ be a finite set, then for any collection of nonnegative conductances $(J_{xy})_{ \{ x, y \} \in E(\Lambda)}$ and any pair of vertices $x , y \in \Lambda$,
    \begin{equation} \label{eq:Wellsineq}
         \left\langle \cos ( \theta_x - \theta_y ) \right\rangle_{\Lambda, J/4}^{\mathrm{XY}} \leq \E_{\nu_{\Lambda , J}} \left[ \left\langle \cos ( \theta_x - \theta_y ) \right\rangle_{\Lambda, J, r}^{\mathrm{XY}} \right].
    \end{equation}
\end{theorem}
We do not give the proof of this theorem here but refer to~\cite{Wellsthesis, bricmont1981periodic,DG25} for the details of the argument, or to Section~\ref{App.A} where the result is proved for a generalised version of the XY model and then applied to the integer-valued height functions (N.B. In all these references, the strategy of the proof is the same).

\iffalse
\subsubsection{The Villain model}

\begin{definition}[The Villain model with general conductances] \label{generalXY}
Let $\Lambda \subseteq \Zd_n$ be a finite set, and let $\left\{ J_{xy} \, : \{x , y \} \in E(\Lambda) \right\} \subseteq [0, \infty)^{E(\Lambda)} $ be a collection of non-negative conductances. We define the finite-volume Gibbs measure on the space $[0 , 2\pi)^{\Lambda}$ according to the formula
    \begin{equation} \label{eq:17131905}
        \mu_{\Lambda, J}^{\mathrm{Vil}} (d\theta) := \frac{1}{Z_{\Lambda, J}} \prod_{x \sim y} v_{J_{xy}} \left(  \theta_x - \theta_y \right) \prod_{x \in \Lambda} d \theta_x,
    \end{equation}
where $Z_{\Lambda, J}$ is the normalizing constant. We denote by $\left\langle \cdot \right\rangle_{\mu_{\Lambda, J}}$ the expectation with respect to the measure~$\mu_{\Lambda, J}$.
\end{definition}

\begin{remark}
***This is the heat kernel on the circle (i.e., the law of the Brownian motion on the circle); this implies that $2$ Villain models one after the other gives a Villain model with smaller conductance***
\end{remark}
\fi

\subsubsection{The Villain model as a metric graph limit of the XY model} \label{sec:sectionwith3models}

In this section, we show that the Villain and XY/Villain models can be obtained as the limit of XY models defined on the extended lattice (see Proposition~\ref{prop:prop3.6}). A useful implication of this observation is that properties proved for the XY model, such as the Ginibre inequality, can be extended to the Villain and XY/Villain models via an approximation procedure.

We will state the result in the level of generality which is needed for the proof of Theorem~\ref{proofKTsitegeneral} and first introduce a disordered version of two of the three spin systems of Definition~\ref{def:threespinsystems}.

\begin{definition}[XY and XY/Villain models on a site percolation configuration]
We fix an integer $L \in \N$, an integer $n \geq 2$, a pair of inverse temperatures $\beta_1 , \beta_2 > 0$, a percolation configuration $r \in \{0,1\}^{\Lambda_L}$, and introduce:
\begin{itemize}
    \item The XY model on the extended lattice $\Zd_n$ on a percolation configuration:
\begin{multline*}
    \mu_{\Lambda_L^n, n, \beta_1 , \beta_2, r}^{\mathrm{XY}}(d\theta) := \frac{1}{Z_{\Lambda_L^n, n, \beta_1,  \beta_2, r }^{\mathrm{XY}}}\exp \left( \beta_1  \sum_{\substack{x \sim_n y \\ x \in  \Zd}} r_x \cos (\theta_x - \theta_y)\right)\\  \times \exp \left( n \beta_2 \sum_{\substack{x \sim_n y \\ x , y \notin \Zd}} \cos (\theta_x - \theta_y) \right) \prod_{x \in \Lambda_L^n} d \theta_x.
\end{multline*}
    \item The XY/Villain model on a percolation configuration:
    \begin{equation*}
        \mu_{\Lambda_L^3, \beta_1 , \beta_2, r}^{\mathrm{XY/Vil}}(d\theta) := \frac{1}{Z_{\Lambda_L^3, \beta_1 , \beta_2, r }^{\mathrm{XY/Vil}}}\exp \left( \beta_1  \sum_{\substack{x \sim_3 y \\ x \in \Zd}} r_x \cos (\theta_x - \theta_y) \right) \prod_{\substack{x \sim_3 y \\ x , y \notin \Zd}} v_{\beta_2}(\theta_x - \theta_y) \prod_{x \in \Lambda_L^3} d \theta_x.
    \end{equation*}
\end{itemize}
\end{definition}

The next proposition asserts that the Villain and XY/Villain model on a percolation configuration can be obtained as the limit of XY models defined on the extended lattice. We refer to Figure~\ref{fig:convergencetoVillain} for an illustration of the result and to~\cite[Section 4]{NW} and~\cite[Appendix A.2]{AHPS} for a proof.

\begin{figure}
\centering
\includegraphics[scale=0.4]{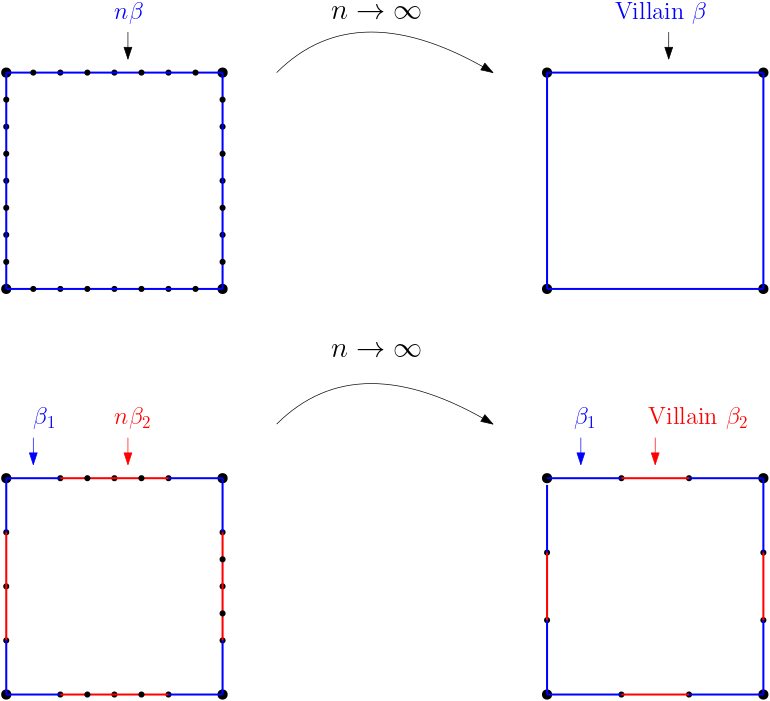}
\caption{The Villain and XY/Villain models as a limit of XY models.} \label{fig:convergencetoVillain}
\end{figure}

\begin{proposition}[\cite{NW, AHPS}] \label{prop:prop3.6}
    For any integer $L \in \N$, any triplet of inverse temperatures $\beta, \beta_1 , \beta_2 \in (0 , \infty)$, any percolation configuration $r \in \{0 , 1 \}^{\Lambda_L}$, and any pair of vertices $x , y \in \Lambda_L$, one has the convergences: 
    \begin{equation*}
        \lim_{n \to \infty} \left\langle \cos \left( \theta_x - \theta_y \right)  \right\rangle_{L, n, \beta, r}^{\mathrm{XY}} =  \left\langle \cos \left( \theta_x - \theta_y \right)  \right\rangle_{L, \beta,r}^{\mathrm{Vil}}
    \end{equation*}
    and
    \begin{equation*}
        \lim_{n \to \infty} \left\langle \cos \left( \theta_x - \theta_y \right)  \right\rangle_{L, n, \beta_1 , \beta_2, r}^{\mathrm{XY}} =  \left\langle \cos \left( \theta_x - \theta_y \right)  \right\rangle_{L, \beta_1 , \beta_2, r}^{\mathrm{XY/Vil}}.
    \end{equation*}
\end{proposition}

\begin{remark}
    We make two remarks about this proposition:
    \begin{itemize}
        \item It can be proved that the restriction to the set $\Lambda_L$ of the measures $\mu_{\Lambda_L^n, n,  \beta,r}^{\mathrm{XY}}$ (resp. $\mu_{\Lambda_L^n, n, \beta_1 , \beta_2,r}^{\mathrm{XY}}$) converges toward $\mu_{\Lambda_L , \beta,r}^{\mathrm{Vil}}$ (resp. $\mu_{\Lambda_L^3, \beta_1 , \beta_2,r}^{\mathrm{XY/Vil}}$).
        \item The result can be extended to more general graphs and to more general conductances.
    \end{itemize}
\end{remark}
Proposition~\ref{prop:prop3.6} can then be combined with the Ginibre inequality (Theorem~\ref{theoremGinibreineq}) to extend the range of application of this inequality from the XY model to the Villain model and the XY/Villain model defined above. In particular, we obtain the following corollary.

\begin{corollary} \label{corollary:ginibreforVillain}
The two-point functions $\left\langle \cos \left( \theta_x - \theta_y \right)  \right\rangle_{L, \beta, r}^{\mathrm{XY}}$,
$\left\langle \cos \left( \theta_x - \theta_y \right)  \right\rangle_{L, \beta, r}^{\mathrm{Vil}}$ and $\left\langle \cos \left( \theta_x - \theta_y \right) \right\rangle_{L, \beta_1 , \beta_2,r}^{\mathrm{XY/Vil}}$ are increasing functions of:
\begin{itemize}
    \item the sidelength $L \in \N$,
    \item the inverse temperatures $\beta, \beta_1, \beta_2 \in (0 , \infty)$,
    \item the percolation configuration $r \in \{0 , 1 \}^{\Lambda_L}$.
\end{itemize}
%For any integer $L \in \N$, any pair of vertices $x , y \in \Lambda_L$ and any percolation configuration $r \in \{0 , 1\}^{\Lambda_L}$,
%\begin{equation*}
%    \beta_1 \mapsto \left\langle \cos \left( \theta_x - \theta_y \right) \right\rangle_{L, \beta_1 , \beta_2,r}^{\mathrm{XY/Vil}}  ~\mbox{is increasing}
%\end{equation*}
%and 
%\begin{equation*}
%      \beta_2 \mapsto \left\langle \cos \left( \theta_x - \theta_y \right) \right\rangle_{L, \beta_1 , \beta_2,r}^{\mathrm{XY/Vil}} ~\mbox{is increasing}.
%\end{equation*}
\end{corollary}

\begin{remark} \label{rem:remark10}
Let us make three remarks about the previous statement:
\begin{itemize}
    \item Similar results hold in a (much more) general setting: they hold for spin systems valued in the circle~$\mathbb{S}^1$ defined on more general graphs where the weights are given by mixtures of the XY and Villain potentials.
    \item For a percolation configuration $r \in \{0,1\}^{\Zd}$, we denote by 
    $$\left\langle \cos \left( \theta_x - \theta_y \right)  \right\rangle_{\beta, r}^{\mathrm{XY}} := \lim_{L \to \infty} \left\langle \cos \left( \theta_x - \theta_y \right)  \right\rangle_{L , \beta, r}^{\mathrm{XY}}, ~ \left\langle \cos \left( \theta_x - \theta_y \right)  \right\rangle_{\beta, r}^{\mathrm{Vil}} := \lim_{L \to \infty} \left\langle \cos \left( \theta_x - \theta_y \right)  \right\rangle_{L, \beta, r}^{\mathrm{Vil}}$$ and 
    $$\left\langle \cos \left( \theta_x - \theta_y \right) \right\rangle_{ \beta_1 , \beta_2,r}^{\mathrm{XY/Vil}} := \lim_{L \to \infty} \left\langle \cos \left( \theta_x - \theta_y \right) \right\rangle_{L, \beta_1 , \beta_2,r}^{\mathrm{XY/Vil}}.$$
    In the case $r \equiv 1$, we omit this subscript from the notation.
    \item Taking the limit $\beta_1 \to \infty$, we have the convergence, for any $L \in \N$, $\beta_2 \geq 0$ and $r \in \{0,1\}^{\Lambda_L}$,
    \begin{equation*}
        \lim_{\beta_1 \to \infty} \left\langle \cos \left( \theta_x - \theta_y \right)  \right\rangle_{L, \beta_1 , \beta_2, r}^{\mathrm{XY/Vil}} = \left\langle \cos \left( \theta_x - \theta_y \right)  \right\rangle_{L, \beta_2, r}^{\mathrm{Vil}}.
    \end{equation*}
    Combining this observation with Corollary~\ref{corollary:ginibreforVillain}, we obtain the inequality, for any $L \in \N$, $\beta_2 \geq 0$ and $r \in \{0,1\}^{\Lambda_L}$,
\begin{equation} \label{ineq:beta1toinfty}
    \left\langle \cos \left( \theta_x - \theta_y \right)  \right\rangle_{L, \beta_1 , \beta_2, r}^{\mathrm{XY/Vil}} \leq \left\langle \cos \left( \theta_x - \theta_y \right)  \right\rangle_{L, \beta_2, r}^{\mathrm{Vil}}.
\end{equation}
\end{itemize}
\end{remark}

\subsubsection{Phase transitions for the XY/Villain model}

In this section, we collect two results regarding the phase transitions of the XY/Villain model introduced in~\eqref{eq:model3XY}. These results are usually stated for the XY model or the Villain model on the lattice $\Zd$. The proof in the two-dimensional case can be found in~\cite{van2023elementary} and in dimension $d \geq 3$ in~\cite{garban2022continuous}, and can be extended without major difficulties to the XY/Villain model. 
%\margin{\blue{Paul: Citation here} C: you mean we should cite more things ? I guess you are right. I am not sure what you had in mind ? I guess \cite{AHPS} could also be cited for this mixture ?}

\begin{proposition}[Phase transitions for the XY/Villain model~\cite{van2023elementary, garban2022continuous}] \label{prop.phasetransitionextended}
The following hold true:
\begin{itemize}
    \item In dimension $d = 2$, there exists an inverse temperature $\beta_{\mathrm{KT}}^{\mathrm{XY/Vil}} < \infty$ such that, for any $\beta_1 , \beta_2 \geq \beta_{\mathrm{KT}}^{\mathrm{XY/Vil}}$,
    \begin{equation*}
        \left\langle \cos(\theta_0 - \theta_x) \right\rangle_{\beta_1 , \beta_2}^{\mathrm{XY/Vil}} ~\mbox{decays polynomially fast as } |x| \to \infty. 
    \end{equation*}
    \item In dimension $d \geq 3$, there exists an inverse temperature $\beta_{ c}^{\mathrm{XY/Vil}}(d) < \infty$ such that, for any $\beta_1 , \beta_2 \geq \beta_{c}^{\mathrm{XY/Vil}}$,
    \begin{equation*}
        \left\langle \cos(\theta_0 - \theta_x) \right\rangle_{\beta_1 , \beta_2}^{\mathrm{XY/Vil}} ~\mbox{remains bounded away from } 0 \mbox{ as } |x| \to \infty.
    \end{equation*}
\end{itemize}
\end{proposition}

\begin{remark} \label{remark2.24}
The following lower bounds can be deduced from the arguments of~\cite{van2023elementary,garban2022continuous}: in dimension $d = 2$ and for $\beta_1 , \beta_2$ sufficiently large, there exists $c > 0$ such that
\begin{equation*}
    \langle \cos \left( \theta_0 - \theta_x \right) \rangle_{\beta_1 , \beta_2}^{\mathrm{XY/Vil}} \geq \frac{c}{|x|}.
\end{equation*}
In dimension $d \geq 3,$ there exists a constant $C := C(d) < \infty$ such that
    \begin{equation*}
    \langle \cos \left( \theta_0 - \theta_x \right) \rangle_{\beta_1 , \beta_2}^{\mathrm{XY/Vil}} \geq 1 - C \sqrt{ \frac{1}{\beta_1} + \frac{1}{\beta_2}}.
\end{equation*}
\end{remark}

\subsection{Stochastic domination of the Wells' disorder} \label{subsection4.1}
In this section, we introduce the Wells' disorder associated with the model~\eqref{eq:model2XY} 
and show that it is stochastically dominated from above by an i.i.d. Bernoulli percolation with sufficiently 
high probability. Specifically, given an integer $L \in \N$, 
an integer $n \in \N$ and two inverse temperatures $\beta_1 , \beta_2 \in (0,\infty)$, 
we define the Wells' disorder associated with the model~\eqref{eq:model2XY}: 
for any percolation configuration $r \in \{0,1\}^{\Lambda_L}$,
$$
\nu_{L , n , \beta_1 , \beta_2}(\{ r \}) := \frac{Z_{\Lambda_L^n, n, \beta_1,  \beta_2, r }}{\bar Z_{\Lambda_L^n, n, \beta_1,  \beta_2}} ~~\mbox{with}~~ \bar Z_{\Lambda_L^n, n, \beta_1,  \beta_2} = \sum_{r \in \{0,1\}^{\Lambda_L^n}} Z_{\Lambda_L, n, \beta_1,  \beta_2, r}.
$$
We note that this measure is obtained from the general Definition~\ref{def.finite-volHamiltonian} by considering $\Lambda = \Lambda_{L}^n$, $J_{xy} = \beta_1$ if $\{x , y\} \cap \Zd \neq \emptyset$ and $J_{xy} = n \beta_2$ otherwise. The main result of this section is the following stochastic domination for the Wells' disorder.

\begin{proposition}[Stochastic domination for the Wells' disorder]\label{prop:stochdomWellsVillain}
    Fix $\beta_1 \in (0 , \infty)$ and set $p_0(\beta_1) := \frac{1}{1 + \exp \left( - 2d \beta_1 \right)} \in (0 , 1)$. Then, for any integer $L \in \N$, any integer $n \in \N$ and any inverse temperature $\beta_2 \in (0 , \infty)$, one has the stochastic domination
    \begin{equation*}
        \nu_{L , n , \beta_1 , \beta_2} \preceq \P_{\Lambda_L, p_0(\beta_1)}^{\mathrm{site}}.
    \end{equation*}
\end{proposition}

\begin{remark}
The main feature of this result is that the probability $p_0$ does not depend on the integer $n$; 
this property will allow us to take the limit $n \to \infty$ and apply Proposition~\ref{prop:prop3.6} 
in Section~\ref{Section5} below. (N.B. this stochastic domination only uses Lemma \ref{Lemma1.5},
which is the very easy part of \cite{liggett1997domination}).
%\margin{C: added this as I always get confused.}
\end{remark}

\begin{proof}
    The proof is very similar to the one of~\cite[Section 4.1]{DG25}, and we present the details below for completeness.
    We first recall the notation for the partition function: for any $r \in \{ 0 , 1 \}^{\Lambda_L}$,
    \begin{equation} \label{def.Zlambdabetar}
        Z_{\Lambda_L^n, n, \beta_1,  \beta_2, r } := \int \exp \left( \beta_1 \sum_{\substack{x \sim_n y \\  x \in \Zd \cap \Lambda_L^n}} r_x \cos (\theta_x - \theta_y) + n \beta_2 \sum_{\substack{x \sim_n y \\ x , y \in \Lambda_L^n \setminus \Zd }} \cos (\theta_x - \theta_y) \right) \prod_{x \in \Lambda_L^n} d \theta_x.
    \end{equation}
   We first use Lemma~\ref{Lemma1.5} to reduce the problem to the following inequality: for any vertex $x \in \Lambda_L$ and any percolation configuration $r_1 \in \{0,1\}^{\Lambda_L \setminus \{ x \}}$,
    \begin{equation} \label{eqnunu'DC}
        \nu \left( r_x = 1 \, | \, r = r_1 ~ \mbox{in} ~ \Lambda_L \setminus \{ x \} \right) \leq p_0(\beta_1).
    \end{equation}
    To show the inequality~\eqref{eqnunu'DC}, let us fix $r_1 \in \{0,1\}^{\Lambda_L \setminus \{ x \}}$ and let $r_1^+, r_1^- \in \left\{ 0,1 \right\}^{\Lambda_L}$ be defined by $r_1^+ = r_1^- = r_1, $ in $\Lambda_L \setminus \{ x \}$ and $r_{1,x}^+ = 1$, $r_{1,x}^- = 0$. Using these definitions, we deduce that
    \begin{equation} \label{eq:1635}
        \nu \left( r_z = 1 \, | \, r = r_1 ~ \mbox{in} ~ \Lambda \setminus \{ x \} \right) = \frac{ Z_{\Lambda_L^n, \beta, r_1^+}}{ Z_{\Lambda_L^n, \beta, r_1^+} + Z_{\Lambda_L^n, \beta, r_1^-} }.
    \end{equation}
    Using the identity~\eqref{def.Zlambdabetar}, together with the inequality, for any angle configuration $\theta \in [0 , 2\pi)^{\Lambda_L^n}$,
    \begin{equation*}
        \exp \left( \beta_1 \sum_{\substack{y \sim_n x}} \cos (\theta_x - \theta_y) \right) \leq  \exp \left( 2d \beta_1 \right),
    \end{equation*}
    we deduce that
    \begin{align}
         Z_{\Lambda_L^n, n , \beta_1 , \beta_2 , r_1^+} & = \int \exp \left( \beta_1 \sum_{\substack{x \sim_n y \\  x \in \Zd \cap \Lambda_L^n}} r_{1,x}^+ \cos (\theta_x - \theta_y) + n \beta_2 \sum_{\substack{x \sim_n y \\ x , y \in \Lambda_L^n \setminus \Zd}} \cos (\theta_x - \theta_y) \right) \prod_{x \in \Lambda_L^n} d \theta_x \notag \\
         & \leq \exp \left( 2d \beta_1 \right)  \int \exp \left( \beta_1 \sum_{\substack{x \sim_n y \\  x \in \Zd \cap \Lambda_L^n}} r_{1,x}^- \cos (\theta_x - \theta_y) + n \beta_2 \sum_{\substack{x \sim_n y \\ x , y \in \Lambda_L^n \setminus \Zd}} \cos (\theta_x - \theta_y) \right) \prod_{x \in \Lambda_L^n} d \theta_x \notag \\
         & = \exp \left( 2d \beta \right) Z_{\Lambda_L^n, n , \beta_1 , \beta_2 , r_1^-}. \notag
    \end{align}
    Combining the previous inequality with~\eqref{eq:1635}, we deduce that
    \begin{equation*}
        \nu \left( r_z = 1 \, | \, r = r_1 ~ \mbox{in} ~ \Lambda \setminus \{ z \} \right) \leq \frac{1}{1 + \exp \left( - 2d \beta_1 \right)} = p_0(\beta_1).
    \end{equation*}
   The proof of the inequality~\eqref{eqnunu'DC} is complete.
\end{proof}

\subsection{Phase transition for the Villain model on a high-density Bernoulli percolation} \label{Section5}

This section is devoted to the proof of Theorem~\ref{proofKTsitegeneral} in the case of a high-density Bernoulli site percolation. We combine the stochastic domination of Proposition~\ref{prop:stochdomWellsVillain} with the Ginibre inequality (Theorem~\ref{theoremGinibreineq}), the Wells' inequality (Theorem~\ref{prop.Wells}) and Propositions~\ref{prop:prop3.6} and~\ref{prop.phasetransitionextended} to show the existence of a phase transition for the Villain model on a high-density i.i.d. Bernoulli site percolation. Specifically, we will prove the following statement.

\begin{proposition} \label{prop3.12}
For any dimension $d \geq 2$, there exist a probability $p_0 := p_0(d) < 1$ and an inverse temperature $\beta_0 := \beta_0(d) \in (0 , \infty)$ such that for any $p \geq p_0$ and any $\beta \geq \beta_0$:
\begin{itemize}
        \item In dimension $d = 2$, the function $x \mapsto \mathbb{E}_p^{\mathrm{site}} [  \left\langle \cos  ( \theta_0 - \theta_x ) \right\rangle_{\beta, r}^\mathrm{Vil}]$ decays polynomially fast in $|x|$.
        \item In dimension $d \geq  3$, the function $x \mapsto \mathbb{E}_p^{\mathrm{site}} [  \left\langle \cos  ( \theta_0 - \theta_x ) \right\rangle_{ \beta, r}^{\mathrm{Vil}}]$ is bounded away from $0$.
    \end{itemize}
\end{proposition}

\begin{proof}
    Let $\beta_{\mathrm{KT}}^{\mathrm{XY/Vil}}$ and $\beta_{c}^{\mathrm{XY/Vil}}$ be the two inverse temperatures provided by Proposition~\ref{prop.phasetransitionextended}. In dimension $d = 2$, choose $\beta_1 = \beta_2 = 4\beta_{\mathrm{KT}}^{\mathrm{XY/Vil}}$ and $p_0 := p_0(4\beta_{\mathrm{KT}}^{\mathrm{XY/Vil}}) \in (0 , 1),$ and in dimension $d \geq 3$, select $\beta_1 = \beta_2 = 4\beta_{c}^{\mathrm{XY/Vil}}$ and $p_0 := p_0(4\beta_{c}^{\mathrm{XY/Vil}}) \in (0 , 1).$
    \medskip
    
    Let us fix a vertex $x \in \Zd$, a (large) integer $L \in \N$ (such that $x \in \Lambda_L$) and an integer $n \in \N$. We first combine the stochastic domination of Proposition~\ref{prop:stochdomWellsVillain} with the Ginibre inequality (Theorem~\ref{theoremGinibreineq}) to deduce that
    \begin{equation*}
        \E_{\nu_{L , n , \beta_1 , \beta_2}} \left[ \left\langle \cos(\theta_0 - \theta_x) \right\rangle_{L, n, \beta_1, \beta_2, r}^{\mathrm{XY}} \right] \leq \E_{p_0}^{\mathrm{site}} \left[ \left\langle \cos(\theta_0 - \theta_x) \right\rangle_{L , n , \beta_1 , \beta_2, r}^{\mathrm{XY}} \right].
    \end{equation*}
    We next apply Wells' inequality (Proposition~\ref{prop.Wells}) to obtain the bound
    \begin{equation} \label{eq:17231003}
        \left\langle \cos(\theta_0 - \theta_x) \right\rangle_{L, n, \beta_1/4, \beta_2/4}^{\mathrm{XY}} \leq \mathbb{E}_{p_0}^{\mathrm{site}} \left[ \left\langle \cos(\theta_0 - \theta_x) \right\rangle_{L , n , \beta_1 , \beta_2, r}^{\mathrm{XY}} \right].
    \end{equation}
    We then take the limit $n \to \infty$ and apply Proposition~\ref{prop:prop3.6} (specifically, the second convergence in the statement of this proposition) to deduce using dominated convergence theorem that
    \begin{equation*}
        \lim_{n \to \infty} \mathbb{E}_{p_0}^{\mathrm{site}} \left[ \left\langle \cos(\theta_0 - \theta_x) \right\rangle_{L , n , \beta_1 , \beta_2, r}^{\mathrm{XY}} \right] = \mathbb{E}_{p_0}^{\mathrm{site}} \left[ \left\langle \cos \left( \theta_0 - \theta_x \right)  \right\rangle_{L, \beta_1 , \beta_2, r}^{\mathrm{XY/Vil}} \right]
    \end{equation*}
    and
    \begin{equation*}
        \lim_{n \to \infty} \left\langle \cos(\theta_0 - \theta_x) \right\rangle_{L, n, \beta_1/4, \beta_2/4}^{\mathrm{XY}} = \left\langle \cos \left( \theta_0 - \theta_x \right)  \right\rangle_{L, \beta_1/4 , \beta_2/4}^{\mathrm{XY/Vil}}.
    \end{equation*}
    We may then combine the two previous convergences with the inequality~\eqref{eq:17231003} to deduce that
    \begin{equation*}
        \left\langle \cos \left( \theta_0 - \theta_x \right)  \right\rangle_{L, \beta_1/4 , \beta_2/4}^{\mathrm{XY/Vil}} \leq \mathbb{E}_{p_0}^{\mathrm{site}} \left[ \left\langle \cos \left( \theta_0 - \theta_x \right)  \right\rangle_{L, \beta_1 , \beta_2, r}^{\mathrm{XY/Vil}} \right].
    \end{equation*}
    We next apply the inequality~\eqref{ineq:beta1toinfty} of Remark~\ref{rem:remark10} (i.e., we take the limit $\beta_1 \to \infty$ on the right-hand side) and obtain
    \begin{equation*}
        \left\langle \cos \left( \theta_0 - \theta_x \right)  \right\rangle_{L, \beta_1/4 , \beta_2/4}^{\mathrm{XY/Vil}} \leq \mathbb{E}_{p_0}^{\mathrm{site}} \left[ \left\langle \cos \left( \theta_0 - \theta_x \right)  \right\rangle_{L, \beta_2, r}^{\mathrm{Vil}} \right].
    \end{equation*}
    Taking the limit $L \to \infty$ on both sides (N.B. Due Corollary~\ref{corollary:ginibreforVillain} both the right and left-hand sides are increasing in $L$), we deduce that
    \begin{equation} \label{eq:17451003}
        \left\langle \cos \left( \theta_0 - \theta_x \right)  \right\rangle_{ \beta_1/4 , \beta_2/4}^{\mathrm{XY/Vil}} \leq \mathbb{E}_{p_0}^{\mathrm{site}} \left[ \left\langle \cos \left( \theta_0 - \theta_x \right)  \right\rangle_{ \beta_2, r}^{\mathrm{Vil}} \right].
    \end{equation}
    Using that the inverse temperatures $\beta_1, \beta_2$ 
%\margin{C: it seems at this point we can also make $\beta_1$ goes to infinity here right ? as the RHS does not depend on $\beta_1$ ? Does it simplify things a bit ? \textbf{NO !!  $p_0 =p_0(\beta_1) $ !!! }}
have been chosen so that the term on the left-hand side decays polynomially fast in two-dimensions and remains bounded away from $0$ in dimension $3$ and higher, we deduce that the same result holds for the right-hand side of~\eqref{eq:17451003}. 
    
    Finally, from Remark~\ref{rem:remark4} and the Ginibre inequality (Theorem~\ref{theoremGinibreineq}), we see that the right-hand side of~\eqref{eq:17451003} is increasing in the probability $p$ and in the inverse temperature $\beta_2$. We thus deduce that, for any probability $p \geq p_0$ and any inverse temperature $\beta \geq \beta_2$, the function $x \mapsto \mathbb{E}_{p}^{\mathrm{site}} \left[ \left\langle \cos \left( \theta_0 - \theta_x \right)  \right\rangle_{ \beta, r}^{\mathrm{Vil}} \right]$ decays polynomially fast in dimension $d = 2$ and remains bounded away from $0$ in dimension $d \geq 3$.
\end{proof}

\subsection{Phase transitions for the Villain model on a supercritical percolation cluster} \label{section5.2}

\begin{figure}
\centering
\includegraphics[scale=0.4]{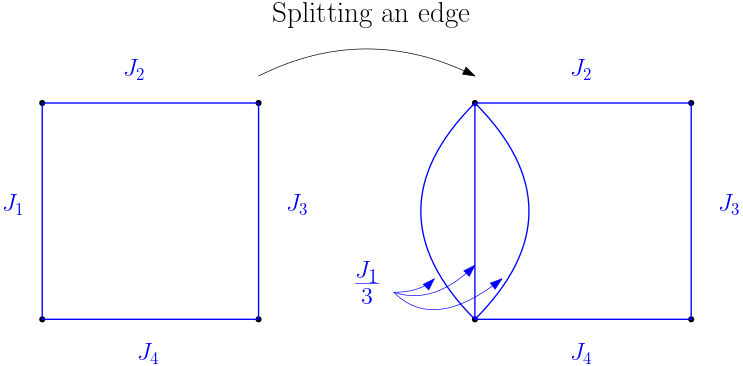}
\caption{Splitting an edge in the Villain model: an edge with conductance $J$ is replaced by two parallel edges with conductances $J/2$. This operation reduces the value of the two-point functions.} \label{fig:splitVillain}
\end{figure}

The proof of the existence of the BKT phase transition in dimension $2$ and of the order/disorder phase transition in $d\geq 3$ on a supercritical Bernoulli (site or edge) percolation is then deduced from the existence of these phase transitions on a high-density Bernoulli site percolation (i.e., the result of Section~\ref{Section5}) by using a renormalisation argument. The proof is essentially identical to the one written in~\cite[Section 5.2]{DG25} in the case of the XY model (up to one difference presented in Lemma~\ref{cor:merging} below) and is similar to the one written for the integer-valued height functions in Section~\ref{sec:renoarmargumentforheight}.  We thus decided not to rewrite the entire argument but to state and prove the property which differs in the cases of the XY and Villain models.

\begin{lemma}[Splitting an edge of the Villain model] \label{cor:merging}
Fix a subset $\Lambda \subseteq \Zd$, a collection of non-negative conductances $J := (J_{xy})_{\{x , y \} \in E \left( \Lambda \right)}$ and consider the Villain model introduced in Definition~\ref{generalXY} (with $n=1$). Then perform the following operation: split one edge $\{x_0 , x_1 \} \in E \left( \Lambda \right)$ into $k$ parallel edges and assign to each of these edges the conductance $J_{x_0 x_1}/k$ (see Figure~\ref{fig:splitVillain}). Denoting by $\tilde J$ the new collection of conductances, we have the inequality, for any $x , y \in \Lambda$,
	\begin{equation*}
		\left\langle \cos \left( \theta_x - \theta_y \right) \right\rangle^{\mathrm{Vil}}_{\Lambda , J} \geq \left\langle \cos \left( \theta_x - \theta_y \right) \right\rangle^{\mathrm{Vil}}_{\Lambda , \tilde J}.
	\end{equation*}
\end{lemma}

\begin{remark}
	In the case of the XY model, this operation does not affect the distribution of the spin system, there is thus equality of the two-point functions.
\end{remark}

\begin{figure}
\centering
\includegraphics[scale=0.4]{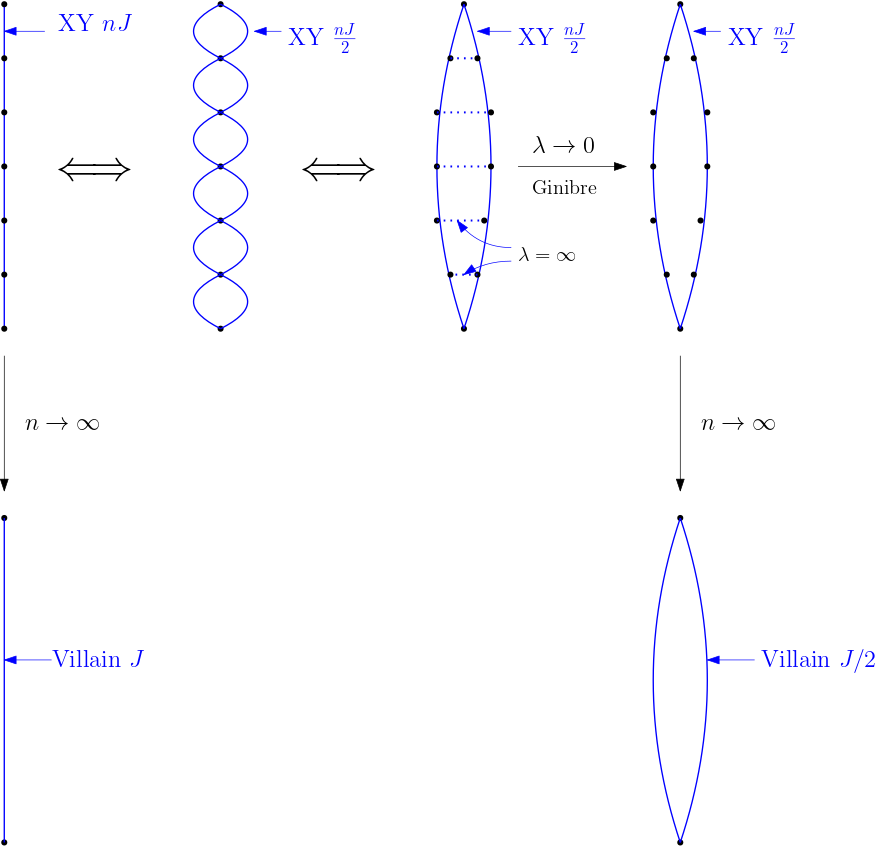}
\caption{An illustration of the proof of Lemma~\ref{cor:merging}. The Villain model is approximated by a chain of XY models. Each edge of the chain is split into $k=2$ edges with reduced conductances. Two parallel edges are then generated by reducing the value of the coupling contants from infinity to $0$. By the Ginibre inequality, this operation reduces the two-point functions.} \label{fig:proofsplittingVillain}
\end{figure}

\begin{proof}
	The proof of this lemma is obtained by approximating the Villain model by XY models and by using the Ginibre inequality. We proceed in 5 steps and refer to Figure~\ref{fig:proofsplittingVillain} for guidance:
	\begin{itemize}
		\item We first approximate the Villain model by a chain of XY models (as in Proposition~\ref{prop:prop3.6}). Specifically, we fix an integer $n \in \N^*$, we add $(n-1)$ vertices on the edge $\{x_0 , x_1\} \in E(\Lambda)$ and assign to each of the $n$ new edges the conductance $n J_{x_0x_1}$. We then consider the XY model on this graph with these conductances. Denoting by $\left\langle \cdot \right\rangle^{\mathrm{XY}}_{\Lambda, n , J}$ the expectation with respect to this measure, we have the convergence
		\begin{equation*}
			\lim_{n \to \infty} \left\langle \cos(\theta_x - \theta_y) \right\rangle^{\mathrm{XY}}_{\Lambda, n , J} = \left\langle \cos(\theta_x - \theta_y) \right\rangle^{\mathrm{Vil}}_{\Lambda, J}.
		\end{equation*}
		We refer to~\cite[Section 4]{NW} and~\cite[Appendix A.2]{AHPS} for a proof of this result.
		\item We then split the $n$ edges connecting the vertices $x_0$ and $x_1$ into $k$ parallel edges and assign to each of these edges the conductance $nJ_{x_0 x_1}/k$. This operation does not change the distribution of the corresponding XY model.
		\item We then separate the edge into $k$ parallel edges (see Figure~\ref{fig:proofsplittingVillain}) and connect the parallel vertices by a coupling constant equal to infinity. This operation does not change the distribution of the corresponding XY model (see Remark~\ref{rem:remarkJ=infty}).
		\item We then reduce the value of the coupling constants from infinity to $0$. Denoting by $\left\langle \cdot \right\rangle^{\mathrm{XY}}_{\Lambda, n , \tilde J}$ the expectation with respect to the new XY model, we have by the Ginibre inequality
		\begin{equation*}
			 \left\langle \cos(\theta_x - \theta_y) \right\rangle^{\mathrm{XY}}_{\Lambda, n , J} \geq \left\langle \cos(\theta_x - \theta_y) \right\rangle^{\mathrm{XY}}_{\Lambda, n , \tilde J}.
		\end{equation*}
		\item We finally note that, by~\cite{NW, AHPS}, we have the convergence
		\begin{equation*}
			\lim_{n \to \infty} \left\langle \cos(\theta_x - \theta_y) \right\rangle^{\mathrm{XY}}_{\Lambda, n , J}= \left\langle \cos(\theta_x - \theta_y) \right\rangle^{\mathrm{Vil}}_{\Lambda, \tilde J}.
		\end{equation*}
	\end{itemize}
A combination of the three previous displays completes the proof of Lemma~\ref{cor:merging}.
\end{proof}

\subsection{Phase transitions for spin systems with annealed Villain interaction} \label{section:annealedVillain}

This section is devoted to the proof of Theorem~\ref{Th:annealedVillain} following the outline presented in Section~\ref{sec:1.4techniquesofproof}.

\begin{proof}[Proof of Theorem~\ref{Th:annealedVillain}]
Let us fix a probability measure $\kappa$ on $(0 , \infty)$, an inverse temperature $\beta \in ( 0 , \infty)$ and an integer $L \in \N$. To simplify the argument, we only write the proof in two dimensions. We first use the definition of the annealed Villain potential to rewrite the two-point function as follows
%\margin{C: I need to check above: on the Annealed Villain, we are acting with $\beta$ as $v_{\beta*J}$ ?}
\begin{align} \label{eq:firstcompannealedVillain}
    \lefteqn{\left\langle \cos \left( \theta_0 - \theta_x  \right)   \right\rangle_{L , \kappa , \beta}^{\mathrm{Ann-Vil}}} \qquad & \\ & = \frac{1}{Z^{\mathrm{Ann-Vil}}_{L, \kappa, \beta}} \int_{[0 , 2\pi)^{\Lambda_L}} \cos \left( \theta_0 - \theta_x  \right)  \prod_{\substack{x \sim y \\ x , y \in \Lambda_L}} F_{\kappa, \beta}(\theta_x - \theta_y) \prod_{x \in \Lambda_L} d\theta_x \notag\\
    & = \frac{1}{Z^{\mathrm{Ann-Vil}}_{L, \kappa, \beta}} \int_{[0 , 2\pi)^{\Lambda_L}} \cos \left( \theta_0 - \theta_x  \right)  \prod_{\substack{x \sim y \\ x , y \in \Lambda_L}} \left( \int_0^\infty v_{\beta J}(\theta_x - \theta_y) \kappa(d J) \right) \prod_{x \in \Lambda_L} d\theta_x  \notag  \\
    & =  \frac{1}{Z^{\mathrm{Ann-Vil}}_{L, \kappa, \beta}}  \int_{(0 , \infty)^{E\left( \Lambda_L \right)}} \int_{[0 , 2\pi)^{\Lambda_L}} \cos \left( \theta_0 - \theta_x  \right)  \prod_{\substack{x \sim y \\ x , y \in \Lambda_L}} v_{\beta J_{xy}}(\theta_x - \theta_y)   \prod_{x \in \Lambda_L} d\theta_x \prod_{\substack{x \sim y \\ x , y \in \Lambda_L}} \kappa(d J_{xy}). \notag
\end{align}
Using Definition~\ref{generalXY}, we may rewrite the right-hand side of the previous display as follows
\begin{equation*}
    \left\langle \cos \left( \theta_0 - \theta_x  \right)   \right\rangle_{L , \kappa , \beta}^{\mathrm{Ann-Vil}} =  \frac{1}{Z^{\mathrm{Ann-Vil}}_{\Lambda, \kappa, \beta}}  \int_{(0 , \infty)^{E\left( \Lambda_L \right)}} Z^{\mathrm{Vil}}_{L , \beta J} \left\langle \cos \left( \theta_0 - \theta_x  \right)  \right\rangle_{L , \beta J}^{\mathrm{Vil}} \prod_{\substack{x \sim y \\ x , y \in \Lambda}} \kappa(d J_{xy}).
\end{equation*}
By the Ginibre correlation inequality (Theorem~\ref{theoremGinibreineq} and Proposition~\ref{prop:prop3.6}), we know that the function $J \mapsto \left\langle \cos \left( \theta_0 - \theta_x  \right)  \right\rangle_{L , \beta J}^{\mathrm{Vil}}$ is increasing. We next show that the function $J \mapsto Z^{\mathrm{Vil}}_{L , \beta J}$ is also increasing. To show this monotonicity property, we use the definition of the Villain interaction potential~\eqref{def:introdefVillain} to write
\begin{align*}
    Z^{\mathrm{Vil}}_{L , \beta J} & = \int_{[0 , 2\pi)^{\Lambda_L}} \prod_{\substack{x \sim y \\ x , y \in \Lambda}} v_{\beta J_{xy}}(\theta_x - \theta_y)   \prod_{x \in \Lambda} d\theta_x \\
    & = \int_{[0 , 2\pi)^{\Lambda_L}} \prod_{\substack{x \sim y \\ x , y \in \Lambda}} \left( \sum_{k \in \Z} e^{- \frac{k^2}{2\beta J_{xy}}} e^{i k (\theta_x - \theta_y)} \right) \prod_{x \in \Lambda} d\theta_x \\
    & = \sum_{\mathbf{k} \in \Z^{E\left(\Lambda_L\right)}} \underset{\mathrm{is \, increasing \, in \, } J}{\underbrace{\left( \prod_{\substack{x \sim y \\ x , y \in \Lambda}} e^{- \frac{\mathbf{k}_{xy}^2}{2 \beta J_{xy}}} \right)}} \underset{\mathrm{is \, equal \, to}\, 0 \mathrm{\, or \,} (2\pi)^{\left| \Lambda_L \right|}}{\underbrace{\left( \int_{[0 , 2\pi)^{\Lambda_L}}  \prod_{\substack{x \sim y \\ x , y \in \Lambda}}  e^{i (\theta_x- \theta_y) \mathbf{k}_{xy}} \prod_{x \in \Lambda} d\theta_x \right)}}
\end{align*}
The partition function $Z^{\mathrm{Vil}}_{L , \beta J}$ can be written as a sum of functions which are increasing in $J$, it is thus increasing in $J$. 

In the rest of this proof, we will denote by $\mathbf{E}$ the expectation with respect to the product measure $\prod_{x \sim y} \kappa (d J_{xy})$. Using the identity 
$$ Z^{\mathrm{Ann-Vil}}_{\Lambda, \kappa, \beta} =  \mathbf{E} \left[ Z^{\mathrm{Vil}}_{L , \beta J} \right],$$
which follows from the same computation as in~\eqref{eq:firstcompannealedVillain}, and applying the FKG inequality (Proposition~\ref{prop.FKGinequality}), we obtain
\begin{equation} \label{ineq:stochdomannVillain}
    \left\langle \cos \left( \theta_0 - \theta_x  \right)   \right\rangle_{L , \kappa , \beta}^{\mathrm{Ann-Vil}} \geq \mathbf{E} \left[ \left\langle \cos \left( \theta_0 - \theta_x  \right)  \right\rangle_{L , \beta J}^{\mathrm{Vil}}  \right].
\end{equation}
We let $p :=  3/4$ and $\beta_{\mathrm{KT}}^{\mathrm{Vil}}(3/4) \geq 0$ be the inverse temperature provided by Theorem~\ref{proofKTsitegeneral}. We next select $\beta_0 \in (0, \infty)$ sufficiently large such that
    \begin{equation*}
        \kappa \left( \left( \frac{\beta_{\mathrm{KT}}^{\mathrm{Vil}}(3/4)}{\beta_0} , \infty \right) \right) \geq \frac{3}{4}.
    \end{equation*}
    Given a collection of conductances $J = (J_{xy})_{ \{ x , y\} \in E \left( \Lambda_L \right)}$, we define the edge percolation configuration $\omega^J \in \{ 0,1\}^{E\left(\Lambda_L \right)}$ according to the identity, for any $\{ x, y \} \in E \left( \Lambda_L \right)$,
    \begin{equation*}
        \omega_{xy}^J := \indc_{\{ \beta_0 J_{xy} \geq \beta_{\mathrm{KT}}^{\mathrm{Vil}}(3/4)\}}.
    \end{equation*}
    We next observe that, by definition of the percolation configuration $\omega^J$ and the Ginibre inequality, we have the inequality, for any collection of conductances $J = (J_{xy})_{ \{ x , y\} \in E \left( \Lambda_L \right)}$,
    \begin{equation*}
        \left\langle \cos \left( \theta_0 - \theta_x \right) \right\rangle_{L ,\beta_0 J}^{\mathrm{Vil}}  \geq \left\langle \cos \left( \theta_0 - \theta_x \right) \right\rangle_{L ,\beta_{\mathrm{KT}}^{\mathrm{Vil}}(3/4), \omega^J}^{\mathrm{Vil}}.
    \end{equation*}
    If we assume that the conductances $J =  (J_{xy})_{ \{ x , y\} \in E \left( \Lambda_L \right)}$ are i.i.d. of law $\kappa$, then the percolation process $\omega^J \in \{0 , 1 \}^{\Lambda_L}$ stochastically dominates an i.i.d. edge percolation of probability 3/4. These observations imply that
    \begin{equation*}
    \mathbf{E}\left[ \left\langle \cos \left( \theta_0 - \theta_x  \right)  \right\rangle_{L , \beta_0 J}^{\mathrm{Vil}} \right] \geq  \mathbf{E}\left[ \left\langle \cos \left( \theta_0 - \theta_x \right) \right\rangle_{L ,\beta_{\mathrm{KT}}^{\mathrm{Vil}}(3/4), \omega^J}^{\mathrm{Vil}} \right] \geq \E_{3/4} \left[ \left\langle \cos \left( \theta_0 - \theta_x \right) \right\rangle_{L ,\beta_{\mathrm{KT}}^{\mathrm{Vil}}(3/4), \omega}^{\mathrm{Vil}} \right].
    \end{equation*}
    Combining this inequality with~\eqref{ineq:stochdomannVillain}, we eventually obtain, for any $\beta \geq \beta_0$,
    \begin{equation*}
    \left\langle \cos \left( \theta_0 - \theta_x  \right)   \right\rangle_{L , \kappa , \beta}^{\mathrm{Ann-Vil}} \geq \mathbf{E}\left[ \left\langle \cos \left( \theta_0 - \theta_x  \right)  \right\rangle_{L , \beta J}^{\mathrm{Vil}} \right] \geq \mathbf{E}\left[ \left\langle \cos \left( \theta_0 - \theta_x  \right)  \right\rangle_{L , \beta_0 J}^{\mathrm{Vil}} \right] \geq \E_{3/4} \left[ \left\langle \cos \left( \theta_0 - \theta_x \right) \right\rangle_{L ,\beta_{\mathrm{KT}}^{\mathrm{Vil}}(3/4), \omega}^{\mathrm{Vil}} \right].
    \end{equation*}
    We finally take the limit $L \to \infty$ on both sides and apply Theorem~\ref{proofKTsitegeneral} to deduce that the function $x \mapsto \left\langle \cos \left( \theta_0 - \theta_x  \right)   \right\rangle_{\kappa , \beta}^{\mathrm{Ann-Vil}}$ decays polynomially fast in two dimensions. This completes the proof of Theorem~\ref{Th:annealedVillain}.
\end{proof}

%\purple{C: Should we add a Remark here that we easily get $XY$ power law decay out of Villain Power law decay here ? (just using that XY is an annealed Villain ? (as long as we have percolation stability for BKT with Villain on the same graph ? This is quite nice). \textbf{No! Again. We need both annealed $\Z$-XY and annealed $\Z$ GFF}}

\section{Integer-valued height functions on a percolation configuration} \label{sec:sec4height}

This section is devoted to the integer-valued height functions and contains the proofs of Theorem~\ref{Thm:delocheightsupercrit} and Theorem~\ref{Thm:delocheightanealedgauss}.

We will proceed in a similar fashion as in Section~\ref{sectiondef.XY} for the Villain model. We first collect, in Section~\ref{sec:sec4.1}, some preliminary definitions and properties which are used in the proofs of Theorems~\ref{Thm:delocheightsupercrit} and~\ref{Thm:delocheightanealedgauss} (e.g., the monotonicity of the variance of the height in the conductances, the Wells' inequality, the existence of a phase transition for height functions, etc.). Section~\ref{App.A} contains the proof of the Wells' inequality for height functions, which as we mentioned earlier  is one of the main novelties of this article.
%\margin{C: I may add this somewhere in the introduction. \textbf{done.}}
 It is obtained by combining the original argument of Wells~\cite{Wellsthesis} with an identity between the XY height function and a version of the XY model originally due to~\cite{van2023duality}. Once equipped with the Wells' inequality, the rest of the argument follows the same outline as the one for the Villain model of Section~\ref{sectiondef.XY}. We first show the existence of a delocalised regime for the height functions on a high-density (site) Bernoulli percolation in Sections~\ref{sec:sec4.3},~\ref{sec:sec4.4} and~\ref{sec:sec4.4}. We then upgrade the result from a high-density (site) percolation to any supercritical (edge) percolation by using a renormalisation argument whose details can be found in Section~\ref{sec:renoarmargumentforheight}. Finally, Section~\ref{section4.8} contains the proof of Theorem~\ref{Thm:delocheightanealedgauss}.

\subsection{Preliminaries on integer-valued height functions} \label{sec:sec4.1}

\subsubsection{The integer-valued Gaussian free field and the XY height function} \label{sec:section4.1}

In this section, we introduce general versions of the two models of integer-valued height functions for which we will prove that the roughening phase transition subsists in the presence of a random disorder. We first start with a definition for the integer-valued Gaussian free field (allowing the model to have general conductances).

\begin{definition}[The integer-valued Gaussian free field with general conductances] \label{def:def4.1}
%Let $\Lambda \subseteq \Z^2$ be a finite subset of $\Z^2$ and let 
%\begin{equation*}
%    \Omega_\Lambda^\Z := \left\{ \varphi : \Lambda \to \Z \, : \, \varphi = 0 ~~\mbox{on}~~ \partial \Lambda \right\}.
%\end{equation*}
Given a finite subset $\Lambda \subseteq \Z^2$, a collection of conductances $\left( J_{xy} \right)_{xy \in E(\Lambda)} \in [0, \infty)^{E(\Z^2)}$, we define the integer-valued Gaussian free field on the set $\Lambda$ with conductances $J = (J_{xy})$ to be the probability distribution on the set $\Omega_\Lambda^\Z$ given by the identity
(recall the definition of the set $\Omega_\Lambda^\Z$ given above in~\eqref{e.ZGFF})
%\margin{C: I added a point to that definition as I was confused}
\begin{equation*} 
    \mu_{\Lambda , J}^{\Z-\mathrm{GFF}} (\left\{ \varphi \right\}) := \frac{1}{Z_{\Lambda,J}} \exp \left( - \sum_{x \sim y}  \frac{\left( \varphi(x) - \varphi(y) \right)^2}{2J_{xy}} \right).
\end{equation*}
We denote by $\mathrm{Var}_{\Lambda , J}^{\Z-\mathrm{GFF}}$ the variance with respect to the measure $\mu_{\Lambda , J}^{\Z-\mathrm{GFF}}$ and write $\mathrm{Var}_{L , J}^{\Z-\mathrm{GFF}}$ when $\Lambda = \Lambda_L$.
\end{definition}

The second model we introduce is the XY height function. 
%This model has already been studied in~\cite{frohlich1981kosterlitz, van2023elementary, van2023duality, AHPS} and its definition is motivated by the existence of a duality relation with the XY model  (see Remark~\ref{rem:remark13} below).
We mention that, as it will be an important ingredient on the proof of Theorem~\ref{Thm:delocheightsupercrit}, we will define the model on the multigraph $\Z^2_{n-\mathrm{mult}}$ (see Definition~\ref{def.multigraph} and Figure~\ref{fig:figure1.3}).

\begin{definition}[The XY height function with general conductances] \label{def:XYheightfunction}
%For $J \in [0 , \infty),$ we define the modified Bessel function according to the identity, for any $k \in \Z$,
%\begin{equation} \label{def:Besselfunction}
%    I_k(J) := \sum_{j = 0}^\infty \frac{1}{j! (j + |k|)!} \left( \frac{J}{2} \right)^{2j + |k|}.
%\end{equation}
%\purple{Equivalently, it may also be defined as the $k^{th}$ Fourier coefficient of the periodic function $\theta \mapsto e^{J \cos(\theta)}$ which explains the link with the $XY$ model. }
Given a finite subset $\Lambda \subseteq \Z^2$, an integer $n \in \N^*$, and a collection of conductances $\left( J_{xy,k} \right)_{xy \in E(\Z^2), 1 \leq k \leq n} \in [0, \infty)^{E(\Z^2) \times \{1 , \ldots, n \}}$, we define the XY height function to be the probability distribution on the set $\Omega_\Lambda^{\Z}$ given by the identity
\begin{equation*}
    \mu_{\Lambda , J}^{\Z-\mathrm{XY}} (\left\{ \varphi \right\}) := \frac{1}{Z_{\Lambda,J}^{\Z-\mathrm{XY}}} \prod_{x \sim y} \prod_{1 \leq k \leq n}  I_{\left(\varphi(x) - \varphi(y)\right)} \left( J_{xy,k} \right) .
\end{equation*}
We denote by $\mathrm{Var}_{\Lambda , J}^{\Z-\mathrm{XY}}$ the variance with respect to the measure $\mu_{\Lambda , J}^{\Z-\mathrm{XY}}$ and write $\mathrm{Var}_{L , J}^{\Z-\mathrm{XY}}$ when $\Lambda = \Lambda_L$.
\end{definition}

\begin{remark} \label{rem:remark13}
Let us make a few remarks about the previous definition:
\begin{itemize}
\item There is no gain in generality in defining the integer-valued Gaussian free field on the multigraph $\Z^d_{n-\mathrm{mult}}$ as, in that case, two parallel edges can be merged together (while suitably adjusting the values of the conductances) without changing the law of the height function.
\item 
The use of the modified Bessel function to define a height function is motivated by the identity: for any $\beta \in [0 , \infty)$ and $\theta \in [0 , 2\pi)$,
\begin{equation} \label{eq:defmodifiedBesselfunction}
    e^{\beta \cos \theta} = \sum_{k \in \Z} I_k(\beta) e^{i k \theta}.
\end{equation}
%As a matter of illustration, the identity~\eqref{eq:defmodifiedBesselfunction} implies the following identity of partition functions (see Figure~\ref{fig:dualityXY})
%\margin{C: I will modify this one. $\Lambda^*$ needs to appear.}
%\begin{equation} \label{eq:identitypartfunctoriginal}
%    (2 \pi)^{\left| \Lambda \right|} Z_{\Lambda,J}^{\Z-\mathrm{XY}} = Z_{\Lambda^*,J^*}^{\mathrm{XY}}.
%\end{equation}
%where $\Lambda^*$ is the dual set and $J^*$ is the dual collection of conductances.
There is a deep connection between the XY model and the XY height function, recently investigated by the second author and Lis~\cite{van2023duality}. This connection allows for the transfer of some analytical tools developed for the XY model to study the XY height function. Notably, the Ginibre inequality (Theorem~\ref{theoremGinibreineq}) can be employed to establish the monotonicity of the variance of the height in the conductances (see Proposition~\ref{prop:monotonicity}). Similarly, the Wells inequality also carries over in this context (see Section~\ref{App.A}).

\item Two properties of the modified Bessel function will be interesting for us. First we have that $I_k(0) = \indc_{\{ k = 0 \}}$; this means that if $J_{xy,k} = 0$ for one of the integers $k \in \{1 , \ldots, n\}$, then the two heights $\varphi(x)$ and $\varphi(y)$ are equal almost-surely (with $\varphi$ sampled according to $\mu_{\Lambda , J}^{\Z-\mathrm{XY}}$). Second, we have the asymptotic expansion as $J \to \infty$ (which is due to~\cite{kirchhoff1854ueber} and appears in~\cite[Example 2]{van2023duality})
%\margin{C: should we state also an asymptotics as $k\to \infty$ ? (Which if I am not mistaken is SOS like ?) I find it even more relevant but maybe I'm wrong. }
\begin{equation} \label{eq:expansionBessel}
    I_k(J) = \frac{e^{J}}{\sqrt{2 \pi J}} \left( 1 - \frac{4k^2-1}{8 J} + O\left( \frac{1}{J^2}\right) \right).
\end{equation}
\end{itemize}
\end{remark}

\begin{remark}
Interestingly, such an asymptotics can be used to give a different proof of the fact that cable XY model leads to Villain model. Such a proof would be more in the spirit of the standard CLT proof rather than a local CLT theorem, which is the approach used in \cite{AHPS}.
\end{remark}

%\margin{C: I added this remark after our discussion.
%D: Yes, this was in my paper with Marcin that is cited here..}

In the rest of Section~\ref{sec:sec4height}, 
three models will play an important role, 
and for this reason we introduce specific pieces of notation for them 
(These models can be compared to the ones introduced in Section~\ref{sec:sectionwith3models}).

\begin{definition}
We fix an integer $L \in \N$, an integer $n \in \N^*$ and three inverse temperatures $\beta, \beta_1, \beta_2 > 0$ and define three models of height functions (see Figures~\ref{fig:XYheight} and~\ref{fig:XYGFFheight}):
\begin{itemize}
    \item The first one is the XY height function on the multigraph $\Z^d_{n-\mathrm{mult}}$ with inverse temperature $n \beta$ on each edge, i.e.,
    \begin{equation*}
        \mu_{\Lambda_L, n, \beta}^{\Z-\mathrm{XY}}(\left\{ \varphi \right\}) := \frac{1}{Z_{\Lambda_L, n, \beta}^{\Z-\mathrm{XY}}}  \prod_{x \sim y}    \left(I_{\left(\varphi(x) - \varphi(y)\right)} \left(n \beta \right)\right)^n.
    \end{equation*}
    \item The second one is an XY height function multigraph $\Z^2_{n- \mathrm{mult}}$ with $n \geq 2$ and with two inverse temperatures $\beta_1 , n \beta_2  \geq 0$ (see Figure~\ref{fig:XYheight}):
\begin{equation*}
        \mu_{\Lambda_L, n, \beta_1 , \beta_2}^{\Z-\mathrm{XY}}(\left\{ \varphi \right\}) := \frac{1}{Z_{\Lambda_L, n, \beta_1,  \beta_2 }^{\Z-\mathrm{XY}}}\prod_{x \sim y} \left( I_{\left(\varphi(x) - \varphi(y)\right)} \left( \beta_1 \right) \right)^2   \left( I_{\left(\varphi(x) - \varphi(y)\right)} \left(  n \beta_2 \right) \right)^{n-2}.
\end{equation*}
    \item The third one is a height function on the multigraph $\Z^2_{3-\mathrm{mult}}$ where the interaction between two vertices is a product of the integer-valued Gaussian distribution and the modified Bessel function (see Figure~\ref{fig:XYGFFheight})
    \begin{equation} \label{eq:defXYGFF}
        \mu_{\Lambda_L, \beta_1 , \beta_2}^{\Z-\mathrm{XY}/\mathrm{GFF}}(\left\{ \varphi \right\}) := \frac{1}{Z_{\Lambda_L, \beta_1 , \beta_2 }^{\Z-\mathrm{XY}/\mathrm{GFF}}} \prod_{x \sim y} \left( I_{\left(\varphi(x) - \varphi(y)\right)} \left( \beta_1 \right) \right)^2 \exp \left( - \frac{\left( \varphi(x) - \varphi(y) \right)^2}{2\beta_2} \right).
    \end{equation}
    We refer to this model as the $\Z$-XY/GFF height function.
\end{itemize}
We denote by $\mathrm{Var}_{L, n,  \beta}^{\Z-\mathrm{XY}}$, $\mathrm{Var}_{L, n, \beta_1 , \beta_2}^{\Z-\mathrm{XY}}$ and $\mathrm{Var}_{L, \beta_1 , \beta_2}^{\Z-\mathrm{XY}/\mathrm{GFF}}$ the variances with respect to these three measures.
\end{definition}

\begin{figure}
\centering
\includegraphics[scale=0.55]{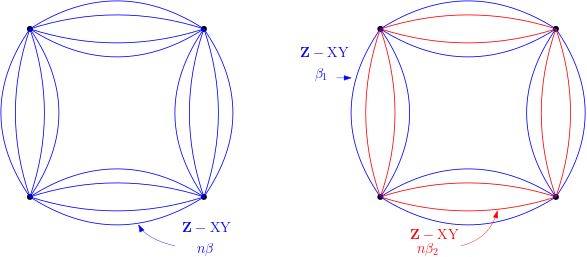}
\caption{The first two models of XY height functions introduced in Section~\ref{sec:section4.1}.} \label{fig:XYheight}
\end{figure}

\begin{figure}
\centering
\includegraphics[scale=0.5]{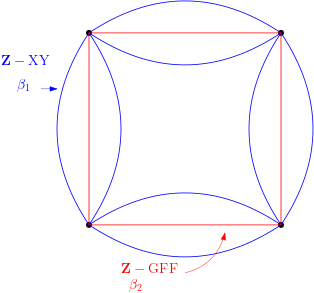}
\caption{The $\Z$-XY/GFF height function on the multigraph $\Z^2_{3-\mathrm{mult}}$.} \label{fig:XYGFFheight}
\end{figure}

\subsubsection{Monotonicity in the conductances}

In this section, we collect an important property of the integer-valued Gaussian free field and the XY height function: the monotonicity of the variance of the height in the conductances. We refer to~\cite{regev2017inequality, van2023duality, AHPS, frohlich1978correlation} for a proof of this result (and more general statements) and mention that this statement can be compared to the Ginibre inequality (Theorem~\ref{theoremGinibreineq}).

\begin{theorem}[Monotonicity in the conductances~\cite{regev2017inequality, van2023duality, AHPS, frohlich1978correlation}] \label{prop:monotonicity}
%\margin{C: maybe Frohlich Park here ?}
For any finite set $\Lambda \subseteq \Z^2$, any collection of conductances $J , \tilde J \in [0, \infty)^{E(\Lambda)}$ and any $x \in \Z^2$, we have
\begin{equation} \label{ineq:condIVGFF}
    \forall \{ y , y' \} \in E(\Lambda), ~ J_{yy'} \leq \tilde J_{yy'} ~ \implies ~ \mathrm{Var}_{\Lambda , J}^{\Z-\mathrm{GFF}} \left[ \varphi(x) \right] \leq \mathrm{Var}_{\Lambda , \tilde J}^{\Z-\mathrm{GFF}} \left[ \varphi(x) \right].
\end{equation}
Similarly, for any $n \in \N^*$ and any collections conductances $J , \tilde J \in [0, \infty)^{E(\Lambda) \times \{1 , \ldots , n\}}$, we have
\begin{equation} \label{ineq:condXYheight}
    \forall \{ x , y \} \in E(\Lambda), ~\forall k \in \{1 , \ldots, n\}, ~ J_{xy,k} \leq \tilde J_{xy,k} ~ \implies ~ \mathrm{Var}_{\Lambda , J}^{\Z-\mathrm{XY}} \left[ \varphi(x) \right] \leq \mathrm{Var}_{\Lambda , \tilde J}^{\Z-\mathrm{XY}} \left[ \varphi(x) \right].
\end{equation}
\end{theorem}

\begin{remark}
Let us make three remarks about this result
\begin{itemize}
    \item This property holds when the integer-valued Gaussian free field and the XY height function are defined on more general graphs/multigraphs.
    \item If $\Lambda , \Lambda'$ are two sets such that $\Lambda' \subseteq \Lambda \subseteq \Z^2_n$ and $J$ is a collection of conductances defined on the edges of $\Lambda$, then we have the inequality (N.B. On the left-hand side, we denote by $J'$ the restrictions of $J$ to the edges of $\Lambda'$)
    \begin{equation*}
        \mathrm{Var}_{\Lambda' , J'}^{\Z-\mathrm{XY}} \left[ \varphi(0) \right] \leq \mathrm{Var}_{\Lambda , J}^{\Z-\mathrm{XY}} \left[ \varphi(0) \right] ~~\mbox{and}~~\mathrm{Var}_{\Lambda' , J'}^{\Z-\mathrm{GFF}} \left[ \varphi(0) \right] \leq \mathrm{Var}_{\Lambda , J}^{\Z-\mathrm{GFF}} \left[ \varphi(0) \right].
    \end{equation*}
    \item A similar result would hold for the $\Z$-XY/GFF height function with general conductances (see Proposition~\ref{prop:approxIVGFFbyXY} below).
\end{itemize}
\end{remark}

\subsubsection{Wells' inequality for the XY height functions} \label{sec:defXYGFFheightfunction}

We next state the Wells' inequality for height functions. This inequality is one of the main ingredients in the proof of Theorem~\ref{Thm:delocheightsupercrit}. Its proof can be found in Section~\ref{App.A} and relies on the duality between spin systems and height functions investigated by the second author and Lis in~\cite{van2023duality}. Before stating the result, we introduce a definition for the XY height function and the integer-valued Gaussian free field on a site percolation configuration.

\begin{definition}[XY height functions and integer-valued Gaussian free field on a site percolation configuration] 
Let us fix an integer $L \in \N$, a percolation configuration $r \in \{0,1\}^{\Lambda_L}$, and define four models of height functions:
\begin{itemize}
    \item Given an inverse temperature $\beta \geq 0$, we define 
    \begin{equation} \label{def:XYgeneralpercsite}
    \mu_{\Lambda_L , \beta,r}^{\Z-\mathrm{XY}} (\left\{ \varphi \right\}) := \frac{1}{Z_{\Lambda_L,\beta,r}^{\Z-\mathrm{XY}} } \prod_{x \sim y} I_{\left(\varphi(x) - \varphi(y)\right)} \left( r_x r_y \beta \right) .
\end{equation}
We denote by $\mathrm{Var}_{L , \beta , r}^{\Z-\mathrm{XY}}$ the variance with respect to the measure $\mu_{\Lambda_L , \beta, r}^{\Z-\mathrm{XY}}$.
    \item Given an inverse temperature $\beta \geq 0$, we define
\begin{equation*}
    \mu_{\Lambda_L , \beta,r}^{\Z-\mathrm{GFF}} (\left\{ \varphi \right\}) := \frac{1}{Z_{\Lambda_L,\beta,r}^{\Z-\mathrm{GFF}} } \exp \left( - \frac{1}{2\beta} \sum_{x \sim y} \frac{\left( \varphi(x) - \varphi(y) \right)^2}{r_x r_y} \right).
\end{equation*}
We denote by $\mathrm{Var}_{L , \beta , r}^{\Z-\mathrm{GFF}}$ the variance with respect to the measure $\mu_{\Lambda_L , \beta, r}^{\Z-\mathrm{GFF}}$.
    \item Given an integer $n \geq 2$ and two inverse temperatures $\beta_1 , \beta_2 \geq 0$, we define
\begin{equation} \label{def:IVGFFmultigeneral}
    \mu_{\Lambda_L , n, \beta_1, \beta_2, r}^{\Z-\mathrm{XY}} (\left\{ \varphi \right\}) := \frac{1}{Z_{\Lambda_L , n, \beta_1, \beta_2, r}^{\Z-\mathrm{XY}} }  \prod_{x \sim y}  \left[ I_{\left(\varphi(x) - \varphi(y)\right)} \left( r_x \beta_1 \right) I_{\left(\varphi(x) - \varphi(y)\right)} \left( r_y \beta_1 \right) \left(I_{\left(\varphi(x) - \varphi(y)\right)} \left(  n \beta_2 \right)\right)^{n-2} \right]. 
\end{equation}
We denote by $\mathrm{Var}_{L , n, \beta_1, \beta_2, r}^{\Z-\mathrm{XY}}$ the variance with respect to the measure $\mu_{\Lambda_L , n, \beta_1, \beta_2, r}^{\Z-\mathrm{XY}}$.
\item Given two inverse temperatures $\beta_1 , \beta_2 \geq 0$, we define
   \begin{equation} \label{eq:defXYGFFperco}
        \mu_{\Lambda_L, \beta_1 , \beta_2 ,r}^{\Z-\mathrm{XY}/\mathrm{GFF}}(\left\{ \varphi \right\}) := \frac{1}{Z_{\Lambda_L, \beta_1 , \beta_2 }^{\Z-\mathrm{XY}/\mathrm{GFF}}} \prod_{x \sim y}  I_{\left(\varphi(x) - \varphi(y)\right)} \left(r_x  \beta_1 \right)  I_{\left(\varphi(x) - \varphi(y)\right)} \left(r_y  \beta_1 \right)  \exp \left( - \frac{\left( \varphi(x) - \varphi(y) \right)^2}{2\beta_2} \right).
    \end{equation}
    We denote by $\mathrm{Var}_{L , \beta_1 , \beta_2 , r}^{\Z-\mathrm{XY/GFF}}$ the variance with respect to the measure $\mu_{\Lambda_L , \beta_1 , \beta_2, r}^{\Z-\mathrm{XY/GFF}}$.
\end{itemize} 
\end{definition}

\begin{remark}
Let us make two remarks about the previous definition:
\begin{itemize}
\item In all the cases, the effect of the percolation configuration on the measure is the same: if $r_x = 0$ for some $x \in \Lambda_L$, then the height $\varphi(x)$ is constrained to be equal to the height of all its neighbours $\{ \varphi(y) \, : \,  y \sim x \}$.
%\margin{C: this is not quite what happens for "Heisenberg" disorder.}
\item  We note that, in the case of~\eqref{def:IVGFFmultigeneral} and~\eqref{eq:defXYGFFperco}, there are many equivalent definitions for the same measure (e.g., the two terms $r_x$ and $r_y$ could be added in any of the arguments of the three Bessel functions on the right-hand side of~\eqref{def:IVGFFmultigeneral}). The choice of this specific convention (i.e., $r_x$ in the argument of the first Bessel function and $r_y$ in the argument of the second) is motivated by the statement of the Wells' inequality. Specifically, the value of the partition function in~\eqref{def:IVGFFmultigeneral}
\begin{equation*}
Z_{\Lambda_L , n, \beta_1, \beta_2, r}^{\Z-\mathrm{XY}} = \sum_{\varphi \in \Omega_L^\Z}\prod_{x \sim y}  \left[ I_{\left(\varphi(x) - \varphi(y)\right)} \left( r_x \beta_1 \right) I_{\left(\varphi(x) - \varphi(y)\right)} \left( r_y \beta_1 \right) \left(I_{\left(\varphi(x) - \varphi(y)\right)} \left(  n \beta_2 \right)\right)^n \right]
\end{equation*} 
is affected by the convention, and this quantity is needed to define the Wells' disorder and to state the Wells' inequality below.
\end{itemize}
\end{remark}

We next introduce the Wells'  disorder associated with the two measures~\eqref{def:XYgeneralpercsite} and~\eqref{def:IVGFFmultigeneral}. This definition and result can be compared to the ones for the XY model stated in Definition~\ref{def.finite-volHamiltonian}.

\begin{definition}[Wells' disorder]
Given an integer $L \in \N$, an integer $n \geq 2$, we define the two Wells' disorders: 
\begin{itemize}
    \item Given an inverse temperature $\beta \geq 0$, we define the Wells' disorder associated with the model~\eqref{def:XYgeneralpercsite}: for any $r \in \{0,1\}^{\Lambda_L},$
    \begin{equation*}
        \nu_{L , \beta} \left( \left\{ r \right\} \right) := \frac{Z_{\Lambda_L,\beta, r}^{\Z-\mathrm{XY}}}{\bar{Z}_{\Lambda_L,\beta}^{\Z-\mathrm{XY}}} ~~\mbox{with}~~ \bar{Z}_{\Lambda_L,\beta}^{\Z-\mathrm{XY}} := \sum_{r \in \{0 , 1 \}^{\Lambda_L} } Z_{\Lambda_L,\beta, r}^{\Z-\mathrm{XY}}.
    \end{equation*}
    \item Given two inverse temperatures $\beta_1 , \beta_2 \geq 0$,  we define the Wells' disorder associated with the model~\eqref{def:IVGFFmultigeneral}: for any $r \in \{0,1\}^{\Lambda_L},$
    \begin{equation*}
        \nu_{L , n, \beta_1 , \beta_2} \left( \left\{ r \right\} \right) := \frac{Z_{\Lambda_L , n, \beta_1 , \beta_2, r}^{\Z-\mathrm{XY}}}{\bar{Z}_{\Lambda_L , n, \beta_1 , \beta_2}^{\Z-\mathrm{XY}}} ~~\mbox{with}~~ \bar{Z}_{\Lambda_L , n, \beta_1 , \beta_2}^{\Z-\mathrm{XY}} := \sum_{r \in \{0 , 1 \}^{\Lambda_L} } Z_{\Lambda_L , n, \beta_1 , \beta_2, r}^{\Z-\mathrm{XY}}.
    \end{equation*}
\end{itemize}
\end{definition}

We are now able to state the main result of this section: the Wells' inequality for the XY height function.

\begin{proposition}[Wells' inequality for the XY height functions] \label{prop:Wellsforheight}
For any $L \in \N$, any $x \in \Lambda_L$ and any $\beta \geq 0$, one has the inequality
\begin{equation*}
     \E_{\nu_{L , \beta}} \left[ \var_{L, \beta , r}^{\Z-\mathrm{XY}}\left[ \varphi(x) \right] \right] \geq \var_{L, \beta/4}^{\Z-\mathrm{XY}}\left[ \varphi(x) \right].
\end{equation*}
Similarly, for any $L \in \N$, any $n \geq 2$ and any $\beta_1 , \beta_2 \geq 0$,
\begin{equation*}
    \E_{\nu_{L , n , \beta_1, \beta_2}} \left[ \var_{L, n , \beta_1, \beta_2 , r}^{\Z-\mathrm{XY}} \left[ \varphi(x) \right] \right] \geq \var_{L , n , \beta_1/2, \beta_2}^{\Z-\mathrm{XY}}\left[ \varphi(x) \right].
\end{equation*}
\end{proposition}

\begin{remark} \label{rem:remark17}
It would be possible to generalise this inequality by considering more general underlying graphs (or mutligraphs), more general conductances and more general observables than the height $\varphi(0)$. Specifically, versions of this inequality should hold on any finite graph (or multigraph), 
 for any collection of (nonnegative) conductances and for any linear combination of the heights. To ease the presentation, we decided to only state and prove the result in the cases which are required for the proof of Theorem~\ref{Thm:delocheightsupercrit}.
\end{remark}

\subsubsection{The integer-valued Gaussian free field as a limit of XY height functions}

%\margin{C: I think I learned this fact from Piet, maybe we should make a comment about this. (Even if I agree this is pretty clear). \textbf{Piet, maybe it was rather the fact a large product of Villain in parallel gives a XY model !!}. Should we add a sentence mentioning this, maybe we should right ?}
%\margin{D: Yes, I was at this discussion and you learned that you can also approximate XY as a limit of Villain. 
%The height function convergence XY -> Villain is probably more classical, and was mentioned at least in my paper with Marcin.}
In this section, we show that the integer-valued Gaussian free field can be obtained as a limit of XY height functions. This result can be compared to Proposition~\ref{prop:prop3.6} which asserts that the Villain model can be obtained as a limit of XY models. We refer to Figure~\ref{fig:convergencetoIVGFF} for an illustration of this result.

\begin{proposition} \label{prop:approxIVGFFbyXY}
 %and any pair of inverse temperatures $\beta \in (0 , \infty)$ , one has the convergence 
    %\begin{equation*}
     %   \lim_{n \to \infty} \mathrm{Var}_{L, n, \beta}^{\Z-\mathrm{XY}} \left[ \varphi(x) \right] =  \mathrm{Var}_{L, \beta} \left[ \varphi(x) \right].
    %\end{equation*}
    For any sidelength $L \in \N$, any $x \in \Lambda_L$, any triplet of inverse temperatures $\beta, \beta_1 , \beta_2 \in (0 , \infty)$ and any percolation configuration $r \in \{ 0,1\}^{\Lambda_L}$,
    \begin{equation*}
        \lim_{n \to \infty} \mathrm{Var}_{L, n, \beta, r}^{\Z-\mathrm{XY}} \left[ \varphi(x) \right] = \mathrm{Var}_{L, \beta, r}^{\Z-\mathrm{GFF}} \left[ \varphi(x) \right]
    \end{equation*}
    and
    \begin{equation*}
        \lim_{n \to \infty}  \mathrm{Var}_{L, n, \beta_1 , \beta_2, r}^{\Z-\mathrm{XY}} \left[ \varphi(x) \right] =   \mathrm{Var}_{L, \beta_1 , \beta_2, r}^{\Z-\mathrm{XY/GFF}} \left[ \varphi(x) \right].
    \end{equation*}
\end{proposition}

\begin{remark}
Let us make two remarks about the previous statement:
\begin{itemize}
\item To be more precise, there is a convergence of the measures $\mu_{\Lambda_L, n,  \beta_1, \beta_2,r}^{\Z-\mathrm{XY}}$ toward $\mu_{\Lambda_L , \beta_1, \beta_2 , r}^{\Z-\mathrm{XY/GFF}}$ .
\item This type of convergence is in fact quite general, and this result could be extended to models of height functions defined on more general graphs and with more general conductances.
\end{itemize}
\end{remark}

\begin{figure}
\centering
\includegraphics[scale=0.42]{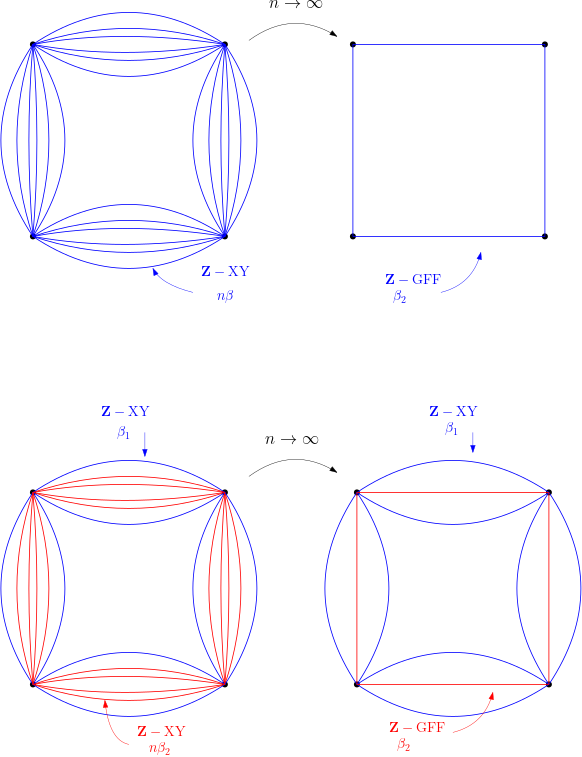}
\caption{The convergence of the XY height function toward the integer-valued Gaussian free field and the $\Z$-XY/GFF height function.} \label{fig:convergencetoIVGFF}
\end{figure}

\begin{proof}
The proof relies on standard techniques, we thus only provide a brief sketch of the argument. The result is a consequence of the following convergence: for any $\beta > 0$ and any $k \in \Z$,
\begin{equation*}
    \left( \frac{I_k(n \beta)}{I_0(n \beta)} \right)^n \underset{n \to \infty}{\longrightarrow} e^{-\frac{k^2}{2 \beta}},
\end{equation*}
which is itself a consequence of the expansion~\eqref{eq:expansionBessel} through the following computation, for any $k \in \Z$,
\begin{equation*}
   \left( \frac{I_k(n \beta)}{I_0(n \beta)} \right)^n = \frac{\left( 1 - \frac{4k^2-1}{8 n\beta} + O\left( \frac{1}{n^2}\right) \right)^n}{\left( 1 + \frac{1}{8 n\beta} + O\left( \frac{1}{n^2}\right) \right)^n} = e^{-\frac{k^2}{2 \beta}} + O\left(\frac{1}n\right).
\end{equation*}
\end{proof}

Proposition~\ref{prop:approxIVGFFbyXY} can be combined with Theorem~\ref{prop:monotonicity} to deduce the following result.

\begin{corollary} \label{cor:corollary4.9}
The variances $\mathrm{Var}_{L , \beta,r}^{\Z-\mathrm{GFF}} \left[ \varphi(x) \right],$ $\mathrm{Var}_{L,  \beta,r}^{\Z-\mathrm{XY}}\left[ \varphi(x) \right] $ and $\mathrm{Var}_{L, \beta_1 , \beta_2,r}^{\Z-\mathrm{XY}/\mathrm{GFF}}\left[ \varphi(x) \right] $ are increasing functions of:
\begin{itemize}
    \item the sidelength $L \in \N$,
    \item the inverse temperatures $\beta, \beta_1, \beta_2 \in (0 , \infty)$,
    \item the percolation configuration $r \in \{0 , 1 \}^{\Lambda_L}$.
\end{itemize}
\end{corollary}

\begin{remark}

The expansion~\eqref{eq:expansionBessel} implies the convergence, for any $k \in \Z$,
    \begin{equation*}
        \frac{I_k(J)}{I_0(J)} \underset{J \to \infty}{\longrightarrow} 1.
    \end{equation*}
    From this observation, one can deduce the following convergence, for any percolation configuration $r \in \{0 ,1 \}^{\Lambda_L},$
    \begin{equation*}
         \mathrm{Var}_{L, \beta_1 , \beta_2,r}^{\Z-\mathrm{XY}/\mathrm{GFF}}\left[ \varphi(x) \right] \underset{\beta_1 \to \infty}{\longrightarrow} \mathrm{Var}_{L, \beta_2,r}^{\Z-\mathrm{GFF}}\left[ \varphi(x) \right].
    \end{equation*}
    Combining this result with Proposition~\ref{prop:monotonicity}, we deduce that, for any pair of inverse temperatures $\beta_1 , \beta_2 \geq 0$,
\begin{equation} \label{eq:ineq4.11}
    \mathrm{Var}_{L , \beta_1 , \beta_2,r}^{\Z-\mathrm{XY/GFF}} \left[ \varphi(x) \right] \leq \mathrm{Var}_{L , \beta_2,r}^{\Z-\mathrm{GFF}} \left[ \varphi(x) \right].
\end{equation}
\end{remark}

\subsubsection{Phase transition for the integer-valued Gaussian free field}

The next ingredient we need is the following fundamental result of Fr\"{o}hlich-Spencer~\cite{frohlich1981kosterlitz} and Lammers~\cite{lammers2022height}, which shows the existence of delocalised regime for the two-dimensional integer-valued Gaussian free field and the $\Z$-XY/GFF height function.

\begin{theorem}[Phase transition for the XY integer-valued height function~\cite{frohlich1981kosterlitz, lammers2022height}] \label{prop:delocXY/IVGFFheight}
There exist an inverse temperature $\beta_{\mathrm{Deloc}} \in (0 , \infty)$ and a constant $c_0 > 0$ such that 
\begin{itemize}
    \item For any $\beta \geq \beta_{\mathrm{Deloc}}$ and any $L \in \N^*$
    \begin{equation*}
         \mathrm{Var}_{L , \beta}^{\Z-\mathrm{GFF}} \left[ \varphi(0) \right] \geq c_0 \ln L.
    \end{equation*}
    \item For any $\beta_1 , \beta_2 \geq \beta_{\mathrm{Deloc}}$ and any $L \in \N^*$
    \begin{equation*}
         \mathrm{Var}_{L , \beta_1 , \beta_2}^{\Z-\mathrm{XY}/\mathrm{GFF}} \left[ \varphi(0) \right] \geq c_0 \ln L.
    \end{equation*}
\end{itemize}
\end{theorem}

\begin{remark}
It can be proved (using a Peierls argument~\cite{peierls1936ising}) that, when the parameters $\beta, \beta_1$ and $\beta_2$ are sufficiently small, the variances of the height $\varphi(0)$ are bounded uniformly in $L$.
\end{remark}

\subsection{Wells' inequality for height functions} \label{App.A}

In this section, we prove the Wells' inequality for height functions (Proposition~\ref{prop:Wellsforheight}). The proof of the result makes use of a duality between height functions and spin systems from~\cite{van2023duality} which is presented in more details in Section~\ref{section:defgeneralXY} below. Once equipped with this duality, the proof of the Wells' inequality follows from an adaptation of the original technique of Wells~\cite{Wellsthesis, bricmont1981periodic} and the details of the argument can be found in Sections~\ref{sectionWellsineq1} and~\ref{sectionWellsineq2}.

We mention that, in order to simplify the presentation, we will only prove the result in the case of the model~\eqref{def:IVGFFmultigeneral} (i.e., we will prove the second inequality of Proposition~\ref{prop:Wellsforheight}), but that the same argument applies to a much more general class of models %\margin{C: Much more general ? (P: n'importe quel graphe fini, n'importe quelle combinaison !  }
(N.B. in particular, the proof of the first inequality of Proposition~\ref{prop:Wellsforheight} follows from the same argument). In all this section, we fix an integer $n \in \N$ and work with the multigraph $\Z^2_{n\mathrm{-mult}}$.

\subsubsection{Duality between spin systems and height functions} \label{section:defgeneralXY}

We first recall some of the results of~\cite{van2023duality}. We will only collect here the results of~\cite{van2023duality} which are needed in order to establish Wells' inequality for height functions (and we adapted the formalism of~\cite{van2023duality} to match the one of this article), but we mention that more precise information and results on the duality between height functions and spin systems can be found in this article.

We first start with the definition of a model of integer-valued height-function which is similar to the one introduced in Definition~\ref{def:XYheightfunction} with an additional weight on the height $\varphi(0)$.

\begin{definition}[XY height functions]\label{def:integervaluedXYwithlambda}
We fix a sidelength $L \in \N$, a parameter $\lambda >0$, a collection of non-negative conductances $\left( J_{xy,k} \right)_{\{x,y\} \in E(\Lambda_L), 1 \leq k \leq n}$, and define the probability distribution on the set $\Omega_{\Lambda_L}^\Z$
\begin{equation*}
    \mu_{\Lambda_L , J, \lambda}^{\Z\mathrm{-XY}} (\left\{ \varphi \right\}) := \frac{1}{Z_{\Lambda_L , n , J, \lambda}^{\Z\mathrm{-XY}}}  I_{\left(\varphi(0)\right)} \left( \lambda \right)\prod_{x \sim y} \prod_{1 \leq k \leq  n} I_{\left(\varphi(x) - \varphi(y)\right)} \left( J_{xy, k} \right) .
\end{equation*}
We denote by $\mathrm{Var}_{L , J , \lambda}^{\Z-\mathrm{XY}}$ the variance with respect to the measure $\mu_{\Lambda_L , J, \lambda}^{\Z\mathrm{-XY}}$.
\end{definition}

\begin{remark} \label{rem:remark23}
When $\lambda \to \infty$, the measure $\mu_{\Lambda_L , J, \lambda}^{\Z-\mathrm{XY}}$ converges weakly toward the measure $\mu_{\Lambda_L , J}^{\Z-\mathrm{XY}}$ introduced in Definition~\ref{def:XYheightfunction}.
\end{remark}

\begin{figure}
\centering
\includegraphics[scale=0.5]{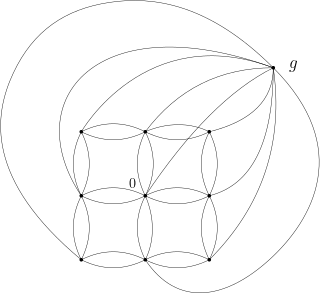}
\caption{The box $\Lambda_{L}^g(n)$ with $L = 1$ and $n = 2$. All the edges on the boundary of the box and the vertex $0$ are connected to the ghost (in the case $L=1$, all the vertices are connected to the ghost).} \label{fig:boxwithghost}
\end{figure}

We next introduce the spin system which is dual to the height function of the previous definition (see the duality formulae of Proposition~\ref{prop:propA3} below). This spin system is a generalisation of the XY model introduced in Definition~\ref{generalXY} (specifically, the angle configurations are not defined on the vertices but on the oriented edges). It is defined as follows:
\begin{itemize}
    \item Let $\Lambda_{L}^g(n)$ be the graph $\Lambda_{L}(n)$ to which a ghost vertex $g$ connected to the boundary of the box and to the vertex $0$ has been added (see Figure~\ref{fig:boxwithghost}). We denote by $E(\Lambda_{L}^g(n))$ and $\vec{E}\left( \Lambda_L^g(n) \right)$ the sets of the unoriented and oriented edges of this graph.  We refer to an oriented edge $\vec{e}$ by giving its two endpoints together with an integer of $\{1 , \ldots, n\}$ (if the ghost is not one of the endpoints). Given a vertex $x \in \Lambda_L \cup \{g\}$, we denote by $\vec{E}_x\left( \Lambda_L^g(n) \right)$ the set of edges of $\vec{E}\left( \Lambda_L^g(n) \right)$ whose first endpoint is $x$.
    \item We then introduce the following spaces of angle configurations:
\begin{equation*}
    \Omega_{\Lambda_L}^{\mathbb{S}^1}  := \left\{ \theta : \vec{E}(\Lambda_{L}^g(n)) \to [0 , 2\pi) \, : \, \forall \vec{e} \in \vec{E}(\Lambda_{L}^g(n)), \, \theta_{-\vec{e}} = - \theta_{\vec{e}}  ~ \mathrm{mod} \, 2\pi \right\}
\end{equation*}
and
\begin{equation} \label{def:divS1group}
    \Omega_{\Lambda_L}^{\mathrm{div-}\mathbb{S}^1}  := \left\{ \theta \in \Omega_{\Lambda_L}^{\mathbb{S}^1}\, : \,  \forall x \in \Lambda_L \cup \{ g \}, \, \sum_{\vec{e} \in \vec{E}_x\left( \Lambda_L^g(n) \right)} \theta_{\vec{e}} = 0 ~ \mathrm{mod} \, 2\pi \right\}.
\end{equation}
    \item We consider an (arbitrary) spanning tree $T \subseteq \Lambda_{L}^g(n)$ and set
    \begin{equation*}
        \Omega_T^{\mathbb{S}^1} := \left\{ \theta : \vec{E}(\Lambda_L^g(n)) \setminus \vec{E}(T) \to [0 , 2\pi) \, ; \, \forall \vec{e} \in \vec{E}(\Lambda_{L}^g(n)), \, \theta_{-\vec{e}} = - \theta_{\vec{e}} ~ \mathrm{mod} \, 2\pi\right\}.
    \end{equation*} 
    Note that this space can be identified with the space of functions $\{ \theta : E(\Lambda_L^g(n)) \setminus E(T) \to [0 , 2\pi) \}$. Following~\cite{van2023duality}, we note that there exists a bijection between the spaces
     \begin{equation*}
        \Omega_T^{\mathbb{S}^1} ~~\mbox{and}~~ \Omega_{\Lambda_L}^{\mathrm{div-}\mathbb{S}^1}.
    \end{equation*} 
    More specifically, any collection of angles $\theta \in \Omega_T^{\mathbb{S}^1}$ can be uniquely extended into an element of $\Omega_{\Lambda_L}^{\mathrm{div-}\mathbb{S}^1}$.
    We denote this bijection by $A_T : \Omega_T^{\mathbb{S}^1} \to \Omega_{\Lambda_L}^{\mathrm{div-}\mathbb{S}^1}$. 
    \item The two spaces $\Omega_{\Lambda_L}^{\mathrm{div-}\mathbb{S}^1}$ and $\Omega_T^{\mathbb{S}^1}$ can be equipped with an additive group structure which makes them compact Abelian groups. For later purposes, we mention that the map $A_T$ is then a group isomorphism and that it satisfies the following property: for any edge $\vec{e}_0 \in \vec{E}\left( \Lambda_L^g(n) \right)$ and any angle configuration $\theta := (\theta_{\vec{e}})_{\vec{e} \in \vec{E}(\Lambda_L^g(n)) \setminus \vec{E}(T)} \in \Omega_T^{\mathbb{S}^1}$, the angle $\left( A_T \theta \right)_{\vec{e}_0}$ can be written as a linear combination with integer-valued coefficients of the angles $(\theta_{\vec{e}})_{_{\vec{e} \in \vec{E}(\Lambda_L^g(n)) \setminus \vec{E}(T)}}.$
    \item We endow the space of functions $\{ \theta : E(\Lambda_L^g(n)) \setminus E(T) \to [0 , 2\pi) \}$, and thus the space $\Omega_T^{\mathbb{S}^1}$, with the product probability measure 
    \begin{equation}\label{def:productHaarmeasure}
        \prod_{e \in E(\Lambda_L^g(n)) \setminus E(T)} (2\pi)^{-1} d \theta_{e}.
    \end{equation}
    We then equip the space $\Omega_{\Lambda_L}^{\mathrm{div-}\mathbb{S}^1}$ with the pushforward of this product measure by the map $A_T$. We will denote this probability measure by $d\theta$ in the definitions and proofs below. This measure is invariant under translations, it is thus a Haar measure on the Abelian group~\eqref{def:divS1group}; since the Haar measure is unique (up to multiplicative constant), we deduce that the definition of the measure $d \theta$ does not depend on the choice of the spanning tree $T$.
\end{itemize}

\begin{figure}
\centering
\includegraphics[scale=0.6]{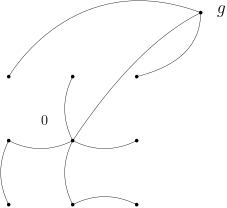}
\caption{A spanning tree of the box $\Lambda_{L}^g(n)$ (with $L = 1$ and $n = 2$).} \label{fig:spanningtree}
\end{figure}

Equipped with all these definitions, we introduce a generalised version of the XY model.
\begin{definition}[Generalised XY model] \label{def:generalversionXY}
    Given an integer $L \in \N$, a collection of non-negative conductances $\left( J_{xy,k} \right)_{\{ x, y\} \in E(\Lambda_L), 1 \leq k \leq n}$, and a parameter $\lambda > 0$, we define the generalised XY model to be the probability distribution on the set $\Omega_{\Lambda_L}^{\mathrm{div-}\mathbb{S}^1}$ given by the identity
    \begin{equation} \label{eq:defgeneralisedXY}
        \mu_{\Lambda_L , J, \lambda}^{\mathrm{XY}} (d \theta) := \frac{1}{Z_{\Lambda_L, J, \lambda}^{\mathrm{XY}}} \exp \left( \lambda \cos(\theta_{0g}) \right) \exp \left( \sum_{x \sim y} \sum_{1 \leq k \leq n} J_{xy,k} \cos(\theta_{xy, k}) \right) d \theta.
    \end{equation}
    We denote by $\left\langle \cdot \right\rangle_{L , J, \lambda}^{\mathrm{XY}}$ the expectation with respect to this measure.
\end{definition}
\begin{remark}
There is a slight abuse of notation in this definition because the conductances are assumed to be defined on non-oriented edges, the angles are defined on oriented edges and we used the same index for both terms (N.B. The orientation of the edges does not matter here because of the parity of the cosine function).
\end{remark}

Equipped with the definitions for the height functions and the generalised XY model, we are able to state a duality formula relating the two models. The proof of this identity can be found in~\cite[Lemma 2.2]{van2023duality}. % (N.B. in~\eqref{eq:idpartfunctgenera} below, there is no multiplicative factor involving $2\pi$ on the left-hand side; this is because this term has been incorporated in the measure~\eqref{def:productHaarmeasure}).

\begin{proposition}[Duality, Lemma 2.2 of~\cite{van2023duality}] \label{prop:propA3}
For any integer $L \in \N$, any parameter $\lambda > 0$ and any collection of non-negative conductances $J:= \left( J_{xy,k} \right)_{\{x,y\} \in E(\Lambda_L), 1 \leq k \leq n}$, one has the identities:
\begin{itemize}
    \item For the partition functions:
    \begin{equation} \label{eq:idpartfunctgenera}
        Z_{\Lambda_L, J, \lambda}^{\Z-\mathrm{XY}} = Z_{\Lambda_L,  J, \lambda}^{\mathrm{XY}}.
    \end{equation}
    \item For the variance of the height:
    \begin{equation*}
    \mathrm{Var}_{L  ,J , \lambda}^{\Z- \mathrm{XY}} \left[ \varphi(0) \right] = \left\langle \lambda \cos \left( \theta_{0g} \right) + \frac{\lambda^2}{2} \cos \left( 2 \theta_{0g} \right)  - \frac{\lambda^2}{2}  \right\rangle_{L , J, \lambda}^{\mathrm{XY}}.
    \end{equation*}
\end{itemize}

\begin{remark}
To be precise, the result is proved in~\cite[Lemma 2.2]{van2023duality} for general graphs, but the extension to multigraph (considered here) is only notational.
%\margin{\blue{Paul: I hope I did not miss something here}} 
%\margin{D: indeed, our result works for all graphs (in particular multigraphs). Either you just allow mutligraphs as being graphs, 
%and then our argument works, or if you don't like that, a multigraph is a ``gluing'' limit of proper graphs.}
\end{remark}

\end{proposition}

\subsubsection{Wells' inequality for the generalised XY models} \label{sectionWellsineq1}

In this section, we extend the range of application of the Wells' inequality from the XY model (as stated in Proposition~\ref{prop.Wells}) to the generalised XY model introduced in Definition~\ref{def:generalversionXY}. As mentioned at the beginning of the section, and in order to simplify the presentation, we will only prove the result in the case of the model~\eqref{def:IVGFFmultigeneral}. We thus specify the notation of the XY height function and generalised XY model in the following definition.

 \begin{definition}[XY height function and generalised XY model on a percolation configuration]
 Given an integer $L \in \N$, an integer $n \geq 2$, a parameter $\lambda > 0$, two inverse temperatures $\beta_1 , \beta_2 \geq 0$, and a site percolation configuration $r \in \{ 0 , 1\}^{\Lambda_L}$, we define
 \begin{itemize}
    \item The XY height function
  \begin{multline*}
     \mu_{L , n ,\beta_1, \beta_2 , \lambda,r }^{\Z - \mathrm{XY}} (\left\{ \varphi \right\}) \\
     := \frac{1}{Z_{\Lambda,J}^{\Z-\mathrm{XY}} }  I_{\left(\varphi(0)\right)} \left( \lambda \right)\prod_{x \sim y}  \left[ I_{\left(\varphi(x) - \varphi(y)\right)} \left( r_x \beta_1 \right) I_{\left(\varphi(x) - \varphi(y)\right)} \left( r_y \beta_1 \right) I_{\left(\varphi(x) - \varphi(y)\right)} \left(  n \beta_2 \right)^{n-2} \right].
 \end{multline*}
 We denote the variance associated with this measure by $\mathrm{Var}_{L , n ,\beta_1, \beta_2 , \lambda,r}^{\Z - \mathrm{XY}}$.
 \item The generalised XY model
 \begin{multline} \label{eq:defXYmodellambda}
     \mu_{L , n , \beta_1, \beta_2 , \lambda, r}^{\mathrm{XY}} (d \theta) \\
     := \frac{1}{Z_{\Lambda_L , n ,\beta_1, \beta_2 , \lambda, r}^{\mathrm{XY}}} \exp \left( \lambda \cos(\theta_{0g}) \right) \exp \left( \beta_1  \sum_{ x \sim y} \left( r_x \cos(\theta_{xy,1}) + r_y \cos(\theta_{xy,n}) \right) \right)  \\ \times \exp  \left( n \beta_2 \sum_{ x \sim y}  \sum_{2 \leq k \leq  n-1} \cos(\theta_{xy,k})  \right)  d \theta.
 \end{multline}
We denote by $\left\langle \cdot \right\rangle_{L , n , \beta_1, \beta_2 , \lambda, r}^{\mathrm{XY}}$ the expectation with respect to this measure.
\end{itemize}
In both cases when the percolation configuration is $r \equiv 1$ (i.e., all the sites are open), we omit the subscript $r$ from the notation.
 \end{definition}
 
 \begin{remark} 
 Since Proposition~\ref{prop:propA3} holds for any values of the conductance, it implies the identities of partition function, for any percolation configuration $r \in \{0,1\}^{\Lambda_L}$,
 \begin{equation} \label{eq:identitypartfunctperc}
 Z_{\Lambda_L , n ,\beta_1 , \beta_2, \lambda, r}^{\mathrm{XY}} = Z_{\Lambda_L , n ,\beta_1 , \beta_2, \lambda, r}^{\Z-\mathrm{XY}},
 \end{equation}
together with the identity
 \begin{equation} \label{eq:dualityonpercolation}
     \mathrm{Var}_{L , n ,\beta_1 , \beta_2, \lambda, r}^{\Z- XY} \left[ \varphi(0) \right] = \left\langle \lambda \cos \left( \theta_{0g} \right) + \lambda^2 \cos \left( \theta_{0g} \right)^2  - \lambda^2  \right\rangle_{L , n ,\beta_1 , \beta_2, \lambda, r}^{\mathrm{XY}}.
 \end{equation}
\end{remark}

We next introduce the Wells' disorder associated with the XY model~\eqref{eq:defXYmodellambda}, then state and prove the Wells' inequality for this model.

\begin{definition}[Wells' disorder]
 Given an integer $L \in \N$, an integer $n \geq 2$, a parameter $\lambda > 0$, two inverse temperatures $\beta_1 , \beta_2 \geq 0$, we define the Wells' disorder by the identity, for $r \in \{ 0 , 1\}^{\Lambda_L}$,
 \begin{equation*}
     \nu_{L , n , \beta_1 , \beta_2 , \lambda}(\{ r \}) := \frac{Z_{\Lambda_L , n , \beta_1 , \beta_2 , \lambda, r}^{\mathrm{XY}}}{\bar{Z}_{\Lambda_L , n , \beta_1 , \beta_2 , \lambda}^{\mathrm{XY}}} ~~\mbox{with}~~ \bar{Z}_{\Lambda_L , n , \beta_1 , \beta_2 , \lambda}^{\mathrm{XY}} := \sum_{r \in \{0, 1\}^{\Lambda_L}} Z_{\Lambda_L , n , \beta_1 , \beta_2 , \lambda, r}^{\mathrm{XY}}.
 \end{equation*}
\end{definition}

\begin{proposition}[Wells inequality the generalised XY model] \label{prop:generalisedWells}
    Given an integer $L \in \N$, an integer $n \geq 3$, two inverse temperatures $\beta_1 , \beta_2 \geq 0$, a parameter $\lambda > 0$ and an integer $m \in \N$, one has the inequality
    \begin{equation} \label{eq:WellsverygeneralXY}
        \left\langle \cos \left( m  \theta_{0g} \right) \right\rangle_{L , n , \beta_1/2 , \beta_2 , \lambda}^{\mathrm{XY}} \leq \E_{\nu_{L , n , \beta_1 , \beta_2 , \lambda}} \left[ \left\langle \cos \left( m \theta_{0g} \right) \right\rangle_{L , n , \beta_1 , \beta_2 , \lambda,r }^{\mathrm{XY}} \right].
    \end{equation}
\end{proposition}

\begin{remark}
Let us make two comments about this statement:
\begin{itemize}
    \item We will only need this result in the cases where $m = 1$ or $m=2$.
    \item As mentioned above, this inequality can be generalised in many directions: it is in particular possible to consider more general multigraphs, more general conductances and more general observables (i.e., the inequality should hold when considering the cosine of any linear combination with integer-valued coefficients of the angles).
\end{itemize}
\end{remark}

\begin{proof}
To simplify the presentation of the argument, we only prove the result in the case $m = 1$.
%\margin{C: if $m=2$ is really needed, isn't this a bit dangereous ? \textbf{Maybe we should add a small proof for $m=2$ as well. Diederik could you try or otherwise no problem I will try adding it just before uploading to Arxiv ?}}
We additionally introduce the notation:
for any $r \in \{0,1\}^{\Lambda_L}$ and any $\theta \in \Omega_{\Lambda_L}^{\mathrm{div-}\mathbb{S}^1}$,
\begin{equation*}
H_{\beta_1, \beta_2, \lambda, r} (\theta) :=
\beta_1 \sum_{\substack{ x, y \in \Lambda_L \\ x \sim y }}
\left( r_x \cos(\theta_{xy,1}) + r_y \cos(\theta_{xy,n}) \right) +
\beta_2 \sum_{\substack{ x, y \in \Lambda_L \\ x \sim y }}
\sum_{2 \leq k \leq  n-1} \cos(\theta_{xy,k}),
\end{equation*}
and omit the subscript $r$ when $r \equiv 1$.
We start from the computation
\begin{align*}
    \lefteqn{\E_{\nu_{L , n ,\beta_1 , \beta_2, \lambda}} \left[ \left\langle \cos \left( \theta_{0g} \right) \right\rangle_{L , n ,\beta_1 , \beta_2, \lambda, r}^{\mathrm{XY}} \right] } \qquad & \\ &
    = \frac{1}{\bar{Z}_{\Lambda_L , n ,\beta_1 , \beta_2, \lambda}^{\mathrm{XY}}} \sum_{r \in \{0,1\}^{\Lambda_L}} Z_{\Lambda_L , n ,\beta_1 , \beta_2, \lambda,r}^{\mathrm{XY}} \left\langle \cos \left( \theta_{0g} \right) \right\rangle_{L , n ,\beta_1 , \beta_2, \lambda, r}^{\mathrm{XY}} \\
    & = \frac{1}{\bar{Z}_{\Lambda_L , n ,\beta_1 , \beta_2, \lambda}^{\mathrm{XY}}} \sum_{r \in \{0,1\}^{\Lambda_L}} \int_{\Omega_{\Lambda_L}^{\mathrm{div-}\mathbb{S}^1}} \cos \left( \theta_{0g} \right)  \exp \left( H_{\beta_1, \beta_2, \lambda, r} (\theta) \right) d\theta.
\end{align*}
The proof is similar to the original argument of Wells~\cite{Wellsthesis} (see also~\cite{bricmont1981periodic} or \cite{DG25}) 
%\margin{C: I also added our \cite{DG25} paper here which gives a nice account I think.} 
with some adaptations to our setting; specifically, we must take into account that $ d\theta$ is not a product measure but the Haar measure on $\Omega_{\Lambda_L}^{\mathrm{div-}\mathbb{S}^1}$. We introduce duplicated variables and proceed by successive reduction. We first note that, in order to prove~\eqref{eq:WellsverygeneralXY}, it is enough to show
\begin{equation} \label{ineq:duplicatedvariables}
    \sum_{r \in \{0,1\}^{\Lambda_L}} \int_{\Omega_{\Lambda_L}^{\mathrm{div-}\mathbb{S}^1} \times \Omega_{\Lambda_L}^{\mathrm{div-}\mathbb{S}^1}} \left( \cos \left( \theta_{0g} \right) - \cos \left( \theta_{0g}' \right) \right) \exp \left(H_{\beta_1, \beta_2, \lambda, r} (\theta) + H_{\beta_1/2, \beta_2, \lambda} (\theta') \right)  d\theta d\theta' \geq 0.
\end{equation}
    The next step of the proof consists in expanding the exponentials on the left-hand side of~\eqref{ineq:duplicatedvariables} and to use the identity, for any $r \in \{0,1\}$ and any $\theta, \theta' \in [0,2\pi)$,
    \begin{equation*} 
        \left( r \cos (  \theta ) + \frac{1}{2} \cos (  \theta') \right)  = \frac{1}{2} \left( r  +  \frac{1}{2} \right) \left(  \cos ( \theta) + \cos ( \theta') \right)  + \frac{1}{2} \left( r -  \frac{1}{2} \right) \left(  \cos (\theta )  - \cos ( \theta') \right).
    \end{equation*}
    to reduce the problem to the proof of the following two inequalities (N.B. expanding the exponential in~\eqref{ineq:duplicatedvariables} involves quite a lot of terms and we decided not to write down all of them, but the argument is the same as in~\cite{Wellsthesis, bricmont1981periodic} or~\cite[Section 3]{DG25}): for any collection of vertices $x_1 , \ldots, x_k \in \Lambda_L$ (not necessarily distinct) and any collection of plus and minus signs
    \begin{equation} \label{eq:Wellswithdisroder}
        \sum_{r \in \{0,1\}^{\Lambda_L}} \prod_{i=1}^k \left( r_{x_i}\pm \frac{1}{2}  \right) \geq 0
    \end{equation}
    and, for any collection of oriented edges $\vec{e}_1, \ldots, \vec{e}_k \in \vec{E}\left( \Lambda_L^g(n) \right)$ (not necessarily distinct) and any collection of plus and minus signs
    \begin{equation} \label{eq:Wellswithangles}
         \int_{\Omega_{\Lambda_L}^{\mathrm{div-}\mathbb{S}^1} \times \Omega_{\Lambda_L}^{\mathrm{div-}\mathbb{S}^1}} \prod_{i=1}^k \left( \cos (  \theta_{\vec{e}_i} ) \pm \cos (  \theta'_{\vec{e}_i}) \right)  d\theta d\theta' \geq 0.
    \end{equation}
    We then prove the two inequalities~\eqref{eq:Wellswithdisroder} and~\eqref{eq:Wellswithangles} separately and start with~\eqref{eq:Wellswithdisroder}.
    
    \medskip
    
    \textit{Proof of the inequality~\eqref{eq:Wellswithdisroder}.}
    The inequality is a consequence of the following result: for any pair of integers $n_1, n_2 \in \N$,
    \begin{equation*}
        \sum_{r \in \{0 , 1\}} \left(r - \frac{1}{2}\right)^{n_1} \left(r + \frac{1}{2}\right)^{n_2} \geq 0.
    \end{equation*}
    This last inequality can be easily verified.
   % \margin{\bf Maybe more detail here, or at least discuss ?}
    
    \medskip
    
    \textit{Proof of the inequality~\eqref{eq:Wellswithangles}.} Regarding the inequality~\eqref{eq:Wellswithangles}, we will make use of the Ginibre inequality and the properties of the Haar measure $d \theta$ introduced in Section~\ref{section:defgeneralXY}. We let $T$ be an arbitrary spanning tree of the graph $\Lambda_L^g(n)$ and denote by $A_T : \Omega_T^{\mathbb{S}^1} \to \Omega_{\Lambda_L}^{\mathrm{div-}\mathbb{S}^1}$ the extension map. We thus have the identity
    \begin{multline} \label{eq:generalGinibreidentity}
     \int_{\Omega_{\Lambda_L}^{\mathrm{div-}\mathbb{S}^1} \times  \Omega_{\Lambda_L}^{\mathrm{div-}\mathbb{S}^1} } \prod_{i=1}^k \left( \cos (  \theta_{\vec{e}_i} ) \pm \cos (  \theta'_{\vec{e}_i}) \right) d\theta d\theta' \\
     =   \int_{\Omega_{T}^{\mathbb{S}^1} \times \Omega_{T}^{\mathbb{S}^1}}  \prod_{i=1}^k \left( \cos (  (A_T \theta)_{\vec{e}_i} ) \pm \cos (  (A_T \theta')_{\vec{e}_i}) \right)  \prod_{e \in E(\Lambda_L^g(n)) \setminus E(T)} (2 \pi)^{-2} d\theta_e d\theta'_e.
    \end{multline}
    We next make use of the following observation from~\cite{van2023duality} (already discussed in Section~\ref{section:defgeneralXY}): for any $i \in \{1, \ldots, k\}$, the angle $ \left( A_T \theta \right)_{\vec{e}_i}$ can be written as a linear combination with integer-valued coefficients of the angles $\left( \theta_{\vec{e}} \right)_{\vec{e} \in \vec{E}(\Lambda_L^g(n)) \setminus \vec{E}(T)}$. Together with the definition~\eqref{def:productHaarmeasure} of the measure $d \theta$, we deduce that there exist $m_1 , \ldots, m_{k} \in \Z^{E(\Lambda_L^g(n)) \setminus E(T)}$ such that
    \begin{multline*}
        \int_{\Omega_{\Lambda_L}^{\mathrm{div-}\mathbb{S}^1} \times  \Omega_{\Lambda_L}^{\mathrm{div-}\mathbb{S}^1}} \prod_{i=1}^k \left( \cos (  \theta_{\vec{e}_i} ) \pm \cos (  \theta'_{\vec{e}_i}) \right)  d\theta d\theta' = \\
        \int_{\Omega_{T}^{\mathbb{S}^1} \times \Omega_{T}^{\mathbb{S}^1}} \prod_{i=1}^k \left( \cos (  m_{i} \cdot \theta ) \pm \cos (  m_{i} \cdot \theta' ) \right)  \prod_{e \in E(\Lambda_L^g(n)) \setminus E(T)}  (2 \pi)^{-2}  d\theta_e d\theta'_e.
    \end{multline*}
    The right-hand side of the previous display is nonnegative due to the Ginibre inequality (specifically, the inequality~\eqref{ineq:Ginibre1} of Theorem~\ref{theoremGinibreineq}). The proof of~\eqref{eq:Wellswithangles} is thus complete.
    \end{proof}

\subsubsection{Wells' inequality for the XY height function} \label{sectionWellsineq2}

We may then combine Proposition~\ref{prop:generalisedWells} with the duality identity stated in Proposition~\ref{prop:propA3} (or specifically, the one stated in~\eqref{eq:dualityonpercolation}) to prove the Wells' inequality for height functions stated in Proposition~\ref{prop:Wellsforheight}.

We first introduce a definition for the XY height function on a percolation configuration (with the additional weight on $\varphi(0)$).

\begin{definition}[XY height function on a percolation configuration]
Given an integer $L \in \N$, an integer $n\geq 2$, a parameter $\lambda > 0$, two inverse temperatures $\beta_1 , \beta_2 \geq 0$ and a percolation configuration $r \in \{0,1\}^{\Lambda_L},$ we define
  \begin{equation*}
     \mu_{L , n ,\beta_1, \beta_2 , \lambda, r}^{\Z - \mathrm{XY}} (\left\{ \varphi \right\}) := \frac{1}{Z_{\Lambda_L,J, r}^{\Z-\mathrm{XY}} }  I_{\left(\varphi(0)\right)} \left( \lambda \right)\prod_{x \sim y}  \left[ I_{\left(\varphi(x) - \varphi(y)\right)} \left( r_x \beta_1 \right) I_{\left(\varphi(x) - \varphi(y)\right)} \left( r_y \beta_1 \right) I_{\left(\varphi(x) - \varphi(y)\right)} \left(  n \beta_2 \right)^{n-2} \right].
 \end{equation*}
 We denote the variance associated with this measure by $\mathrm{Var}_{L , n ,\beta_1, \beta_2 , \lambda, r}^{\Z - \mathrm{XY}}$.
 %\margin{Was not defined already before ? Nope}
 \end{definition}
 
 We next prove the Wells' inequality stated in Proposition~\ref{prop:Wellsforheight}.

\begin{proof}
The proof relies on Proposition~\ref{prop:generalisedWells} and on the duality identities stated in Proposition~\ref{prop:propA3}. We thus introduce the parameter $\lambda > 0$ (and will send this parameter to infinity at the end of the argument) and first prove the inequality: for any $L \in \N$ and any $\beta_1 , \beta_2, \lambda > 0$,
\begin{equation} \label{eq:A4generalWellsheightlambda}
\mathrm{Var}_{L , n , \beta_1/2, \beta_2, \lambda}^{\Z-\mathrm{XY}} \left[ \varphi(0) \right]  \leq \E_{\nu_{L , n , \beta_1, \beta_2, \lambda}} \left[ \mathrm{Var}_{L , n , \beta_1, \beta_2, \lambda, r}^{\Z-\mathrm{XY}} \left[ \varphi(0) \right] \right].
\end{equation}
We first recall the identity~\eqref{eq:dualityonpercolation}: for any percolation configuration $r \in \{ 0 , 1 \}^{\Lambda_L}$,
\begin{equation*}
    \mathrm{Var}_{L , n , \beta_1, \beta_2, \lambda,r}^{\Z-\mathrm{XY}} \left[ \varphi(0) \right] = \left\langle \lambda \cos \left( \theta_{0g} \right) + \frac{\lambda^2}{2} \cos \left( 2 \theta_{0g} \right)  - \frac{\lambda^2}{2}  \right\rangle_{L , n , \beta_1, \beta_2, \lambda, r}^{\mathrm{XY}}.
\end{equation*}
By Proposition~\ref{prop:generalisedWells} (with the values $m = 1$ and $m = 2$), we have that
\begin{equation*}
    \left\langle  \cos \left( \theta_{0g} \right)  \right\rangle_{L , n , \beta_1/2, \beta_2, \lambda}^{\mathrm{XY}} \leq \E_{\nu_{L , n , \beta_1, \beta_2, \lambda,r}} \left[  \left\langle  \cos \left( \theta_{0g} \right)  \right\rangle_{L , n , \beta_1, \beta_2, \lambda,r}^{\mathrm{XY}} \right]
\end{equation*}
and 
\begin{equation*}
    \left\langle  \cos \left( 2 \theta_{0g} \right)  \right\rangle_{L , n , \beta_1/2, \beta_2, \lambda}^{\mathrm{XY}} \leq \E_{\nu_{L , n , \beta_1, \beta_2, \lambda}} \left[  \left\langle  \cos \left( 2 \theta_{0g} \right)  \right\rangle_{L , n , \beta_1, \beta_2, \lambda,r}^{\mathrm{XY}} \right].
\end{equation*}
Combining the three previous displays, we obtain 
\begin{align*}
     \mathrm{Var}_{L , n , \beta_1/2, \beta_2, \lambda}^{\Z-\mathrm{XY}} \left[ \varphi(0) \right]  & = \left\langle \lambda \cos \left( \theta_{0g} \right) + \frac{\lambda^2}{2} \cos \left( 2 \theta_{0g} \right)  - \frac{\lambda^2}{2}  \right\rangle_{L , n , \beta_1/2, \beta_2, \lambda}^{\mathrm{XY}} \\
     & \leq \E_{\nu_{L , n , \beta_1, \beta_2, \lambda}} \left[ \left\langle \lambda \cos \left( \theta_{0g} \right) + \frac{\lambda^2}{2} \cos \left( 2 \theta_{0g} \right)  - \frac{\lambda^2}{2}  \right\rangle_{L , n , \beta_1, \beta_2, \lambda, r}^{\mathrm{XY}}  \right] \\
     & = \E_{\nu_{L , n , \beta_1, \beta_2, \lambda}} \left[ \mathrm{Var}_{L , n , \beta_1, \beta_2, \lambda, r}^{\Z-\mathrm{XY}} \left[ \varphi(0) \right] \right].
\end{align*}
The proof of the inequality~\eqref{eq:A4generalWellsheightlambda} is complete.

We next take the limit $\lambda \to \infty$ to deduce the second inequality of Proposition~\ref{prop:Wellsforheight}. To this end, we note that, for any $r \in \{0,1\}^{\Lambda_L}$, we have the convergences
\begin{equation*}
    \begin{aligned}
    \mathrm{Var}_{L , n , \beta_1, \beta_2, \lambda, r}^{\Z-\mathrm{XY}} \left[ \varphi(0) \right] & \underset{\lambda \to \infty}{\longrightarrow}  \mathrm{Var}_{L , n , \beta_1, \beta_2, r}^{\Z-\mathrm{XY}} \left[ \varphi(0) \right] \\
    \frac{Z_{\Lambda_L , n , \beta_1, \beta_2, \lambda, r}^{\Z-\mathrm{XY}}}{I_0(\lambda)} & \underset{\lambda \to \infty}{\longrightarrow} Z_{\Lambda_L , n , \beta_1, \beta_2, r}^{\Z-\mathrm{XY}}.
    \end{aligned}
\end{equation*}
Since there are finitely many percolation configurations in the box $\Lambda_L$, we deduce from these two convergences that
%\margin{C: $I_0(\lambda)$ does not converge to 1as $\lambda$ goes to infinity here, an additional normalisation here ? \textbf{Oh I see, I think I agree here!}}
\begin{equation*}
    \frac{\bar{Z}_{\Lambda_L , n , \beta_1, \beta_2, \lambda}^{\Z-\mathrm{XY}}}{I_0(\lambda)} = \sum_{r \in \{0,1\}^{\Lambda_L}} \frac{Z_{\Lambda_L , n , \beta_1, \beta_2, \lambda, r}^{\Z-\mathrm{XY}}}{I_0(\lambda)} \underset{\lambda \to \infty}{\longrightarrow} \sum_{r \in \{0,1\}^{\Lambda_L}} Z_{\Lambda_L , n , \beta_1, \beta_2,r}^{\Z-\mathrm{XY}} =  \bar{Z}_{\Lambda_L , n , \beta_1, \beta_2}^{\Z-\mathrm{XY}}.
\end{equation*}
A combination of the two previous displays with~\eqref{eq:identitypartfunctperc} implies that
%\margin{C: I missed why we use (4.15) here. Aren't we only working with Z valued functions at this stage and maybe 4.15 was relevant slightly before when going through Wells or something ? \textbf{This is because the last wells type measure is defined via XY weight instead of Z weights!}}
\begin{align*}
    \mathrm{Var}_{L , n , \beta_1/2, \beta_2}^{\Z-\mathrm{XY}} \left[ \varphi(0) \right]  = \lim_{\lambda \to \infty} \mathrm{Var}_{L , n , \beta_1/2, \beta_2, \lambda}^{\Z-\mathrm{XY}} \left[ \varphi(0) \right] & \leq \liminf_{\lambda \to \infty} \E_{\nu_{L , n , \beta_1, \beta_2, \lambda}} \left[ \mathrm{Var}_{L , \beta , \lambda, r}^{\Z-\mathrm{XY}} \left[ \varphi(0) \right] \right] \\
    & = \liminf_{\lambda \to \infty} \sum_{r \in \{0,1\}^{\Lambda_L}}  \mathrm{Var}_{L , n , \beta_1, \beta_2, \lambda, r}^{\Z-\mathrm{XY}} \left[ \varphi(0) \right] \frac{Z_{\Lambda_L , n , \beta_1, \beta_2, \lambda, r}^{\Z-\mathrm{XY}}}{\bar{Z}_{\Lambda_L , n , \beta_1, \beta_2, \lambda}^{\Z-\mathrm{XY}}} \\
    & = \sum_{r \in \{0,1\}^{\Lambda_L}}  \mathrm{Var}_{\Lambda_L , n , \beta_1, \beta_2, r}^{\Z-\mathrm{XY}} \left[ \varphi(0) \right] \frac{Z_{\Lambda_L , n , \beta_1, \beta_2, r}^{\Z-\mathrm{XY}}}{\bar{Z}_{\Lambda_L , n , \beta_1, \beta_2}^{\Z-\mathrm{XY}}} \\
    & =  \E_{\nu_{L , n , \beta_1, \beta_2}} \left[ \mathrm{Var}_{L , n , \beta_1, \beta_2, r}^{\Z-\mathrm{XY}} \left[ \varphi(0) \right] \right].
\end{align*}
This completes the proof of the second inequality of Proposition~\ref{prop:Wellsforheight}.
\end{proof}

\subsection{Stochastic domination for the Wells' disorders} \label{sec:sec4.3}

As it was the case for the XY and Villain models, a feature of the Wells' disorder is that it is stochastically dominated by an i.i.d. Bernoulli measure. In the following proposition (and below), we set, for $\beta \geq 0$,
\begin{equation*}
    p_0(\beta) := \frac{1}{1 + \exp \left( - 2d \beta \right)} \in (0 , 1).
\end{equation*}

\begin{proposition}[Stochastic domination for the Wells' disorder] \label{prop:domstochdisorder}
For any $L \in \N$ and any $\beta \geq 0$, one has the stochastic domination
\begin{equation} \label{eq:stochdomappendixmubeta}
    \nu_{L , \beta} \preceq \P_{\Lambda_L , p_0(\beta)}^{\mathrm{site}}.
\end{equation}
For any $L \in \N$, any $n \in \N$ and any $\beta_1 , \beta_2 \in (0, \infty)$, one has the stochastic domination
%\margin{C: SHOULD WE HAVE A $\lambda$ here and just above ?}
\begin{equation} \label{eq:stochdomappendixmubetabeta}
    \nu_{L , n , \beta_1 , \beta_2}  \preceq \P_{\Lambda_L , p_0(\beta_1)}^{\mathrm{site}}.
\end{equation}
\end{proposition}

\begin{remark}
For the proof of Theorem~\ref{Thm:delocheightsupercrit}, the crucial observation is that the probability $p_0(\beta_1) \in (0,1)$ on the right-hand side of~\eqref{eq:stochdomappendixmubetabeta} does not depend on the values of $n$ and $\beta_2$ (this will allow to send the values of these parameters to infinity while still preserving a nontrivial stochastic domination).
\end{remark}

\begin{proof}
We only show the inequality~\eqref{eq:stochdomappendixmubetabeta}. There are various possible proofs and we will prove the result using Proposition~\ref{prop:propA3}. We thus reintroduce the parameter $\lambda > 0$ and use the same argument as in the proof of Proposition~\ref{prop:stochdomWellsVillain}, we deduce that, for any $\lambda > 0$,
\begin{equation*}
     \nu_{L , n , \beta_1 , \beta_2, \lambda}  \preceq \P_{\Lambda_L , p_0(\beta_1)}^{\mathrm{site}}.
\end{equation*}
We may then take the limit $\lambda \to \infty$ (and note that the right-hand side of the previous display does not depend on $\lambda$) to complete the proof.
\end{proof}

%\margin{C: isn't the fact $\Z$ non-compact a potential issue here ?? I naively thought we would need to see something at this stage.  In the $\S^1$ case, I think the key point was that if $\beta$ is large, then }
%
%\purple{\bf Stoch Domination is much easier to see with XY, as in previous than directly starring at $\Z$ ! I should check here !! and the point here is that we have a WELLS type environment on $XY$ instread of $\Z$ valued !! This is why we don't much t check here even though it is very KEY !! }

\subsection{Delocalisation for integer-valued height functions on a high-density site percolation} \label{sec:sec4.4}

This section is devoted to the proof of the delocalisation of the height functions on a high-density site percolation. In the case of the XY height function, the argument relies on a combination of four properties stated above: the monotonicity of the variance of the height in the conductances (Proposition~\ref{prop:monotonicity}), the existence of a delocalised phase for the XY height function (Proposition~\ref{prop:delocXY/IVGFFheight}), the stochastic domination for the Wells' disorder (Proposition~\ref{prop:domstochdisorder}) and the Wells' inequality for height functions (Proposition~\ref{prop:Wellsforheight}).

In the case of the integer-valued Gaussian free field, the argument has the same structure as the proof of Proposition~\ref{prop3.12} establishing the existence of a phase transition for the Villain model on a high-density percolation and relies on the introduction of the $\Z$-XY/GFF height function~\eqref{eq:defXYGFF} together with a combination of the aforementioned results. To be more specific, we will prove the following statement.

\begin{proposition} \label{prop:prop4.19}
There exist a probability $p_0 := p_0(d) < 1$, an inverse temperature $\beta_0 \in (0 , \infty)$ and a constant $c_0 > 0$ such that for any $p \geq p_0$, any $\beta \geq \beta_0$ and any $L \in \mathbb{N}^*$,
\begin{equation*}
\mathbb{E}_p^{\mathrm{site}} \left[ \mathrm{Var}_{L , \beta, r}^{\Z-\mathrm{XY}}  \left[ \varphi(0) \right] \right] \geq c_0 \ln L ~~\mbox{and}~~ \mathbb{E}_p^{\mathrm{site}} \left[ \mathrm{Var}_{L , \beta, r}^{\Z-\mathrm{GFF}}  \left[ \varphi(0) \right] \right] \geq c_0 \ln L.
\end{equation*}
\end{proposition}

\begin{proof}
We first prove the result for the XY height function and then for the integer-valued Gaussian free field.

\textit{\underline{Proof for the XY height function:}} We let $\beta_{\mathrm{Deloc}} > 0$ be the inverse temperature which appears in the statement of Proposition~\ref{prop:delocXY/IVGFFheight} and set $\beta_0 := 4 \beta_{\mathrm{Deloc}}$ and $p_0 := p_0(\beta_0)$. Then, we may apply Proposition~\ref{prop:delocXY/IVGFFheight} together with Wells' inequality (Proposition~\ref{prop:Wellsforheight}) to deduce that there exists a constant $c_0 > 0$ such that, for any $L \in \N^*$,
\begin{equation} \label{ineq:delocinWellsdisroder}
      \E_{\nu_{L , \beta_0}} \left[ \mathrm{Var}^{\Z-\mathrm{XY}}_{L , \beta_0, r} \left[ \varphi(0) \right] \right] \geq \mathrm{Var}^{\Z-\mathrm{XY}}_{L , \beta_0/4} \left[ \varphi(0) \right] \geq   c_0 \ln L.
\end{equation}
We next note that, by Corollary~\ref{cor:corollary4.9}, the map
\begin{equation*}
    r \in \{0,1\}^{\Lambda_L} \mapsto \mathrm{Var}^{\Z-\mathrm{XY}}_{L , \beta_0, r} \left[ \varphi(0) \right] ~~\mbox{is increasing.}
\end{equation*}
Combining this observation with the stochastic domination for the Wells' disorder stated in Proposition~\ref{prop:domstochdisorder}, we deduce that
\begin{equation*}
    \E_{p_0}^{\mathrm{site}} \left[ \mathrm{Var}^{\Z-\mathrm{XY}}_{L , \beta_0, r} \left[ \varphi(0) \right] \right] \geq \E_{\nu_{L , \beta_0}} \left[ \mathrm{Var}^{\Z-\mathrm{XY}}_{L , \beta_0, r} \left[ \varphi(0) \right] \right].
\end{equation*}
Together with~\eqref{ineq:delocinWellsdisroder}, we obtain 
\begin{equation*}
    \E_{p_0}^{\mathrm{site}} \left[ \mathrm{Var}^{\Z-\mathrm{XY}}_{L , \beta_0, r} \left[ \varphi(0) \right] \right] \geq c_0 \ln L .
\end{equation*}
Increasing the values of the probability $p_0$ or of the inverse temperature $\beta_0$ increases the term on the left-hand side (due to Remark~\ref{rem:remark4} and Proposition~\ref{prop:monotonicity}). This concludes the proof of Proposition~\ref{prop:prop4.19} for the XY height function.

\medskip

\textit{\underline{Proof for the integer-valued Gaussian free field:}} 
We let $\beta_{1}= \beta_{2} = 2\beta_{\mathrm{Deloc}}$ where $\beta_{\mathrm{Deloc}}$ is the inverse temperature appearing in the statement of Proposition~\ref{prop:delocXY/IVGFFheight} and set $p_0 := p_0(2\beta_{\mathrm{Deloc}})$. 

We let $L \in \N$ be a (large) sidelength and fix a (large) integer $n \in \N$. We first combine Proposition~\ref{prop:domstochdisorder} with Corollary~\ref{cor:corollary4.9} to deduce that
%\margin{C: here is it $\beta_1/2$ or $\beta_1/4$ after Wells ?}
\begin{equation*}
    \E_{p_0}^{\mathrm{site}} \left[ \var_{L, n , \beta_{1} , \beta_{2}, r}^{\Z-\mathrm{XY}}\left[ \varphi(0) \right] \right] \geq \E_{\nu_{L , n , \beta_{1}, \beta_{2}}} \left[ \var_{L, n , \beta_{1}, \beta_{2}, r}^{\Z-\mathrm{XY}} \left[ \varphi(0) \right] \right].
\end{equation*}
Combining the previous bound with Wells' inequality (Proposition~\ref{prop:Wellsforheight}), we deduce that 
\begin{equation*}
     \E_{p_0}^{\mathrm{site}} \left[ \var_{L, n , \beta_{1} , \beta_{2}, r}^{\Z-\mathrm{XY}}\left[ \varphi(0) \right] \right] \geq \var_{L, n , \beta_{1} /2, \beta_{2}}^{\Z-\mathrm{XY}}\left[ \varphi(0) \right].
\end{equation*}
Since this inequality holds for any value of the integer $n \in \N$ (and since the probability $p_0$ does not depend on $n$), we may take the limit $n \to \infty$ and apply Proposition~\ref{prop:approxIVGFFbyXY} to deduce that 
    \begin{equation*}
     \E_{p_0}^{\mathrm{site}} \left[ \var_{L, \beta_{1} , \beta_{2}, r}^{\Z-\mathrm{XY/GFF}}\left[ \varphi(0) \right] \right] \geq \var_{L, \beta_{1}/2 ,  \beta_{2}}^{\Z-\mathrm{XY/GFF}}\left[ \varphi(0) \right] .
\end{equation*}
Applying the inequality~\eqref{eq:ineq4.11} allows to simplify the left-hand side
    \begin{equation*}
    \E_{p_0}^{\mathrm{site}} \left[ \var_{L,\beta_{2}, r}^{\Z-\mathrm{GFF}}\left[ \varphi(0) \right] \right] \geq  \var_{L, \beta_{1}/2 ,  \beta_{2}}^{\Z-\mathrm{XY/GFF}}\left[ \varphi(0) \right].
\end{equation*}
Finally, by applying Proposition~\ref{prop:delocXY/IVGFFheight}, we deduce that there exists $c _0 > 0$ such that, for any $L \in \N^*$,
\begin{equation*}
      \E_{p_0}^{\mathrm{site}} \left[ \var_{L, \beta_{2}, r}^{\Z-\mathrm{GFF}}\left[ \varphi(x) \right] \right] \geq c_0 \ln L.
\end{equation*}
Increasing the values of the probability $p_0$ or of the inverse temperature $\beta_{2}$ increases the term on the left-hand side (due to Remark~\ref{rem:remark4} and Proposition~\ref{prop:monotonicity}). We may thus conclude the proof of Proposition~\ref{prop:prop4.19} in the case of the integer-valued Gaussian free field.
\end{proof}

\subsection{Delocalisation for the integer-valued Gaussian free field on a supercritical percolation} \label{sec:renoarmargumentforheight}

In this section, we implement the renormalisation argument used to prove the existence of a phase transition for the integer-valued height functions on any supercritical percolation (and establish Theorem~\ref{Thm:delocheightsupercrit}). We will only prove the result for the integer-valued Gaussian free field (N.B. A similar -- although more involved -- argument could be used to prove the result for the XY height function; to minimise the technicality of the proof, we will proceed differently and show the result for the XY height function in Section~\ref{sec:section4.7} by making use of the tools developed there to study height functions with annealed Gaussian potential).

We first introduce some known results on the supercritical phase of Bernoulli edge percolation in Section~\ref{subsec:preliminarysuperctiBern} and then use these results in the proof of Theorem~\ref{Thm:delocheightsupercrit} in Section~\ref{subsec:preliminarysuperctiBern2}.

\subsubsection{Preliminaries: supercritical percolation} \label{subsec:preliminarysuperctiBern}

We introduce in this section two (standard) definitions in Bernoulli percolation: the notions of dual graph and dual percolation configurations. We then introduce the definition of certain ``good events" and state a result from~\cite{penrose-pisztora-1996, pisztora-percolation} asserting that, in the supercritical phase, the probability of these good events are close to $1$.

\begin{figure}
    \centering
    \begin{minipage}{.5\textwidth}
        \centering
        \includegraphics[scale= 0.7]{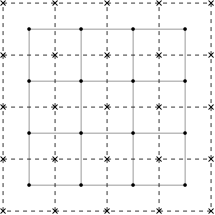}
        \caption{An illustration of the lattice $\Z^2$ (with dots and unbroken lines) together with the dual lattice (with crosses and dotted lines)}\label{fig:duallattice}
    \end{minipage}%
    \begin{minipage}{0.5\textwidth}

    \end{minipage}
\end{figure}

\begin{figure}
    \centering
        \includegraphics[scale= 0.5]{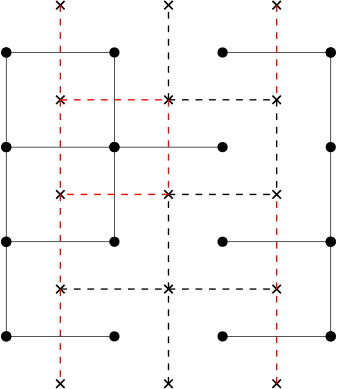}  \caption{A representation of the percolation configurations $\omega$ (with straight lines) and $\omega^*$ (with black dotted lines). We will be interested in the closed edges of the percolation configuration $\omega^*$; these edges are drawn in red dotted lines.} \label{fig:dualperco} %, $\tilde r$ and $\tilde r^*$. The percolation configuration $r$ is defined on the sites of the lattice: an open site is represented by a (full) disk and a closed site is represented by a circle. The percolation configuration $\tilde r$ is then defined on the edges of $\Z^2$; an edge is open if and only if one of its endpoints is closed in the percolation configuration $r$. The percolation configuration $\tilde r^*$ is the dual of the percolation configuration $\tilde r$; its open edges are in the dual lattice and are represented by dotted lines.}
\end{figure}

\begin{definition}[Dual graph]
    Consider the graph $(\Z^2)^*$ defined to be the copy of $\Z^2$ translated by the vector $(1/2, 1/2)$. We refer to this graph as the dual lattice and denote by $E((\Z^2)^*)$ its edge set. Each edge $e^* \in E((\Z^2)^*)$ is naturally associated with an edge of $e$ of $\Z^2$ (which is the only edge of $\Z^2$ intersecting the edge~$e^*$). We refer to Figure~\ref{fig:duallattice} for an illustration.
\end{definition}

\begin{definition}[Percolation configurations]
 Given an edge percolation configuration $\omega \in \{0, 1\}^{\Z^2}$, we define the dual percolation configuration $\omega^* \in \{0,1\}^{E((\Z^2)^*)}$ on the edges of the dual graph $(\Z^2)^*$ by setting $\omega_{e^*} =1- \omega_e$ where $e^*$ is the dual edge of $e$. We refer to Figure~\ref{fig:dualperco} for an illustration.
\end{definition}

%\begin{remark}
%These two percolations are $1$-dependent and dual to each other (see Figure~\ref{fig:dualperco}), the percolation $r$ is decreasing in the probability $p$ and the percolation $\tilde r^*$ is increasing in the probability $p$.
%\end{remark}

\begin{definition}[Good event for percolation configurations on the dual lattice]
For any $x \in (\Z^2)^*$ and any (large) integer $L \in \N$, we let
\begin{multline*}
    G_{x,L} := \left\{ \mbox{Any horizontal or vertical rectangle with short and long sidelengths equal to}~ \left\lceil \frac{L}{100} \right\rceil ~\mathrm{and} ~ \left\lceil \frac{22 L}{10} \right\rceil  \right. \\ \left. \mbox{included in } \left(x + \Lambda_{\left\lceil \frac{11L}{10} \right\rceil} \right) \mbox{ is crossed in the long direction by a cluster of \underline{closed edges} for}~ \omega^* \right\}.
\end{multline*}
We refer to Figure~\ref{fig:goodbox} for an illustration of this event.
\end{definition}

\begin{figure}
\centering
    \includegraphics[scale = 0.6]{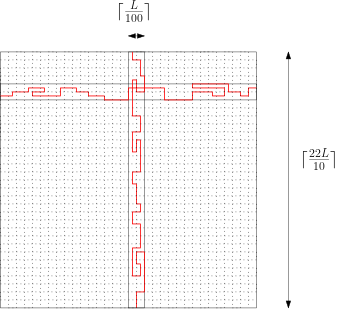}
    \caption{An illustration of a good box: on the figure, two rectangles are drawn and they are crossed in the long direction by paths of closed edges in the dual lattice for the percolation configuration $\omega^*$ (drawn in red)}
    \label{fig:goodbox}
\end{figure}

The following proposition states that, for any $p > 1/2$, the probability of the event $G_{x,L}$ is exponentially close to $1$ as the sidelength of the box tends to infinity (N.B. This is because the law of the closed edges of $\omega^*$ is the one of a Bernoulli edge percolation of probability $p$).

\begin{proposition}[\cite{penrose-pisztora-1996, pisztora-percolation}] \label{prop:prop4.15}
    For any probability $p > 1/2$, their exist two constants $C , c \in (0, \infty)$ (depending on $p$) such that, for any $x \in (\Z^2)^*$,
    \begin{equation*}
        \mathbb{P}_p \left[  G_{x,L} \right] \geq 1 - C e^{- c L}.
    \end{equation*}
\end{proposition}

\subsubsection{Preliminaries: delocalisation for the integer-valued Gaussian free field on the diagonal graph} \label{prelim2delocwithdiagonal}

For a technical reason in the renormalisation argument of the next section, we will need to use a (slight) generalisation of the result proved in Section~\ref{sec:sec4.4} to the graph $\Z^2$ with diagonal edges. To be more specific, let us denote by $E_{\times}\left(  \Z^2 \right)$ the set of horizontal, vertical and diagonal edges of the lattice $\Z^2$ and introduce the following version of the integer-valued Gaussian free field.

\begin{figure}
    \centering
    \includegraphics[scale = 0.7]{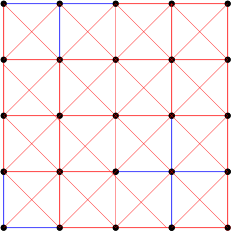}
    \caption{An illustration for the model introduced in Definition~\ref{def:ivgff+diagonal}. The graph $\Z^2$ with the diagonal edges is depicted together with a percolation configuration $r \in \{0,1\}^{\Lambda_L}$: the edges $\{x , y\}$ along which $r_x r_y = 0$ are drawn in blue.}
    \label{fig:ivgff+diagonal}
\end{figure}

\begin{definition}[Integer-valued Gaussian free field on $\Z^2$ with diagonal edges] \label{def:ivgff+diagonal}
    Given an integer $L \in \N$, an inverse temperature $\beta > 0$ and a site percolation configuration $r \in \{0,1\}^{\Lambda_L},$ we define
    \begin{multline*}
        \mu^{\Z-\mathrm{GFF}, \times}_{L , \beta, r}(\left\{ \varphi \right\}) \\
        := \frac{1}{Z^{\Z-\mathrm{GFF}, \times}_{L , \beta , r}}\exp \left( - \frac{1}{2\beta}\sum_{\{ x , y\} \in E(\Z^2)} \frac{\left( \varphi(x) - \varphi(y)\right)^2}{r_{x} r_y } - \frac{1}{2\beta} \sum_{\{ x , y\} \in E_\times(\Z^2) \setminus E(\Z^2)} \left( \varphi(x) - \varphi(y)\right)^2 \right).
    \end{multline*}
    We denote by $\mathrm{Var}^{\Z-\mathrm{GFF}, \times}_{L , \beta, r}$ the variance with respect to this measure.
\end{definition}

\begin{remark}
The percolation configuration only affects the nearest-neighbour edges (as depicted in Figure~\ref{fig:ivgff+diagonal}).
\end{remark}

The result we will need is an extension of the one of Section~\ref{sec:sec4.4} to this model.

\begin{proposition} \label{prop:delocwithdiagonal}
There exist a probability $p_0 \in (0,1)$, an inverse temperature $\beta_{\mathrm{Deloc}, \times} \in (0 , \infty)$ and a constant $c_0 \in (0, \infty)$ such that, for any $p \geq p_0$, any $\beta \geq \beta_{\mathrm{Deloc}, \times}$ and any $L \in \N^*$,
\begin{equation*}
    \E_p^{\mathrm{site}} \left[ \mathrm{Var}^{\Z-\mathrm{GFF}, \times}_{L , \beta, r} \left[ \varphi(0) \right] \right] \geq c_0 \ln L.
\end{equation*}
\end{proposition}

\begin{proof}
We provide a brief outline of two arguments that may be used to establish his result. The first option would be to rewrite the proof written in Section~\ref{sec:sec4.4} for this model (the modifications required are mostly notational). A second option would be to use the graph surgery technique introduced in Section~\ref{section4.8} to show the following inequality: for any percolation configuration $r \in \{0 , 1 \}^{\Lambda_L}$,
\begin{equation*}
    \mathrm{Var}^{\Z-\mathrm{GFF}, \times}_{L , \beta, r} \left[ \varphi(0) \right] \geq \mathrm{Var}^{\Z-\mathrm{GFF}}_{L , \beta/5, r} \left[ \varphi(0) \right].
\end{equation*}
Using this inequality, Proposition~\ref{prop:delocwithdiagonal} follows from the result of Section~\ref{sec:sec4.4}.
\end{proof}

\subsubsection{Proof of the delocalisation for the integer-valued Gaussian free field on a supercritical percolation} \label{subsec:preliminarysuperctiBern2}

\begin{figure}
 \centering
    \includegraphics[scale = 0.55]{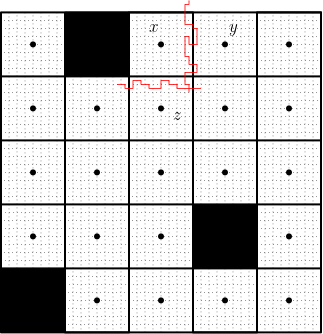}
    \caption{An example of the site percolation on the renormalised lattice (closed sites are drawn in black). Between pairs of open sites (here the boxes $(x + \Lambda_{L_0}), (y + \Lambda_{L_0})$ and $(z + \Lambda_{L_0})$) a path of closed edges for the percolation configuration $\omega^*$ in the dual lattice is drawn (in red).}
    \label{fig:twocontours}
\end{figure}

\begin{figure}
    \centering
    \includegraphics[scale = 0.6]{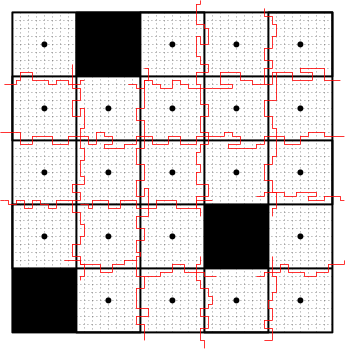}
    \caption{The collection of edges $\mathcal{C}^*$ of the dual lattice $(\Z^2)^*$.}
    \label{fig:manycontours}
\end{figure}

\begin{figure}
    \centering
    \includegraphics[scale = 0.4]{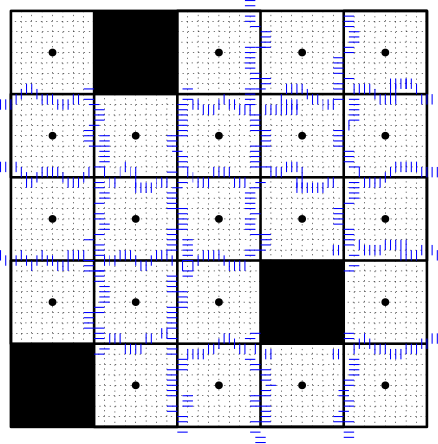}
    \caption{The collection of edges of $\Z^2$ which are dual to the edges of $\mathcal{C}^*$.}
    \label{fig:bluecontours}
\end{figure}

\begin{figure}
    \centering
    \includegraphics[scale = 0.5]{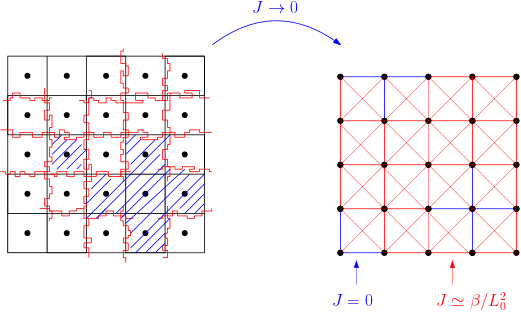}
    \caption{Reducing the values of the conductance of the edges which are not in $\mathcal{C}$ to $0$ has the effect of constraining all the sites which are in the same connected component of the graph $\left(\Z^2 , E \left( \Z^2 \right) \setminus \mathcal{C} \right)$ to take the same values (two of these connected components are represented in blue on the figure). This operation maps the model to an integer-valued Gaussian free field on the lattice $\Z^2$ on a high-density and $1$-dependent percolation (whose variance can be lower bounded using Theorem~\ref{th.Liggett} and the result of Section~\ref{prelim2delocwithdiagonal}).}
    \label{fig:freezingfaces}
\end{figure}

\begin{proof}[Proof of Theorem~\ref{Thm:delocheightsupercrit} for the integer-valued Gaussian free field]
We fix a probability $p > 1/2$, let $\beta_{\mathrm{Deloc}, \times}$ and $p_{0}$ be the inverse temperature and the probability appearing in the statement of Proposition~\ref{prop:delocwithdiagonal}. We next let $\bar p \in (0,1)$ be the probability provided by Theorem~\ref{th.Liggett} for $1$-dependent percolation with probability $p = p_0$ and let $L_0 \in \N$ be a (large) integer chosen such that
\begin{equation} \label{eq:defL0goodbox}
    \forall x \in (\Z^2)^*, ~~ \mathbb{P}_p \left[  G_{x,L_0} \right] \geq \bar p.
\end{equation}
The existence of the integer $L_0$ is guaranteed by Proposition~\ref{prop:prop4.15}. Using these definitions, we set
$$\beta_{\mathrm{Deloc}}(p) := L_0^2 \beta_{\mathrm{Deloc}, \times} \in (0,\infty).$$
We then partition the dual lattice $(\Z^2)^*$ into boxes of sidelength $L_0$ according to the identity
    \begin{equation*}
        (\Z^2)^* := \bigcup_{z \in (2L_0 +1) \Z^2 + \left(\frac12 , \frac12 \right)} (z + \Lambda_{L_0}).
    \end{equation*}
Using this partition, we may define a renormalised graph as follows. We denote by $\mathcal{G}_{L_0}$ the graph whose vertices are the boxes of the form $(z + \Lambda_{L_0}) \subseteq  (\Z^2)^*$ with $z \in  (2L_0 +1) \Z^2 + \left(\frac12 , \frac12 \right)$ and the edges are the pairs of boxes $\{ z + \Lambda_{L_0} , z' + \Lambda_{L_0} \}$ which have adjacent side. We note that this graph is isomorphic to $\Z^2$ and refer to Figure~\ref{fig:twocontours} for an illustration.
    
Given an edge percolation $\omega \in \{0,1\}^{E \left(\Z^2 \right)}$, we define a site percolation configuration $r_1^* \in \{ 0 , 1 \}^{\mathcal{G}_{L_0}}$ on the renormalised lattice by declaring a site $(z + \Lambda_{L_0}) \in \mathcal{G}_{L_0}$ open if and only if the event $G_{z , L_0}$ is realised. We note that this percolation is $1$-dependent and that, by the inequality~\eqref{eq:defL0goodbox}, the probability of a site to be open is larger $\bar p$. By Theorem~\ref{th.Liggett}, the percolation $r_1^*$ stochastically dominates and i.i.d. percolation of probability $p_0$ on the sites of $\mathcal{G}_{L_0}$.
    
Given an edge percolation configuration $\omega \in \{0,1\}^{E \left( \Z^2 \right)}$, we further define a collection of paths in the dual lattice as follows. For each pair of neighbouring sites $(x_0 + \Lambda_{L_0}), (x_1 + \Lambda_{L_0}) \in \mathcal{G}_{L_0}$ which are both open in the percolation configuration~$r_1^*$, we select a path of closed edges in the percolation configuration $\omega^*$ which crosses in the long direction a rectangle $\mathcal{R}$ satisfying the following properties (see Figure~\ref{fig:twocontours}):
\begin{itemize}
    \item $\mathcal{R}$ has short and long sidelengths equal to $\left\lfloor \frac{L}{100} \right\rfloor$ and $\left\lfloor \frac{22 L}{10} \right\rfloor$.
    \item The long side of $\mathcal{R}$ is parallel to the common side of the squares $(x_0 + \Lambda_{L_0})$ and $(x_1 + \Lambda_{L_0})$.
    \item $\mathcal{R}$ intersects both $(x_0 + \Lambda_{L_0})$ and $(x_1 + \Lambda_{L_0})$.
\end{itemize}
We then denote by $\mathcal{C}^*$ the union of these paths over the pairs of neightbouring sites $(x_0 + \Lambda_{L_0}), (x_1 + \Lambda_{L_0}) \in \mathcal{G}_{L_0}$ which are both open in the percolation configuration $r_1^*$ and by $\mathcal{C}$ the collection of the edges of $\Z^2$ which are dual to an edge of $\mathcal{C}^*$ (see Figures~\ref{fig:manycontours} and~\ref{fig:bluecontours}). 

We remark that all the edges of $\mathcal{C}$ are open in the percolation configuration $\omega$ and that there may be more than one choice for these paths (or even the rectangles which they cross). We only impose two conditions on their selection: we assume that each connected component of the graph $\left(\Z^2 , E \left( \Z^2 \right) \setminus \mathcal{C} \right)$ (i.e., the graph with vertex set $\Z^2$ and edge set $E \left( \Z^2 \right) \setminus \mathcal{C}$) contains a vertex of $(2L_0 +1) \Z^2$ (i.e., we rule out the existence of a small loop in the set $\mathcal{C}^*$ which does not surround a box of the form $(z + \Lambda_{L_0})$) and maximise the number of intersections between these paths.

Equipped with these definitions, we let $\omega_{\mathcal{C}} \in \{0 , 1\}^{E(\Z^2)}$ be the percolation configuration on the edges of $\Z^2$ defined by setting $\omega_{\mathcal{C}, xy} := 1$ for $\{ x, y\} \in \mathcal{C}$ and $\omega_{\mathcal{C} , xy} := 0$ otherwise. This definition ensures that
\begin{equation} \label{eq:randrC}
\forall \{x,y\} \in E \left(\Z^2 \right), ~~ \omega_{xy} \geq \omega_{\mathcal{C},xy}.
\end{equation}
Let us fix an inverse temperature $\beta$ satisfying $\beta \geq \beta_{\mathrm{Deloc}}(p)$ and a sidelength $L \in \N$ which is assumed to be divisible by $(2L_0 + 1)$ (N.B. This assumption will be lifted at the end of the proof). By Proposition~\ref{prop:monotonicity}, the inequality~\eqref{eq:randrC} implies
\begin{equation} \label{eq:ineq4.18}
    \mathrm{Var}^{\Z-\mathrm{GFF}}_{L , \beta, \omega} \left[ \varphi(0) \right] \geq \mathrm{Var}^{\Z-\mathrm{GFF}}_{L , \beta, \omega_\mathcal{C}} \left[ \varphi(0) \right].
\end{equation}
We next show that the variance of the height $\varphi(0)$ of the integer-valued Gaussian free field on the percolation configuration $\omega_\mathcal{C}$ can be lower bounded by the one of a Gaussian free field on a high-density percolation of the box $\Lambda_{L / (2L_0+1)}$. More specifically, we proceed with the following identifications (and refer to Figure~\ref{fig:freezingfaces} for guidance):
\begin{itemize}
    \item As mentioned above, we may identify the renormalised lattice $\mathcal{G}_0$ with the lattice $\Z^2$ by identifying the vertex $(z + \Lambda_{L_0}) \in \mathcal{G}_0$ with the vertex $\left( z - (1/2,1/2) \right) / (2L_0 + 1) \in \Z^2$. We can thus see $r_1^*$ as a site percolation on $\Z^2$. This percolation stochastically dominates an i.i.d. site percolation of probability $p_0$ on $\Z^2$.
    \item The integer-valued Gaussian free field can then be identified with an integer-valued Gaussian free field on $\Z^2$ (with diagonal edges) where the conductance of an edge $\{ x , y \} \in E_\times (\Z^2)$ is given by the formula:
    \begin{itemize}
        \item If $\{ x , y \} \in E (\Z^2)$ (i.e., it is an horizontal or a vertical edge) and either $r_{1, x}^* = 0$ or $r_{1,y}^* = 0$ (i.e., one of the two sites is closed), then the conductance is equal to $0$.
        \item Otherwise we denote by $\mathcal{C}_x$ (resp. $\mathcal{C}_y$) the connected component of the vertex $(2L_0 + 1)x \in \Z^2$ (resp. $(2L_0 + 1)y \in \Z^2$) in the graph $\left(\Z^2 , E \left( \Z^2 \right) \setminus \mathcal{C} \right)$, and assign to the edge $\{ x , y \}$ a conductance which is equal to the parameter $\beta$ divided by the number of edges of $\mathcal{C}$ which have one endpoint in $\mathcal{C}_x$ and one endpoint in $\mathcal{C}_y$.
    \end{itemize}
    Let us make two comments about these identifications. First the reason one has to consider the integer-valued Gaussian free field on the graph $\Z^2$ with diagonal edges is that there may be edges of $\mathcal{C}$ whose endpoints are in two connected components of diagonal neighbours. Second, for any pair of vertices $x , y \in \Z^2$, the number of edges which have one endpoint in $\mathcal{C}_x$ and one endpoint in $\mathcal{C}_y$ is smaller than~$L_0^2$ (this is because these edges must lie in a rectangle with small and long sidelength equal to $\lfloor \frac{L_0}{100} \rfloor$ and $\lfloor \frac{11 L_0}{10} \rfloor$ whose cardinality is smaller than $L_0^2$)
\end{itemize} 
Combining these observations with Proposition~\ref{prop:monotonicity}, we deduce that
\begin{equation*}
    \E_p \left[ \mathrm{Var}^{\Z-\mathrm{GFF}}_{L , \beta, \omega_\mathcal{C}} \left[ \varphi(0) \right] \right] \geq \E_p \left[ \mathrm{Var}^{\Z-\mathrm{GFF}, \times}_{L/(2L_0+1) , \beta / L_0^2 , r_1^*} \left[ \varphi(0) \right] \right].
\end{equation*}
We next use that the percolation $r_1^*$ stochastically dominates an i.i.d. percolation of probability $p_0$ together with Proposition~\ref{prop:monotonicity} to deduce that
\begin{equation*}
    \E_p \left[ \mathrm{Var}^{\Z-\mathrm{GFF}}_{L , \beta, \omega_\mathcal{C}} \left[ \varphi(0) \right] \right] \geq \E_{p_0}^{\mathrm{site}} \left[ \mathrm{Var}^{\Z-\mathrm{GFF}, \times}_{L/(2L_0+1) , \beta / L_0^2 , r} \left[ \varphi(0) \right] \right].
\end{equation*}
Under the assumption that $\beta \geq \beta_{\mathrm{Deloc}}(p)$, we can apply the result of Proposition~\ref{prop:delocwithdiagonal} to deduce that
\begin{equation*}
      \E_{p_0}^{\mathrm{site}} \left[ \mathrm{Var}^{\Z-\mathrm{GFF}, \times}_{L/(2L_0+1) , \beta / L_0^2 , r} \left[ \varphi(0) \right] \right] \geq c_0 \ln \left( \frac{L}{2L_0+1} \right).
\end{equation*}
Combining the three previous inequalities with~\eqref{eq:ineq4.18}, we deduce that
\begin{equation*}
    \E_{p} \left[ \mathrm{Var}^{\Z-\mathrm{GFF}}_{L , \beta, \omega} \left[ \varphi(0) \right] \right] \geq c_0 \ln \left( \frac{L}{2L_0+1} \right). 
\end{equation*}
The factor $1/(2L_0+1)$ can then be removed from the right-hand side by reducing the value of the constant $c_0$. The assumption that $L$ is divisible by $(2L_0+1)$ can be similarly removed by reducing the value of $c_0$ and using Corollary~\ref{cor:corollary4.9}. This completes the proof of Theorem~\ref{Thm:delocheightsupercrit} for the integer-valued Gaussian free field.

\end{proof}

\subsection{Delocalisation for integer-valued height functions with annealed Gaussian potential and for the XY height function on a percolation cluster} \label{sec:section4.7}

This section is devoted to the proof of Theorem~\ref{Thm:delocheightanealedgauss} and to the proof of Theorem~\ref{Thm:delocheightsupercrit} in the case of the XY height function following the outline presented in Section~\ref{sec:1.4techniquesofproof}. %The argument relies on the observation that, by an application of the FKG inequality, the variance of the height $\varphi(0)$ for the integer-valued height function with annealed Gaussian potential can be lower bounded by the variance of the height $\varphi(0)$ of an integer-valued Gaussian free field with i.i.d. random conductances. This second model can be analysed using the results of Theorem~\ref{Thm:delocheightsupercrit}.

\begin{proof}[Proof of Theorem~\ref{Thm:delocheightanealedgauss}]
Let $V$ be an annealed Gaussian potential. We first write
    \begin{equation*}
        \mathrm{Var}_{L , \beta}^{\Z-V} \left[ \varphi(0) \right] = \frac{\sum_{\varphi \in \Omega_{\Lambda_L}^\Z} \left| \varphi(0) \right|^2 \exp \left( - \sum_{x \sim y} V \left(\frac{ \varphi(x) - \varphi(y)}{\sqrt{\beta}}\right) \right)}{Z_{\Lambda_L,\beta}^{\Z-V}}.
    \end{equation*}
    We next use the identity~\eqref{eq:defannealedgaussian} to rewrite the numerator of the previous display as follows
    \begin{align} \label{eq:computannealedgauss}
        \lefteqn{\sum_{\varphi \in \Omega_{\Lambda_L}^\Z} \left| \varphi(0) \right|^2 \exp \left( - \sum_{x \sim y} V \left(  \frac{ \varphi(x) - \varphi(y)}{\sqrt{\beta}} \right)\right) } \qquad & \\ 
        & = \sum_{\varphi \in \Omega_{\Lambda_L}^\Z} \left| \varphi(0) \right|^2 \prod_{x \sim y} \int_{(0 , \infty)} \exp \left( - \frac{ \left| \varphi(x) - \varphi(y) \right|^2}{2 \beta J} \right) \mu_V(d J) \notag \\
        & =  \int_{(0 , \infty)^{E \left( \Lambda_L \right)}} \sum_{\varphi \in  \Omega_{\Lambda_L}^\Z} \left| \varphi(0) \right|^2 \exp \left( - \sum_{x \sim y} \frac{\left| \varphi(x) - \varphi(y) \right|^2}{2 \beta J_{xy}} \right) \prod_{x \sim y} \mu_V(d J_{xy}). \notag
    \end{align}
    In the rest of this proof, we will denote by $\mathbf{E}$ the expectation with respect to the product measure $\prod_{x \sim y} \mu_V(d J_{xy})$ (and, to ease the notation, we will denote this product measure by $\mathbf{P}$). We next use the identity, which follows from Definition~\ref{def:def4.1},
    \begin{equation*}
        \sum_{\varphi \in  \Omega_{\Lambda_L}^\Z} \left| \varphi(0) \right|^2  \exp \left( - \sum_{x \sim y}  \frac{\left| \varphi(x) - \varphi(y) \right|^2}{2 \beta J_{xy}} \right) =  \mathrm{Var}_{L , \beta J}^{\Z-\mathrm{GFF}} \left[ \varphi(0) \right] \times  Z_{\Lambda_L , \beta  J}^{\Z-\mathrm{GFF}},
    \end{equation*}
    which, together with~\eqref{eq:computannealedgauss}, implies the identity
    \begin{equation*}
        \sum_{\varphi \in \Omega_{\Lambda_L}^\Z} \left| \varphi(0) \right|^2 \exp \left( - \sum_{x \sim y} V \left( \frac{\varphi(x) - \varphi(y)}{\sqrt{\beta}} \right) \right) = \mathbf{E} \left[ \mathrm{Var}_{L , \beta J}^{\Z-\mathrm{GFF}} \left[ \varphi(0) \right] \times  Z_{\Lambda_L , \beta J}^{\Z-\mathrm{GFF}} \right].
    \end{equation*}
    The same computation yields the identity for the partition function
    \begin{equation*}
        Z_{\Lambda_L,\beta}^{\Z-V} = \sum_{\varphi \in \Omega_{\Lambda_L}^\Z} \exp \left( - \sum_{x \sim y} V \left( \frac{ \varphi(x) - \varphi(y)}{\sqrt{\beta}} \right) \right) = \mathbf{E} \left[  Z_{\Lambda_L , \beta J}^{\Z-\mathrm{GFF}} \right].
    \end{equation*}
    Combining the few previous results, we have obtained
    \begin{equation*}
       \mathrm{Var}_{L , \beta}^{\Z-V} \left[ \varphi(0) \right] =  \frac{\mathbf{E} \left[   \mathrm{Var}_{L , \beta J}^{\Z-\mathrm{GFF}} \left[ \varphi(0) \right] \times  Z_{\Lambda_L , \beta J}^{\Z-\mathrm{GFF}} \right]}{\mathbf{E} \left[ Z_{\Lambda_L , \beta J}^{\Z-\mathrm{GFF}} \right]}.
    \end{equation*}
    We finally note that the two functions $J \mapsto \mathrm{Var}_{L , \beta J}^{\Z-\mathrm{GFF}} \left[ \varphi(0) \right]$ and $J \mapsto Z_{\Lambda_L , \beta J}^{\Z-\mathrm{GFF}}$ are both increasing in the conductances $J = (J_{xy})_{ \{ x , y\} \in E \left( \Lambda_L \right)}$. Indeed the monotonicity for the first function is a consequence of Theorem~\ref{prop:monotonicity} and the monotonicity of the second function is a consequence of the identity
    \begin{equation*}
         Z_{\Lambda_L , \beta J}^{\Z-\mathrm{GFF}} :=  \sum_{\varphi \in  \Omega_{\Lambda_L}^\Z} \exp \left( - \sum_{x \sim y}  \frac{\left| \varphi(x) - \varphi(y) \right|^2}{2 \beta J_{xy}} \right),
    \end{equation*}

    together with the observation that each term inside the sum is an increasing function of $J = (J_{xy})_{ \{ x , y\} \in E \left( \Lambda_L \right)}$. We may thus apply the FKG-inequality stated in Proposition~\ref{prop.FKGinequality} to obtain the inequality
    \begin{equation} \label{eq:04041220}
        \mathrm{Var}_{L , \beta }^{\Z-V} \left[ \varphi(0) \right] =  \frac{\mathbf{E} \left[  \mathrm{Var}_{L , \beta J}^{\Z-\mathrm{GFF}} \left[ \varphi(0) \right] \times  Z_{\Lambda_L , \beta J} \right]}{\mathbf{E} \left[  Z_{\Lambda_L , \beta J}^{\Z-\mathrm{GFF}} \right]} \geq \mathbf{E} \left[  \mathrm{Var}_{L ,\beta J}^{\Z-\mathrm{GFF}} \left[ \varphi(0) \right]  \right].
    \end{equation}
    We next let $p_0 = 3/4$ and $\beta_{\mathrm{Deloc}}(3/4) \geq 0$ be the probability and inverse temperature provided by Theorem~\ref{Thm:delocheightsupercrit} (for the integer-valued Gaussian free field) and select $\beta_V \in (0, \infty)$ such that %(N.B. the left-hand side tends to $1$ as $\beta_V \to \infty$)
    \begin{equation*} 
        \mu_V \left( \left( \frac{\beta_{\mathrm{Deloc}}(3/4)}{\beta_V} , \infty \right) \right) \geq 3/4.
    \end{equation*}
    Given a collection of conductances $J = (J_{xy})_{ \{ x , y\} \in E \left( \Lambda_L \right)}$, we define the edge percolation configuration $\omega^J \in \{ 0,1\}^{E \left( \Lambda_L \right)}$ according to the identity, for any $\{ x , y\} \in E \left( \Lambda_L \right)$,
    \begin{equation*}
        \omega_{xy}^J := \indc_{\{ \beta_V J_{xy} \geq \beta_{\mathrm{Deloc}(3/4)}\}}.
    \end{equation*}
    From the monotonicity of the variance of the height in the conductances (Proposition~\ref{prop:monotonicity}), we have the inequality, for any $J = (J_{xy})_{ \{ x , y\} \in E \left( \Lambda_L \right)}$,
    \begin{equation} \label{eq:stochdompercoprocess}
        \mathrm{Var}_{L , \beta_V J}^{\Z-\mathrm{GFF}} \left[ \varphi(0) \right]  \geq \mathrm{Var}_{L ,\beta_{\mathrm{Deloc}}(3/4) ,\omega^J}^{\Z-\mathrm{GFF}} \left[ \varphi(0) \right].
    \end{equation}
    If we now assume that the conductances $J =  (J_{xy})_{ \{ x , y\} \in E \left( \Lambda_L \right)}$ are i.i.d. of law $\mu_V$, then $\omega^J$ stochastically dominates the i.i.d. Bernoulli edge percolation of probability $3/4$. Combining this result with Proposition~\ref{prop:monotonicity} and the inequality~\eqref{eq:stochdompercoprocess}, we obtain the inequality
    \begin{equation*}
        \mathbf{E} \left[  \mathrm{Var}_{L ,\beta J}^{\Z-\mathrm{GFF}} \left[ \varphi(0) \right]  \right] \geq \mathbf{E} \left[  \mathrm{Var}_{L ,\beta_{\mathrm{Deloc}}(3/4) ,\omega^J}^{\Z-\mathrm{GFF}} \left[ \varphi(0) \right]  \right] \geq \mathbb{E}_{3/4} \left[  \mathrm{Var}_{L ,\beta_{\mathrm{Deloc}}(3/4) ,\omega}^{\Z-\mathrm{GFF}} \left[ \varphi(0) \right]  \right].
    \end{equation*}
    We finally use the inequality~\eqref{eq:04041220} and Theorem~\ref{Thm:delocheightsupercrit} to deduce that there exists a constant $c_0 > 0$ such that
    \begin{equation*}
        \mathrm{Var}_{L , \beta }^{\Z-V} \left[ \varphi(0) \right] \geq \mathbf{E} \left[  \mathrm{Var}_{L ,\beta J}^{\Z-\mathrm{GFF}} \left[ \varphi(0) \right]  \right] \geq \mathbb{E}_{3/4} \left[  \mathrm{Var}_{L ,\beta_{\mathrm{Deloc}}(3/4) ,\omega}^{\Z-\mathrm{GFF}} \left[ \varphi(0) \right]  \right] \geq c_0 \ln L.
    \end{equation*}
    The proof of Theorem~\ref{Thm:delocheightanealedgauss} is complete.
\end{proof}

Using the same strategy, we will show Theorem~\ref{Thm:delocheightsupercrit} for the XY height function.

\begin{proof}[Proof of Theorem~\ref{Thm:delocheightsupercrit} for the XY height function]
We first use that, by~\cite[Lemma 9.6]{AHPS}, the Bessel function is an annealed Gaussian interaction and, for any $\beta \geq 0$, there exists a probability measure $\kappa_\beta$ on $(0 , \infty)$ such that, for any $k \in \Z$,
\begin{equation*}
    \frac{I_k(\beta)}{I_0(\beta)} = \int_0^\infty e^{- \frac{k^2}{2 J}} \kappa_\beta(d J).
\end{equation*}
We next note that the measures $(\kappa_\beta)_{\beta >0}$ satisfy the following property (as otherwise there would be a contradiction with~\eqref{eq:expansionBessel}): for any $K \in (0 , \infty),$
\begin{equation} \label{eq:convergencekappabeta}
     \kappa_\beta \left( \left[ K , \infty \right) \right) \underset{\beta \to \infty}{\longrightarrow} 1. 
\end{equation}
Using the same computation as in the proof of Theorem~\ref{Thm:delocheightanealedgauss} above (and using the FKG inequality), we obtain the following result. For any $L \in \N$, any $\beta \geq 0$ and any percolation configuration $\omega \in \{ 0 , 1 \}^{E(\Lambda_L)}$, if we let $(J_{xy})_{xy \in E \left( \Lambda_L \right)}$ be a collection of i.i.d. conductances distributed according to the law $\kappa_\beta$ (and denoting by $\mathbf{E}_{\kappa_\beta}$ the corresponding expectation) then we have the inequality
\begin{equation} \label{eq:eqXYforallomega}
    \mathrm{Var}_{L , \beta , \omega}^{\Z-\mathrm{XY}} \left[ \varphi(0) \right] \geq \mathbf{E}_{\kappa_\beta} \left[ \mathrm{Var}_{L , J \omega}^{\Z-\mathrm{GFF}} \left[ \varphi(0) \right]  \right].
\end{equation}
Let us now fix a (supercritical) probability $p > 1/2$, set $\tilde p = (p + 1)/2$ and let $\beta_{\mathrm{Deloc}}(\tilde p)$ be the inverse temperature provided by Theorem~\ref{Thm:delocheightsupercrit} in the case of the integer-valued Gaussian free field on a percolation of probability $\tilde p > 1/2$. Assuming that $\omega \in \{  0,1 \}^{E \left( \Lambda_L \right)}$ is random and sampled according to an i.i.d. edge percolation of probability $p$, we deduce from~\eqref{eq:eqXYforallomega} (by taking the expectation on both sides of the inequality) that
\begin{equation*}
    \E_p \left[ \mathrm{Var}_{L , \beta , \omega}^{\Z-\mathrm{XY}} \left[ \varphi(0) \right] \right] \geq \E_p \left[ \mathbf{E}_{\kappa_\beta} \left[ \mathrm{Var}_{L , J \omega}^{\Z-\mathrm{GFF}} \left[ \varphi(0) \right] \right] \right].
\end{equation*}
The random variables $J = (J_{xy})$ and $\omega = (\omega_{xy})$ on the right-hand side are all independent with $(J_{xy})$ distributed according to the law $\kappa_\beta$ and $(\omega_{xy})$ following the Bernoulli distribution of probability $p$. Combining this observation with the convergence~\eqref{eq:convergencekappabeta}, we obtain the following result: there exists an inverse temperature $\beta_0 := \beta_0(p) \geq 0$ such that if we set, for $\{ x , y \} \in E \left( \Lambda_L \right)$,
\begin{equation*}
    \omega^J_{xy} := \mathbf{1}_{\left\{ J_{xy} \omega_{xy} \geq \beta_{\mathrm{Deloc}}(\tilde p) \right\}},
\end{equation*}
then, for any $\beta \geq \beta_0$, the collection of random variables $\omega^J := \left( \omega^J_{xy} \right)$ stochastically dominates an i.i.d. edge percolation of probability $\tilde p$. Combining this result with Theorem~\ref{Thm:delocheightsupercrit} (in the case of the integer-valued Gaussian free field) and Theorem~\ref{prop:monotonicity}, we deduce that
\begin{align*}
     \E_p \left[ \mathrm{Var}_{L , \beta , \omega}^{\Z-\mathrm{XY}} \left[ \varphi(0) \right] \right] \geq \E_p \left[ \mathbf{E}_{\kappa_\beta} \left[ \mathrm{Var}_{L , J \omega}^{\Z-\mathrm{GFF}} \left[ \varphi(0) \right] \right] \right] & \geq \E_p \left[ \mathbf{E}_{\kappa_\beta} \left[ \mathrm{Var}_{L , \beta_{\mathrm{Deloc}}(\tilde p) , \omega^J}^{\Z-\mathrm{GFF}} \left[ \varphi(0) \right] \right] \right] \\
     & \geq \E_{\tilde p} \left[  \mathrm{Var}_{L , \beta_{\mathrm{Deloc}}(\tilde p), \omega}^{\Z-\mathrm{GFF}} \left[ \varphi(0) \right] \right] \\
     & \geq c_0 \ln L.
\end{align*}
This completes the proof of Theorem~\ref{Thm:delocheightsupercrit} in the case of the XY height function.
\end{proof}

%\begin{remark}
%This proof can be generalised to show that any integer-valued height function with an annealed Gaussian potential on a supercritical percolation cluster is delocalised at sufficiently high temperature. \textcolor{blue}{Rewrite this}
%\end{remark}

\subsection{Delocalisation for integer-valued height functions with annealed Gaussian potential on two-dimensional shift-invariant graphs} \label{section4.8}

This section is devoted to the proof of Theorem~\ref{Thm:delocheightanealedgaussgeneralgraph} on the delocalisation of integer-valued height functions with annealed Gaussian potential on two-dimensional shit-invariant graphs. It is decomposed into three subsections and is structured as follows:
\begin{itemize}
    \item In Section~\ref{subsec521}, we collect two additional properties of the integer-valued Gaussian free field: the monotonicities under addition and identification of vertices.
    \item In Section~\ref{sec:sec531delocZ2longrange}, we prove Theorem~\ref{Thm:delocheightanealedgaussgeneralgraph} in the case where the underlying graph is $\Z^2$ with long-range edges. The proof is similar to the one of Theorem~\ref{Thm:delocheightanealedgauss} but an additional step (which makes use of the monotonicities of Section~\ref{subsec521} and Theorem~\ref{th.Liggett}) is required to tackle the long-range edges.
    \item In Section~\ref{subsec533}, we show Theorem~\ref{Thm:delocheightanealedgaussgeneralgraph} in the general case in the case of two-dimensional shift-invariant graphs.
\end{itemize}

\subsubsection{Preliminaries: additional results on the integer-valued Gaussian free field} \label{subsec521}

We collect in this section two additional monotonicity properties of the integer-valued Gaussian free field: the monotonicity under identification of vertices (Proposition~\ref{corollary.identification}) and the monotonicity under addition of vertices (Proposition~\ref{prop.prop2.9}). We will need to state the results in the case where the integer-valued Gaussian free field is defined on a (general) finite connected graph equipped with conductances and we thus introduce the following definition.

\begin{figure}
\centering
\includegraphics[scale=0.65]{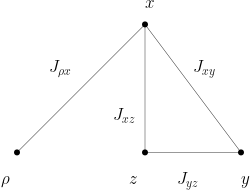}
\caption{An example of rooted graph equipped with conductances} \label{fig:rootedgraph}
\end{figure}

\begin{definition}[Integer-valued Gaussian free field on a general graph]
	To each finite connected rooted graph $G := (W , E , \rho)$ (with vertex set $W$, edge set $E$, root $\rho \in W$) equipped with conductances $J := (J_{yz})_{\{y,z\} \in E} \in (0 , \infty)^E$, we associate the space of functions $\Omega_W := \left\{ \varphi : W \to \Z \, : \, \varphi(\rho) = 0 \right\}$ and equip it with the probability distribution
	\begin{equation*}
	\mathbb{P}_{G , J}^{\Z\mathrm{-GFF}}( \{ \varphi \}) := \frac{1}{Z_{G , J}} \exp \left( - \sum_{yz \in E}  \frac{\left( \varphi(y) - \varphi(z) \right)^2}{J_{yz}} \right).
	\end{equation*}
	We denote by $\mathrm{Var}_{G , J}^{\Z\mathrm{-GFF}}$ the variance with respect to the measure $\mathbb{P}^{\Z\mathrm{-GFF}}_{G , J}$.
\end{definition}

The first result of this section is the monotonicity under identification of vertices.

\begin{proposition}[Monotonicity under identification of vertices] \label{corollary.identification}
Let $y , z \in W$ be two vertices of $G$ and $J \in (0 , \infty)^E$ be a collection of conductances. Denote by $\bar G$ the graph $G$ in which the vertices $y$ and $z$ have been identified and by $\bar J$ the induced collection of conductances (see Figure~\ref{fig:identificationvertices}). Then
\begin{equation*}
    \forall x \in  V,~\mathrm{Var}_{\bar G , \bar J}^{\Z\mathrm{-GFF}}\left[  \varphi(x) \right] \leq \mathrm{Var}_{G , J}^{\Z\mathrm{-GFF}}\left[ \varphi(x) \right].
\end{equation*}
\end{proposition}

\begin{remark}
This property is a direct consequence of Proposition~\ref{prop:monotonicity} by adding (if necessary) a new edge to the graph between the vertices $y$ and $z$ with a conductance equal to infinity, and then by sending the value of this conductance to $0$ (N.B. the result of Proposition~\ref{prop:monotonicity} holds in this general setup, and a similar result would hold for the XY height function).
\end{remark}

\begin{figure}
\centering
\includegraphics[scale=0.65]{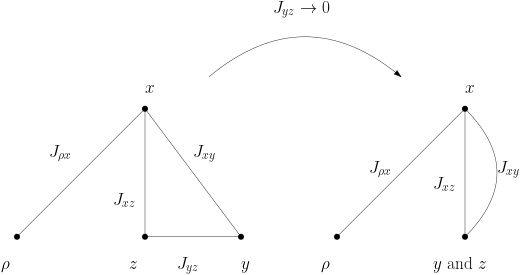}
\caption{Identification of vertices: on this graph, the two vertices $y$ and $z$ are identified. This operation reduces the variances of the heights. On the picture on the right, the two edges carrying the conductances $J_{xz}$ and $J_{xy}$ could be merged together and replaced by one edge with conductance $\left( J_{xz}^{-1} + J_{xy}^{-1} \right)^{-1}$.} \label{fig:identificationvertices}
\end{figure}

\begin{figure}
\centering
\includegraphics[scale=0.65]{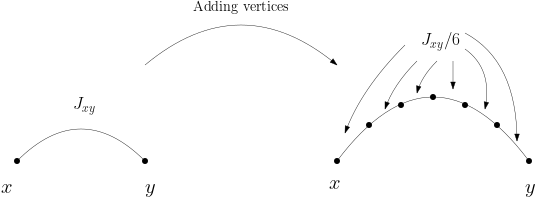}
\caption{Addition of vertices on an edge: this operation reduces the variance of the height of any vertex.} \label{fig:addingvertices}
\end{figure}

The second monotonicity property allows to add vertices to a graph while reducing the variances of the heights upon suitable modification of the conductances (see Figure~\ref{fig:addingvertices}). We refer to the article~\cite[Proof of Theorem 1.2]{AHPS} for a proof of this result.

\begin{proposition}[Monotonicity under addition of vertices] \label{prop.prop2.9}
Let $k \in \N$. Given an edge $\{ y, z\} \in E$, we denote by $\tilde G$ the graph obtained by adding $k$ new vertices on the edge $yz$ (see Figure~\ref{fig:addingvertices}). We denote by $\tilde J$ the collection of conductances on the graph $\tilde G$ defined as follows:
\begin{itemize}
\item On all the edges of $\tilde G$ which are not incident to one of the new vertices, we set $\tilde J = J$,
\item On the edges which are incident to one of the new vertices, we set $\tilde J = J/(k+1)$.
\end{itemize}
Then
\begin{equation*}
    \forall ~ x \in  V,~\mathrm{Var}_{\tilde G , \tilde J}^{\Z\mathrm{-GFF}}\left[ \varphi(x) \right] \leq \mathrm{Var}_{G , J}^{\Z\mathrm{-GFF}}\left[\varphi(x) \right].
\end{equation*}
\end{proposition}

\subsubsection{Delocalisation for integer-valued height functions with annealed Gaussian potential on $\Z^2$ with finite-range interactions} \label{sec:sec531delocZ2longrange}

\begin{figure}
\centering
\includegraphics[scale=0.5]{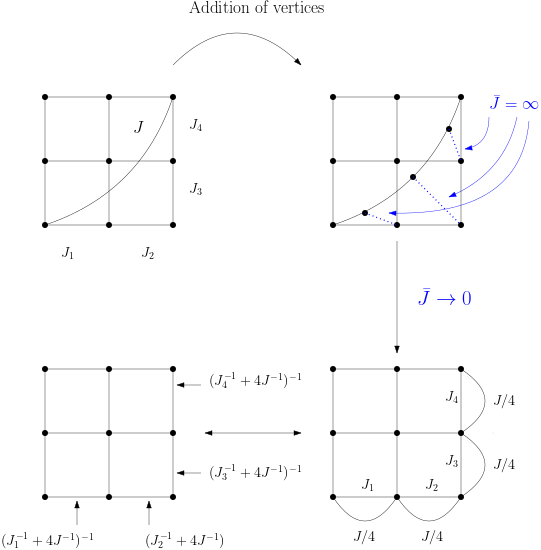}
\caption{The graph surgery used in the proof of Theorem~\ref{Thm:delocheightanealedgaussgeneralgraph}. Vertices are added on the long-range edge and then identified with the vertices of the underlying lattice. This procedure removes the long-range edge and yields an integer-valued Gaussian free field with nearest neighbour interactions.} \label{fig:graphsurgery}
\end{figure}

\begin{proof}[Proof of Theorem~\ref{Thm:delocheightanealedgaussgeneralgraph} for $\Z^2$ with finite-range interaction]
We consider a height function with potential $V$ at inverse temperature $\beta$ when the underlying graph is $\Z^2$ with long-range (but finite-range) edges. We will denote by $E^{\mathrm{LR}} (\Lambda_L)$ the collection of all the edges with at least one endpoint in $\Lambda_L$ and will make use of the pieces of notation
\begin{equation*}
    \mathrm{Var}_{L , \beta }^{\Z-V,  \mathrm{LR}} \left[ \varphi(0) \right] ~~\mbox{and}~~ \mathrm{Var}_{L ,\beta J}^{\Z-\mathrm{GFF},\mathrm{LR}} \left[ \varphi(0) \right]  
\end{equation*}
to refer to the variance of $\varphi(0)$ when the height function is defined in the box $\Lambda_L$ with boundary conditions $0$ outside the box $\Lambda_L$ in the first case for the height function with potential $V$ and inverse temperature $\beta$, and in the second case for the integer-valued Gaussian free field with conductances $(\beta J_{xy})_{\{x , y \} \in E^{\mathrm{LR}} (\Lambda_L)}$.

The first part of the proof is identical to the proof of Theorem~\ref{Thm:delocheightanealedgauss}: we first use that the potential $V$ is annealed Gaussian together with the FKG inequality to lower bound the variance of the height $\varphi(0)$ by the variance of the height $\varphi(0)$ of an integer-valued Gaussian free field with i.i.d. conductances. Specifically, we have the inequality
\begin{equation*}
    \mathrm{Var}_{L , \beta }^{\Z-V,  \mathrm{LR}} \left[ \varphi(0) \right] \geq \mathbf{E} \left[  \mathrm{Var}_{L ,\beta J}^{\Z-\mathrm{GFF},\mathrm{LR}} \left[ \varphi(0) \right]  \right],
\end{equation*}
where we denote by $\mathbf{E}$ the expectation with respect to the product measure $\mathbf{P}:= \prod_{xy\in E^{\mathrm{LR}}(\Lambda_L)} \mu_V(d J_{xy})$. We then make use of Proposition~\ref{corollary.identification} and Proposition~\ref{prop.prop2.9} to perform a series of operations which will all reduce the variance of the height $\varphi(0)$ and map the graph $\Lambda_L$ with long-range edges to the graph $\Lambda_L$ with only nearest neighbour-interactions. Specifically, we proceed as follows (and refer to Figure~\ref{fig:graphsurgery} for guidance):
\begin{itemize}
    \item We first consider a long-range edge connecting the vertices $x , y \in \Lambda_L$, use Proposition~\ref{prop.prop2.9} to add $|x-y|_1-1$ vertices on the edge connecting $x$ and $y$ and assign to each new edge the conductance $\beta J_{xy}/|x-y|_1$ (see Figure~\ref{fig:graphsurgery}).
    \item We next consider a path going from $x$ to $y$ of length $|x-y|_1$ in $\Z^2$ (selected according to some deterministic procedure) and identify the vertices added to the long-range edge $\{ x, y\}$ in the previous step to the vertices of this path (see Figure~\ref{fig:graphsurgery}). 
\end{itemize}
We next perform these two operations on all the long-range edges, and note that, after this procedure, we obtain a nearest neighbour integer-valued Gaussian free field with random conductances denoted by~$\tilde J$. Since all these operations reduce the variance of $\varphi(0)$, we deduce that
\begin{equation} \label{eq:04041220bis}
    \mathrm{Var}_{L , \beta }^{\Z-V,  \mathrm{LR}} \left[ \varphi(0) \right] \geq \mathbf{E} \left[  \mathrm{Var}_{L ,\beta J}^{\Z-\mathrm{GFF},\mathrm{LR}} \left[ \varphi(0) \right]  \right] \geq \mathbf{E} \left[  \mathrm{Var}_{L ,\beta  \tilde J}^{\Z-\mathrm{GFF}} \left[ \varphi(0) \right]  \right].
\end{equation}
The conductances $\tilde J = (\tilde J_{xy})_{xy \in E(\Lambda_L)}$ are (deterministic) functions of the i.i.d. conductances $J = (J_{xy})_{xy \in E^{\mathrm{LR}}(\Lambda_L)}$, they are not independent nor identically distributed but they satisfy the following properties: they have a finite-range of dependence, and there is a finite number of possibilities for the laws of these conductances. In both cases, the range of dependency and the number of laws depend only on the structure of the underlying graph (i.e., the periodic structure of the graph and the length of the longest edges), but they do not depend on the sidelength $L$. We denote the range of dependency by $k$, i.e., we assume that the conductances are $k$-dependent.

We next let $p_0 \in (0,1)$ and $\beta_{0} \geq 0$ be the inverse temperature provided by Proposition~\ref{prop:prop4.19}. We then let $\bar p\in (0,1)$ be the probability provided by Theorem~\ref{th.Liggett} for $k$-dependent percolation with probability $p = p_0$. From the observations of the previous paragraph, we may deduce that:
\begin{itemize}
    \item There exists an inverse temperature $\beta_V \geq 0$ (depending only on the potential $V$, the probability $\bar p$ and the structure of the underlying graph) such that
\begin{equation} \label{eq:betaVpineq}
    \forall  x \in \Lambda_L, ~  \mathbf{P} \left( \forall y \sim x , \, \beta_V \tilde J_{xy} \geq \beta_{0} \right) \geq \bar p
\end{equation}
    (N.B. This is because there is a finite number of laws for the conductances~$(\tilde J_{xy})_{\{x , y\} \in E(\Lambda_L)}$).
    \item If we define the site percolation configuration $r^{\tilde J} \in \{ 0,1\}^{\Lambda_L}$ according to the identity, for any $x \in \Lambda_L$,
    \begin{equation*}
        r_x^{\tilde J} := \indc_{\{ \forall y \sim x, \, \beta_V \tilde J_{xy} \geq \beta_{0}\}}.
    \end{equation*}
    Then this site percolation is $k$-dependent and the probability for a site to be open is larger than~$\bar p$. We may thus apply Theorem~\ref{th.Liggett} to deduce that this site percolation stochastically dominates an i.i.d. Bernoulli site percolation of parameter $p_0$.
\end{itemize}
Combining these two results with Proposition~\ref{prop:monotonicity}, we deduce that
\begin{equation*}
    \mathbf{E} \left[  \mathrm{Var}_{L ,\beta_V   \tilde J}^{\Z-\mathrm{GFF}} \left[ \varphi(0) \right]  \right] \geq \mathbf{E} \left[  \mathrm{Var}_{L ,\beta_{0},  r^{\tilde J}}^{\Z-\mathrm{GFF}} \left[ \varphi(0) \right] \right] \geq \mathbb{E}_{p_0}^{\mathrm{site}} \left[  \mathrm{Var}_{L ,\beta_{0}, r}^{\Z-\mathrm{GFF}} \left[ \varphi(0) \right] \right].
\end{equation*}
We finally use the inequality~\eqref{eq:04041220bis} and Proposition~\ref{prop:prop4.19} to deduce that there exists a constant $c_0 > 0$ such that, for any $\beta \geq \beta_V$,
    \begin{equation*}
        \mathrm{Var}_{L , \beta }^{\Z-V, \mathrm{LR}} \left[ \varphi(0) \right] \geq \mathbf{E} \left[ \mathrm{Var}_{L ,\beta  \tilde J}^{\Z-\mathrm{GFF}} \left[ \varphi(0) \right] \right] \geq \mathbf{E} \left[ \mathrm{Var}_{L ,\beta_V  \tilde J}^{\Z-\mathrm{GFF}} \left[ \varphi(0) \right]   \right] \geq \mathbb{E}_{p_0}^{\mathrm{site}} \left[  \mathrm{Var}_{L ,\beta_{0} ,r} \left[ \varphi(0) \right]  \right] \geq c_0 \ln L.
    \end{equation*}
    This completes the proof of Theorem~\ref{Thm:delocheightanealedgaussgeneralgraph} when the underlying graph is $\Z^2$ with finite-range edges.
\end{proof}

\subsubsection{Delocalisation for integer-valued height functions on two-dimensional shift-invariant graphs} \label{subsec533}

\begin{figure}
\centering
\includegraphics[scale=0.5]{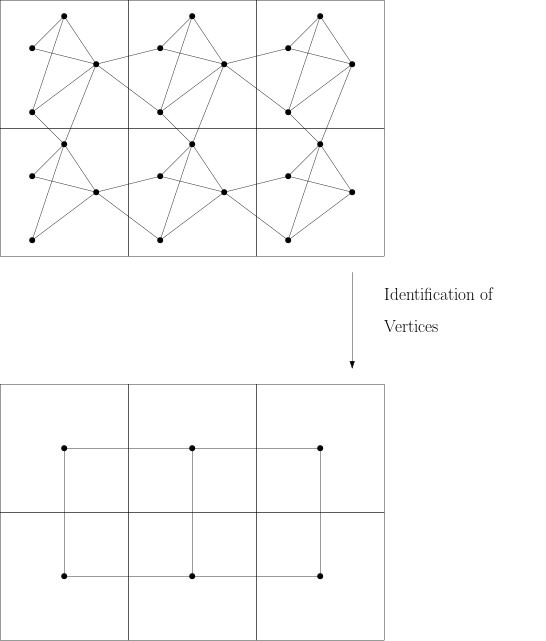}
\caption{The graph surgery used in the proof of Theorem~\ref{Thm:delocheightanealedgauss}. All the vertices in the same square are identified together to reduce the problem to the case of the graph $\Z^2$ with finite-range interactions.} \label{fig:graphsurgery2}
\end{figure}

\begin{proof}[Proof of Theorem~\ref{Thm:delocheightanealedgaussgeneralgraph}]
We start with the same computation as in the beginning of the proof of Theorem~\ref{Thm:delocheightanealedgauss}: we use that the potential $V$ is annealed Gaussian together with the FKG inequality to derive the inequality
\begin{equation*}
    \mathrm{Var}_{L , G , \beta }^{\Z-V} \left[ \varphi(0) \right] \geq \mathbf{E} \left[  \mathrm{Var}_{L , G ,\beta J}^{\Z-\mathrm{GFF}} \left[ \varphi(0) \right]  \right],
\end{equation*}
where $\mathbf{E}$ denotes the expectation under the product measure $\prod_{xy\in E} \mu_V(d J_{xy})$. We then let $R_G \in \N$ be an integer chosen sufficiently large (depending on the graph $G$) such that,
\begin{equation*}
    \forall k \in  R_G \Z^2, ~~ W \cap \left( k + \left[- \frac{R_G}{2} , \frac{R_G}{2} \right]^2 \right) \neq \emptyset.
\end{equation*}
The existence of such an integer is guaranteed by the assumptions on $G$ (in particular, that it is invariant under the action of a lattice). We then perform the following operations: for each $k \in R_G \Z^2$, we identify all the vertices of the set $\W \cap \left( k + \left[- \frac{R_G}{2} , \frac{R_G}{2} \right]^2 \right)$ together. This operation reduces the variances of the height, and the resulting model is an integer-valued Gaussian free field on the lattice $\Z^2$ with some long-range edges and random conductances (see Figure~\ref{fig:graphsurgery2}). The edges have a finite-range, the conductances are sampled according to a finite number of distributions and have a finite-range of dependence, we may thus conclude the proof using the same arguments as in Section~\ref{sec:sec531delocZ2longrange}.
\end{proof}

\bibliographystyle{alpha}
\bibliography{bibliography}

\end{document}